\documentclass[aop,preprint]{imsart}

\setattribute{journal}{name}{}

\usepackage[T1]{fontenc}	
\usepackage{graphicx}
\usepackage{amsmath, amsthm, amssymb}
\usepackage{upgreek}
\usepackage{mathrsfs}
\usepackage{pdfsync}
\usepackage[a4paper,top=3cm,bottom=3cm,left=2.5cm,right=2.5cm, bindingoffset=0mm]{geometry}

\usepackage[usenames,dvipsnames]{xcolor}
\usepackage[inline,shortlabels]{enumitem}
\usepackage[hyphens]{url}									
\usepackage{verbatim}								
\usepackage{dsfont}									

\usepackage{mathtools}
\usepackage{xparse}

\usepackage[pdftex,bookmarks,
						hidelinks,					
						breaklinks,unicode,
						citecolor=OliveGreen,
						linkcolor=Maroon
						]{hyperref} 

\usepackage{stmaryrd}

\newcommand{\boldGamma}{{\boldsymbol\Gamma}}

\newcommand{\boldTau}{\mbfT}

\newcommand{\Beta}{\mathrm{B}}

\newcommand{\boldalpha}{{\boldsymbol \alpha}}

\newcommand{\boldgamma}{{\boldsymbol \gamma}}
\newcommand{\bolddelta}{{\boldsymbol \delta}}

\ExplSyntaxOn
\NewDocumentCommand{\makeabbrev}{mmm}
 {
  \yoruk_makeabbrev:nnn { #1 } { #2 } { #3 }
 }

\cs_new_protected:Npn \yoruk_makeabbrev:nnn #1 #2 #3
 {
  \clist_map_inline:nn { #3 }
   {
    \cs_new_protected:cpn { #2 } { #1 { ##1 } }
   }
 }
 \ExplSyntaxOff

\makeabbrev{\textbf}{tbf#1}{a,b,c,d,e,f,g,h,i,j,k,l,m,n,o,p,q,r,s,t,u,v,w,x,y,z,A,B,C,D,E,F,G,H,I,J,K,L,M,N,O,P,Q,R,S,T,U,V,W,X,Y,Z}

\makeabbrev{\textbf}{bf#1}{a,b,c,d,e,f,g,h,i,j,k,l,m,n,o,p,q,r,s,t,u,v,w,x,y,z,A,B,C,D,E,F,G,H,I,J,K,L,M,N,O,P,Q,R,S,T,U,V,W,X,Y,Z}

\makeabbrev{\textsf}{tsf#1}{a,b,c,d,e,f,g,h,i,j,k,l,m,n,o,p,q,r,s,t,u,v,w,x,y,z,A,B,C,D,E,F,G,H,I,J,K,L,M,N,O,P,Q,R,S,T,U,V,W,X,Y,Z}

\makeabbrev{\mathsf}{mss#1}{a,b,c,d,e,f,g,h,i,j,k,l,m,n,o,p,q,r,s,t,u,v,w,x,y,z,A,B,C,D,E,F,G,H,I,J,K,L,M,N,O,P,Q,R,S,T,U,V,W,X,Y,Z}

\makeabbrev{\mathfrak}{mf#1}{a,b,c,d,e,f,g,h,i,j,k,l,m,n,o,p,q,r,s,t,u,v,w,x,y,z,A,B,C,D,E,F,G,H,I,J,K,L,M,N,O,P,Q,R,S,T,U,V,W,X,Y,Z}

\makeabbrev{\mathrm}{mrm#1}{a,b,c,d,e,f,g,h,i,j,k,l,m,n,o,p,q,r,s,t,u,v,w,x,y,z,A,B,C,D,E,F,G,H,I,J,K,L,M,N,O,P,Q,R,S,T,U,V,W,X,Y,Z}

\makeabbrev{\mathbf}{mbf#1}{a,b,c,d,e,f,g,h,i,j,k,l,m,n,o,p,q,r,s,t,u,v,w,x,y,z,A,B,C,D,E,F,G,H,I,J,K,L,M,N,O,P,Q,R,S,T,U,V,W,X,Y,Z}

\makeabbrev{\mathcal}{mc#1}{A,B,C,D,E,F,G,H,I,J,K,L,M,N,O,P,Q,R,S,T,U,V,W,X,Y,Z}

\makeabbrev{\mathbb}{mbb#1}{A,B,C,D,E,F,G,H,I,J,K,L,M,N,O,P,Q,R,S,T,U,V,W,X,Y,Z}

\makeabbrev{\mathscr}{ms#1}{A,B,C,D,E,F,G,H,I,J,K,L,M,N,O,P,Q,R,S,T,U,V,W,X,Y,Z}

\makeabbrev{\mathrm}{#1}{
Id,id,ran,rk,diag,stab,ann,conv,pr,ev,tr,End,Hom,sgn,im,op,can,fin,ext,red,tot,
rot,usc,lsc,Lip,lip,bLip,osc,AC,loc,spec,
supp,Opt,Adm,Cpl,Geo,GeoOpt,GeoAdm,GeoCpl,reg,
bd,co,Ric,Exp,dExp,dist,seg,Seg,cut,fcut,Cut,SDiff,Iso,Isom,diam,cl,Homeo,Diff,Der,vol,dvol,inj,relint,
var,law,Var,Poi,Gam,pa,so,iso,fs,inv,pqi,mix,
}

\makeabbrev{\mathsf}{#1}{CD,BE,RCD,MCP,Ent,wMTW,MTW}

\makeabbrev{\mathsc}{#1}{mmaf,cg}

\newcommand{\eps}{\varepsilon}

\renewcommand{\div}{\mathrm{div}}
\newcommand{\Div}{\mathbf{div}}
\newcommand{\defeq}{\eqqcolon}
\renewcommand{\complement}{\mathrm{c}}
\newcommand{\bid}{{\star}}

\newcommand{\mathsc}[1]{\text{\textsc{#1}}}

\newcommand{\emparg}{{\,\cdot\,}}

\DeclareMathOperator*{\argmin}{argmin}

\DeclareMathOperator{\Diffeo}{Diff}
\DeclareMathOperator{\Flow}{Flow}
\newcommand{\fl}{\uppsi}
\newcommand{\Fl}{\Psi}

\newcommand{\grad}{\boldsymbol\nabla}

\DeclareMathOperator{\re}{re}

\newcommand{\forallae}[1]{{\textrm{\,for ${#1}$-a.e.\,}}}

\newcommand{\Leb}{{\mathcal L}}

\newcommand{\tra}{{{}_{\!\top}}}

\newcommand{\dom}[1]{\msD(#1)}

\newcommand{\slo}[1]{\abs{\mathrm D#1}\!}

\DeclareMathOperator{\eqdef}{\coloneqq}

\let\epsilon\varepsilon
\let\temp\phi
\let\phi\varphi
\let\varphi\temp

\newcommand{\rar}{\rightarrow}

\newcommand{\Rar}{\,\Longrightarrow\,}

\newcommand{\nlim}{\lim_{n}}								

\newcommand{\nliminf}{\liminf_{n }}

\newcommand{\nlimsup}{\limsup_{n }\,}

\newcommand{\diff}{\mathop{}\!\mathrm{d}}						
\newcommand{\del}{\partial}

\newcommand{\tabs}[1]{\lvert#1\rvert}	
\newcommand{\abs}[1]{\left\lvert#1\right\rvert}						
\newcommand{\norm}[1]{\left\lVert#1\right\rVert}					
\newcommand{\set}[1]{\left\{#1\right\}}	
\newcommand{\tset}[1]{\{#1\}}							
\newcommand{\tonde}[1]{\left(#1\right)}
\newcommand{\ttonde}[1]{{\big(#1\big)}}						
\newcommand{\quadre}[1]{\left[#1\right]}							
\newcommand{\tscalar}[2]{\big\langle #1 \, |\, #2\big\rangle}		
\newcommand{\scalar}[2]{\left\langle #1 \,\middle |\, #2\right\rangle}		

\newcommand{\seq}[1]{\tonde{#1}}								
\newcommand{\tseq}[1]{{(#1)}}
\newcommand{\Cb}{{\mcC_{b}}}							

\newcommand{\Mbp}{\mathscr M_b^+}

\newcommand{\pfwd}{\sharp}

\DeclareMathOperator*{\essinf}{essinf}

\DeclareMathOperator{\car}{\mathds 1}

\DeclareMathOperator{\emp}{\varnothing} 

\DeclareMathOperator{\N}{{\mathbb N}}
\DeclareMathOperator{\R}{{\mathbb R}}

\DeclareMathOperator{\Q}{{\mathbb Q}}
\DeclareMathOperator{\Z}{{\mathbb Z}}

\newcommand{\cont}{{\mathrm{cont}}} 

\newcommand{\restr}{\big\lvert}

\newcommand{\iref}[1]{\ref{#1}}

\newcommand{\comm}{\,\textrm{,}\;\,}
\newcommand{\semicolon}{\,\textrm{;}\;\,}
\newcommand{\fstop}{\,\textrm{.}}

\DeclareMathOperator{\zero}{{\mathbf 0}}

\DeclareMathOperator{\DiffSP}{Diff^\infty_+}

\DeclareMathOperator{\bexp}{\mathbf{exp}}

\DeclareMathOperator{\Graph}{Graph}

\newcommand{\prid}{\mathrel{\ooalign{$\lneq$\cr\raise.22ex\hbox{$\lhd$}\cr}}}

\newcommand{\n}[1]{{\overline{#1}}}

\newcommand{\ptws}{\mathrm{ptws\,}}

\newcommand{\gscal}[2]{\scalar{#1}{#2}_\mssg}

\newcommand{\Test}{\mcF\mcC}
\newcommand{\TestV}{\mcX\mcC}

\newcommand{\DF}{{\mcD}}

\newcommand{\PP}{{\mcP}}
\newcommand{\RP}{{\mcR}}
\newcommand{\MS}{{\mcM}}
\newcommand{\MSI}{{\mcS}}
\newcommand{\BB}{{\mcW_0}}

\newcommand{\empargop}{{-}}

\newcommand{\Vect}{\mfX}
\newcommand{\VectCount}{\msX}
\newcommand{\MFD}{M}

\allowdisplaybreaks

\numberwithin{equation}{section}
\theoremstyle{plain}
\newtheorem{thm}{Theorem}[section]
\newtheorem*{thm*}{Theorem}

\newtheorem{prop}[thm]{Proposition}
\newtheorem{lem}[thm]{Lemma}
\newtheorem{cor}[thm]{Corollary}

\theoremstyle{definition}
\newtheorem{defs}[thm]{Definition}
\newtheorem{ass}[thm]{Assumption}

\theoremstyle{remark}
\newtheorem{rem}[thm]{Remark}
\newtheorem*{rem*}{Remark}
\newtheorem{ese}[thm]{Example}

\begin{document}
\begin{frontmatter}

\title{A Rademacher-type theorem\\on \texorpdfstring{$L^2$}{L2}-Wasserstein spaces over\\closed Riemannian manifolds\thanksref{T2}}

\begin{aug}
\author{\fnms{Lorenzo} \snm{Dello Schiavo}\thanksref{t1}}
\ead[label=e1]{delloschiavo@iam.uni-bonn.de}

\thankstext{T2}{Research supported by the CRC 1060 and the Hausdorff Center for Mathematics (University of Bonn).}

\thankstext{t1}{The author is grateful to Prof.s K.-T.~Sturm, E.~W.~Lytvynov and N.~Gigli for useful remarks and comments, and to Prof.~A.~Eberle for providing the reference~\cite{Ebe96}. He is also grateful to an anonymous referee for several remarks and comments that significantly improved the presentation of this work.
Since part of this research was carried out during the \emph{Intense Activity Period on Metric Measure Spaces and Ricci Curvature} (September 4--29, 2017), it is a pleasure to thank the Max Planck Institute for Mathematics in Bonn for the hospitality.}

\runauthor{L.~Dello~Schiavo}

\affiliation{Institut f\"ur Angewandte Mathematik}

\address{Institut f\"ur Angewandte Mathematik\\
Rheinische Friedrich-Wilhelms-Universit\"at Bonn\\
Endenicher Allee 60\\
DE 53115 Bonn\\
Germany\\
\printead{e1}
}

\end{aug}

\begin{abstract}
Let~$\mbbP$ be any Borel probability measure on the $L^2$-Wasserstein space $(\msP_2(\MFD),W_2)$ over a closed Riemannian manifold~$\MFD$. We consider the Dirichlet form~$\mcE$ induced by~$\mbbP$ and by the Wasserstein gradient on~$\msP_2(\MFD)$. Under natural assumptions on~$\mbbP$, we show that $W_2$-Lipschitz functions on~$\msP_2(\MFD)$ are contained in the Dirichlet space $\dom{\mcE}$ and that~$W_2$ is dominated by the intrinsic metric induced by~$\mcE$.
We illustrate our results by giving several detailed examples.
\end{abstract}



\end{frontmatter}

\tableofcontents

\section{Introduction}
We consider the $L^2$-Wasserstein space~$\msP_2=\tonde{\msP_2(\MFD),W_2}$ associated to a closed Riemannian manifold~$(\MFD,\mssg)$. Since the seminal work of F.~Otto~\cite{Ott01}, the geometry of~$\msP_2$ has been widely studied from several view points. Definitions have been proposed and thoroughly studied of a ``weak Riemannian structure'' on~$\msP_2$, Lott~\cite{Lot07}, of a gradient for ``smooth'' functions on~$\msP_2$, of tangent space(s) to~$\msP_2$ at a point, Gigli~\cite{Gig11}, of an exponential map~\cite{Gig11}, of a Levi-Civita connection~\cite{Gig12}, of differential forms~\cite{GanKimPac10}.
This heuristic picture of~$\msP_2$ as an infinite-dimensional Riemannian manifold calls for the existence of a measure on~$\msP_2$ canonically and uniquely associated to the metric structure.
As it is the case for a differentiable manifold, such a measure ---~if any~--- would deserve the name of \emph{Riemannian volume measure} which we shall adopt in the following.

In this framework, the question of the existence of such a Riemannian volume measure on~$\msP_2$ has been insistently posed, e.g.,~\cite{Gig11,Stu11,vReStu09,Cho12}.
In the case of~$\MFD=\mbbS^1$, M.-K.~von~Renesse and \mbox{K.-T.}~Sturm~\cite{vReStu09}  proposed as a candidate the \emph{entropic measure} on~$\msP_2(\mbbS^1)$, see Example~\ref{ese:EntMeas}.
Whereas a suitable definition of entropic measure on~$\msP_2(\MFD)$ for a closed Riemannian manifold~$\MFD$ was given by K.-T.~Sturm in~\cite{Stu11}, most of its properties in this general case remain unknown.
Here, we rather address the question of discerning the properties of a \emph{volume measure}~$\mbbP$ on~$\msP_2$. By ``volume measure'' we shall mean any analogue on~$\msP_2$ of \emph{a} measure on a differentiable manifold induced by a volume form via integration.

We do so by proving a Rademacher-type result on the $\mbbP$-a.e.~\emph{Fr\'echet} differentiability of $W_2$-Lipschitz functions (Thm.~\ref{t:main}).
Namely, we consider a Dirichlet space~$\msF$ associated to~$\mbbP$ and to a natural gradient, with core being the algebra~$\Test^\infty$ of cylinder functions induced by smooth potential energies (Dfn.~\ref{d:CylFunc}). 
Combining the strategy of~\cite{RoeSch99} with the fine analysis of tangent plans performed by N.~Gigli in~\cite{Gig11}, we study, for functions in~$\msF$, suitable concepts of \emph{directional derivative} and \emph{differential}, proving their consistency on~$\Test^\infty$. 
We show that, if~$\mbbP$ is quasi-invariant with respect to the family of shifts defining the gradient, then the space of $W_2$-Lipschitz functions is contained in~$\msF$.
 
The requirement of the Rademacher property is indeed a natural one for a volume measure. For instance, it was recently shown by G.~De~Philippis and F.~Rindler~\cite[Thm.~1.14]{DePRin16} that, if~$\mu$ is a positive Radon measure on~$\R^d$ such that every Lipschitz function is $\mu$-a.e.~differentiable, then~$\mu$ is absolutely continuous with respect to the Lebesgue measure~$\Leb^d$.
In infinite dimensions, the problem has been addressed in linear spaces~\cite{BogMay96}, in particular on the abstract Wiener space~\cite{EncStr93}, and ---~in the non-flat case~--- on configuration spaces~\cite{RoeSch99}. 

\paragraph{Outline of the paper}
Section~\ref{s:Radem} presents the main Theorem~\ref{t:main}, together with the assumptions and several remarks.
Further preliminaries and auxiliary results on Dirichlet forms, optimal transport, and the metric geometry of~$\msP_2$ are collected in Section~\ref{s:Prelim}.
The proof of the main theorem is presented in Section~\ref{s:Proof}, together with some heuristics and preparatory lemmas.
In Section~\ref{s:Examples} we detail some examples of measures satisfying, fully or in part, our assumptions. These are mainly taken from the theory of point processes and include \emph{normalized mixed Poisson measures}, the \emph{Dirichlet--Ferguson measure}~\cite{Fer73}, as well as the \emph{entropic measure}~\cite{vReStu09} and an image on~$\msP_2(\mbbS^1)$ of the \emph{Malliavin--Shavgulidze measure}~\cite{MalMal90}.
We show through these examples how the situation on~$\msP_2$ is opposite to the aforementioned result in~\cite{DePRin16}. In particular, there exist mutually singular measures with full support on~$\msP_2$ satisfying the Rademacher property.
Auxiliary results are collected in Section~\ref{s:Appendix}, together with a discussion of the notion of ``tangent bundle'' to~$\msP_2$ from the point of view of global derivations of the algebra~$\Test^\infty$.

\section{A Rademacher Theorem on~\texorpdfstring{$\msP_2$}{P2}}\label{s:Radem}
Let~$(\MFD,\mssg)$ be a Riemannian manifold, and~$\msP$ be the space of all Borel probability measures on~$\MFD$.
In order to perform computations for functions on~$\msP$ in the spirit of~\cite{Ott01, Lot07}, we recall the definition of \emph{potential energy}, in the sense of~\cite[\S5.2.2]{Vil03}. Namely, given a continuous bounded function~$f\colon \MFD\rar \R$, the potential energy~$f^\bid\colon \msP\rar\R$ associated to~$f$ is defined as
\begin{align}\label{eq:Test0}
f^\bid\mu\longmapsto \mu f\eqdef \int_\MFD f \diff\mu\comm \qquad f\in\Cb(\MFD)\fstop
\end{align}

We endow~$\msP$ with the \emph{narrow} (or: \emph{weak}) topology, that is, the coarsest topology making all functions of the form~\eqref{eq:Test0} continuous, and with the corresponding Borel $\sigma$-algebra.

\paragraph{A metric structure on~$\msP$} Given~$\mu_1,\mu_2\in\msP$, denote by~$\Cpl(\mu_1,\mu_2)$ the set of \emph{couplings} (or: \emph{transport plans}) between~$\mu_1$ and~$\mu_2$, i.e.~the set of Borel probability measures on~$\MFD^{\times 2}$ such that~$\pr^i_\pfwd\pi=\mu_i$ for~$i=1,2$. Every distance~$\mssd$ on~$\MFD$ induces a natural distance on the subspace of all measures in~$\msP$ with finite $\mssd^2$-moment.
\begin{defs}[$L^2$-Wasserstein space]
For a fixed~$x_0\in \MFD$ set
\begin{align*}
\msP_2\eqdef\set{\mu\in\msP : \int_\MFD\mssd^2(x,x_0)\diff\mu(x)<\infty}
\end{align*}
and
\begin{align}\label{eq:WD}
W_2(\mu_1,\mu_2)\eqdef \tonde{\inf_{\pi\in\Cpl(\mu_1,\mu_2)} \int_{\MFD^{\times 2}} \mssd^2(x,y)  \diff\pi(x,y)}^{1/2} \fstop
\end{align}
The space~$(\msP_2,W_2)$ is a metric space~\cite[Dfn.s~6.1,~6.4]{Vil09}, called~\emph{$L^2$-Wasserstein space} \emph{over~$(\MFD,\mssd)$}.
\end{defs}
\begin{rem} By triangle inequality for~$\mssd$, the space~$\msP_2$ does not depend on the choice of~$x_0$. The set~$\Opt(\mu_1,\mu_2)$ of \emph{optimal plans}~$\pi\in \Cpl(\mu_1,\mu_2)$ attaining the infimum in~\eqref{eq:WD} is always non-empty~\cite[Thm.~4.1]{Vil09}. The $L^2$-Wasserstein distance~$W_2$ is lower semi-continuous w.r.t.~the narrow topology on~$\msP$,~\cite[Rmk.~6.12]{Vil09}.
\end{rem}

\paragraph{A ``differential'' structure on~$\msP$} In the following, we shall always assume that
$(\MFD,\mssg)$ is \emph{closed}, and~$\mssd=\mssd_\mssg$ is the intrinsic distance.
%
For ease of exposition we also take~$\MFD$ to be smooth and connected. All the results below remain valid if~$\MFD$ is of class~$\mcC^3$.

\medskip

We regard the algebra~$\Test^\infty$ of smooth potential energies as the algebra of ``smooth'' functions on~$\msP$.

\begin{defs}[Cylinder functions]\label{d:CylFunc}
Say that a function~$u\colon \msP\rar \R$ is \emph{cylinder} if there exist~$k\geq 0$, and functions~$F\in \mcC^\infty(\R^k)$ and~$f_1,\dotsc, f_k\in \mcC^\infty(\MFD)$, so that
\begin{align}\label{eq:TestFC}
u(\mu)= F(f_1^\bid\mu,\dotsc, f_k^\bid\mu)\comm \qquad \mu\in\msP \fstop
\end{align}

We denote by~$\Test^\infty$ the algebra of all cylinder functions on~$\msP$.
\end{defs}

\begin{rem}\label{r:NonUniqueCylFunc}
The representation of $u\in\Test^\infty$ by some $F$~and $\seq{f_i}_{i\leq k}$ as in~\eqref{eq:TestFC} is never unique.
\end{rem}

\begin{rem}\label{r:ContTest}
Since~$\MFD$ is compact,~$(\msP_2,W_2)$ is a \emph{compact} geodesic metric space~\cite[Thm.~2.10]{AmbGig11}, and it coincides with~$\msP$ as a topological space~\cite[Dfn.~6.8]{Vil09}.
By compactness of~$\msP_2$, in the definition above one might equivalently take~$F\in\mcC^\infty_c(\R^k)$. The given definition makes more apparent that~$f^\bid\in \Test^\infty$ for all~$f\in\mcC^\infty(\MFD)$. Since~$f^\bid$ is continuous on~$\msP_2=\msP$ by definition of narrow topology, all cylinder functions are continuous and thus (Borel) measurable.
\end{rem}

Motivated by the analogous choice in the framework of configuration spaces (cf.~\cite[Eq.~(1.1)]{RoeSch99}, see~\S\ref{ss:NMPM} below), we define the gradient of~$u\in\Test^\infty$ by
\begin{align}\label{eq:grad0}
\grad u(\mu)(x)\eqdef\sum_i^k (\partial_i F)(f_1^\bid\mu,\dotsc, f_k^\bid\mu)\, \nabla f_i(x) \fstop
\end{align}

This choice is consistent, by chain rule, with the Fr\'echet differentiability of~$f^\bid$ with respect to a natural Riemannian structure on the space of absolutely continuous measures~$\mu=\rho\,\dvol_\mssg\in\msP$,~e.g.,~\cite{Lot07} or \cite[\S9.1]{Vil03}, and more generally with the differentiability of functionals on probability measures, e.g.~\cite{AmbGigSav08}; furthermore, it is also consistent with the definition of a Wasserstein gradient in the recent work~\cite{ChoGan17}. See in particular~\cite[Dfn.~2.3 and Rmk.~2.4]{ChoGan17}.

We will also need a concept of directional derivative for functions in~$\Test^\infty$ and thus a concept of \emph{direction} at a point~$\mu$ in~$\msP$. It is not surprising that such a definition ought to be inherited from the differentiable structure of the manifold~$\MFD$, henceforth the \emph{base space}. 
Indeed, let~$T_x\MFD$ be the tangent space to~$\MFD$ at the point~$x$. We denote by~$\Vect^0$ the space of continuous vector fields, that is, sections of the tangent bundle~$T\MFD$, endowed with the supremum norm
\begin{align*}
\norm{w}_{\Vect^0}\eqdef \sup_{x\in \MFD} \abs{w_x}_\mssg \fstop
\end{align*}

We let further~$\Vect^\infty\subset \Vect^0$ be the algebra of smooth vector fields on~$\MFD$.
For any~$w\in\Vect^\infty$ we denote by~$\tseq{\fl^{w,t}}_{t\in\R}$ the flow generated by~$w$, i.e.~a map~$\fl^{w,t}\colon \MFD\rar \MFD$ such that
\begin{align*}
\dot \fl^{w,t}(x)=w(\fl^{w,t}(x)) \quad \textrm{and} \quad \fl^{w,0}(x)=x \comm \qquad x\in \MFD\comm
\end{align*}
where by~$\dot\fl^{w,t}(x)$ we mean the velocity of the curve~$s\mapsto \fl^{w,s}(x)$ at time~$t$.
By compactness of~$\MFD$ every~$w\in\Vect^\infty$ admits a unique flow, well-defined and a smooth orientation-preserving diffeomorphism in~$\DiffSP(\MFD)$ for all times~$t\in \R$, e.g.~\cite[\S{1.3.7(ii)}]{Ban97}.
If we denote by
\begin{align*}
\Fl^{w,t}\eqdef \fl^{w,t}_\pfwd\colon \msP\rar\msP
\end{align*}
the push-forward via~$\fl^{w,t}$, then a straightforward computation (see Lem.~\ref{l:Derivations} below) shows that
\begin{align}\label{eq:DirDer0}
(\grad_w u)(\mu)\eqdef \diff_t\restr_{t=0} (u\circ \Fl^{w,t})(\mu)=\tscalar{\grad u(\mu)}{w}_{\Vect_\mu}\comm \quad u\in\Test^\infty\comm
\end{align}
where
\begin{align}\label{eq:Scalar}
\scalar{w^0}{w^1}_{\Vect_\mu}\eqdef \int_\MFD \gscal{w^0_x}{w^1_x} \diff\mu(x) \comm \qquad w^0,w^1\in \Vect^\infty\fstop
\end{align}

Again in analogy with the case of configuration spaces~\cite{RoeSch99}, this motivates to define the tangent space to~$\msP$ at a point~$\mu$ as the space~$\Vect_\mu\eqdef \cl_{L^2_\mu}\Vect^\infty$, that is, the abstract linear completion of~$\Vect^\infty$ with respect to the norm~$\norm{\emparg}_{\Vect_\mu}$ induced by the Hilbert scalar product~$\scalar{\emparg}{\emparg}_{\Vect_\mu}$, the (non-relabeled) extension to~$\Vect_\mu$ of the scalar product~\eqref{eq:Scalar}.
We shall also write~$T^\Der_\mu\msP_2$ for~$\Vect_\mu$ and thus~$T^\Der\msP_2$ for the associated fiber-``bundle''.
Denote by~$\Vect^\infty_\nabla\eqdef \nabla\mcC^\infty(\MFD)$ the family of vector fields \emph{of gradient type} and let us further set~$T^\nabla_\mu\msP_2\eqdef \cl_{\Vect_\mu} \Vect^\infty_\nabla$; the associated fiber-``bundle'' will be denoted by~$T^\nabla\msP_2$.
As it is well-established in the optimal transport literature, e.g.~\cite{AmbGig11,Gig11,Gig12,GanKimPac10}, the space~$T^\nabla_\mu\msP_2$ is the space of those geodesic directions at~$\mu$ that are induced by optimal transport maps in the sense of Brenier--McCann Theorem~\ref{t:McCann}.
In the following we will make use of both \emph{non-equivalent} definitions. An exhaustive discussion of this choice is postponed to \S\ref{ss:Tangent}. However, let us remark here that the two definitions are in fact equivalent on configuration spaces.

\begin{ass}\label{ass:P} Let~$\mbbP$ be a Borel probability measure on~$\msP_2$ satisfying:
\begin{enumerate}[label=\ensuremath{\ref{ass:P}\,(\roman*)}]
\item\label{ass:P0} $\mbbP$ has full support;
\item\label{ass:P1} $\mbbP$ has no atoms;
\item\label{ass:P3} $\mbbP$ verifies the following integration-by-parts formula. If~$u,v\in \Test^\infty$ and~$w\in \Vect^\infty$, then there exists a measurable function~$\mu\mapsto\grad_w^*v\in \Vect_\mu$ such that
\begin{align}\label{eq:DefDirForm}
\int_\msP \grad_wu \cdot v \diff\mbbP=\int_\msP u\cdot \grad_w^* v \diff\mbbP \semicolon
\end{align}
\item\label{ass:P4} $\mbbP$ is quasi-invariant with respect to the action of the family of flows~$\Flow(\MFD)$ on~$\msP$, i.e.~$\mbbP$ and~$\Fl^{w,t}_\pfwd\mbbP$ are mutually absolutely continuous for all~$w\in\Vect^\infty$ and~$t\in \R$. Moreover the Radon--Nikod\'ym derivative
\begin{align}\label{eq:assP:RN}
R_r^w \eqdef \frac{\diff(\Fl^{w,r}_\pfwd \mbbP) \otimes \diff r}{\diff \mbbP\otimes\diff r}\comm \qquad r\in\R \comm
\end{align}
satisfies
\begin{align}\label{eq:ass:P4}
\forallae{\mbbP}\mu\in\msP \qquad \Leb^1\textrm{-}\essinf_{r\in (s,t)}\, R_r^w(\mu)>0 \comm \qquad s,t\in\R\comm s\leq t \fstop
\end{align}
\end{enumerate}
\end{ass}

The validity and necessity of these four assumptions are widely illustrated through examples in~\S\ref{s:Examples}. In practice however, it is easier to verify the stronger Assumption~\ref{ass:P5} below. We show in Proposition~\ref{p:P5} that, if~$\mbbP$ verifies Assumption~\ref{ass:P5}, then it verifies Assumption~\ref{ass:P} as well.
\begin{ass}\label{ass:P5} 
$\mbbP$ satisfies Assumptions~\ref{ass:P0} and~\ref{ass:P4}, and the Radon--Nikod\'ym derivative~$R^w_r$ in~\eqref{eq:assP:RN} is additionally such that, for every~$w\in\Vect^\infty$,
\begin{itemize}
\item[$\bullet$] $r\mapsto R^w_r(\mu)$ is differentiable in a neighborhood of~$0$ for~$\mbbP$-a.e.~$\mu$;
\item[$\bullet$] $\mu\mapsto \abs{\partial_r R^w_r(\mu)}$ is integrable w.r.t.~$\mbbP$ uniformly in~$r$ on a neighborhood of~$0$.
\end{itemize}
\end{ass}

Our discussion of a differential structure on~$\msP$ is completed by defining a divergence operator on cylinder vector fields.

\begin{defs}[Cylinder vector fields]\label{d:VectCinfty}
Let~$\TestV^\infty\eqdef \Test^\infty\otimes_{\R}\Vect^\infty$ denote the vector space of \emph{cylinder vector fields} on~$\msP$, i.e.~the $\R$-vector space of sections~$W$ of~$T^\Der\msP$ of the form
\begin{align}\label{eq:TestV}
W(\mu)(x)= \sum_j^n v_j(\mu) w_j(x)\comm \qquad n\in\N\comm v_j\in\Test^\infty \comm w_j\in\Vect^\infty
\end{align}
By~$\otimes_{\R}$ we denoted here the \emph{algebraic} $\R$-tensor product.

Further, let~$\TestV_\mbbP$ be the abstract linear completion of~$\TestV^\infty$ endowed with the pre-Hilbertian norm
\begin{align*}
\norm{W}_{\TestV_\mbbP}\eqdef \tonde{\sum_j^n\int_\msP \abs{v_j(\mu)}^2\norm{w_j}_{\Vect_\mu}^2 \diff \mbbP(\mu) }^{1/2}\fstop
\end{align*}
\end{defs}

Setting, for~$W$ as in~\eqref{eq:TestV},
\begin{align*}
\Div_\mbbP W(\mu)\eqdef -\sum_i^n \grad_{w_i}^* v_i(\mu) \comm
\end{align*}
it follows by linearity from Assumption~\ref{ass:P3} that
\begin{align}\label{eq:Div}
 \int_\msP \tscalar{\grad u}{W}_{\Vect_\emparg} \diff\mbbP=-\int_\msP u \cdot \Div_\mbbP W \diff\mbbP \comm \qquad u\in \Test^\infty\comm W\in \TestV^\infty \fstop
\end{align}

Then, $(\Div_\mbbP, \TestV^\infty)$ is a densely defined linear operator from the space of sections~$\Gamma_{L^2_\mbbP} T^\Der\msP_2$ to $L^2_\mbbP(\msP)$ and we denote its adjoint by~$(\mbfd_\mbbP, \mbfW^{1,2})$. By definition, functions in~$\mbfW^{1,2}$ are weakly differentiable, in the sense that~\eqref{eq:Div} holds for all~$u\in \mbfW^{1,2}$ with~$\mbfd_\mbbP u$ in place of~$\grad u$.

\bigskip

We denote by~$\mcF$ the set of all bounded measurable functions~$u$ on~$\msP$ for which there exists a measurable section~$\mbfD u$ of~$T^\Der\msP_2$ such that
\begin{align}\label{eq:L2norm}
\mcE(u,u)\eqdef \int_\msP \scalar{\mbfD u(\mu)}{\mbfD u(\mu)}_{\Vect_\mu} \diff\mbbP(\mu)<\infty
\end{align}
and such that for every~$w\in\Vect^\infty$ and~$s\in\R$ there exists the directional derivative
\begin{align}\label{eq:DefF}
\frac{u\circ \Fl^{w,t} -u}{t}\xrightarrow{t\rar 0}\scalar{\mbfD u}{w}_{\Vect_\emparg}  \quad \textrm{in} \quad L^2(\msP,\Fl^{w,s}_\pfwd \mbbP)\fstop
\end{align}

Finally, set~$\mcF_\cont\eqdef\mcF\cap\mcC(\msP)$ and note that~$\Test^\infty\subset \mcF_\cont\subset \mcF$ and that, \emph{a priori}, every inclusion may be a strict one.

\medskip

Before stating the main result, we introduce the following ---~quite restrictive~--- assumption on the base space. We will comment extensively about this assumption, and about its connection with the Ma--Trudinger--Wang curvature condition, in~\S\ref{ss:AssB}.

\begin{ass}[Smooth Transport Property]\label{ass:B} We say that~$\MFD$ satisfies the \emph{smooth transport property} if, whenever~$\mu,\nu\in\msP$,~$\mu,\nu\ll \dvol_\mssg$ with smooth nowhere vanishing densities, then there exists a \emph{smooth} optimal transport map~$g\colon \MFD\rar \MFD$ mapping~$\mu$ to~$\nu$ in the sense of Thm.~\ref{t:McCann} below.
\end{ass}

\begin{thm}\label{t:main}
Suppose that~$\mbbP$ satisfies Assumptions~\ref{ass:P1} and~\ref{ass:P3}. Then,
\begin{enumerate}[label=\ensuremath{\ref{t:main}\, (\arabic*)}]
\item\label{i:t:main1} the bilinear forms~$(\mcE,\Test^\infty)$,~$(\mcE,\mcF_\cont)$ and~$(\mcE,\mcF)$ are closable and their closures, respectively denoted by~$(\mcE,\msF_0)$,~$(\mcE,\msF_\cont)$ and~$(\mcE,\msF)$ are strongly local Dirichlet forms, with~$\msF_0\subset \msF_\cont\subset \msF$;

\item\label{i:t:main2} for each~$u\in\msF$ there exists a measurable section~$\mbfD u$ of the tangent bundle~$T^\Der\msP_2$ such that
\begin{align}\label{eq:t:main.Identif}
\mbfD u=\grad u\comm\qquad u\in\Test^\infty\comm
\end{align}
and
\begin{align}\label{eq:Form}
\mcE(u,u)=\int_\msP \norm{\mbfD u(\mu)}_{\Vect_\mu}^2 \diff\mbbP(\mu)\comm
\end{align}
i.e.~the form~$(\mcE,\msF)$ admits carr\'e du champ~$\boldGamma(u)(\mu)\eqdef \norm{\mbfD u(\mu)}_{\Vect_\mu}^2$;

\item\label{i:t:main3} \emph{(Rademacher property)} let~$u\colon\msP\rar \R$ be~$W_2$-Lipschitz continuous. Then~$u\in\msF_\cont$ and, if additionally Assumption~\ref{ass:B} holds, then~$u\in\msF_0$. Furthermore, there exist a measurable set~$\Omega^u\subset\msP$ of full~$\mbbP$-measure and a measurable section~$\mbfD u$ of~$T^\Der\msP_2$, satisfying~\eqref{eq:t:main.Identif} and~\eqref{eq:Form}, such that
\begin{enumerate}[label=\ensuremath{\ref{t:main}\,(3.\roman*)}]
\item\label{i:t:main3.1} for all~$\mu\in\Omega^u$ it holds that~$\norm{\mbfD u(\mu)}_{\Vect_\mu}\leq \Lip[u]$;
\item\label{i:t:main3.2} if additionally Assumption~\ref{ass:P4} holds, then
\begin{align}\label{eq:i:t:main3.2-1}
\norm{\mbfD u(\mu)}_{\Vect_\mu}\leq \slo{u}(\mu)\comm \qquad \mu\in\Omega^u \comm
\end{align}
where~$\slo{u}$ is the \emph{slope} of~$u$ (see~\eqref{eq:slo} below), and, for all~$w\in\Vect^\infty$
\begin{align}\label{eq:i:t:main3.2-2}
\lim_{t\rar 0} \frac{(u\circ\Fl^{w,t}-u)(\emparg)}{t}=\scalar{\mbfD u(\emparg)}{w}_{\Vect_{\emparg}}
\end{align}
pointwise on~$\Omega^u$ and in~$L^2_{\mbbP}(\msP)$.
\end{enumerate}
\end{enumerate}
\end{thm}

We now collect some remarks on the statement of our main theorem.

\begin{rem}
Let us notice that, in the definition of the pre-domain~$\mcF$, we only ask for the existence of \emph{a} measurable section~$\mbfD u$ of~$T^\Der\msP_2$ with finite $L^2$-norm in the sense of~\eqref{eq:L2norm}. In particular, we do not require~\eqref{eq:t:main.Identif} which is only used to uniquely identify the form~$(\mcE,\msF)$ in the theorem. We pay this arbitrariness with the Fr\'echet differentiability~\eqref{eq:DefF} of~$u\in\mcF$ in~$L^2(\msP,\Fl^{w,s}_\pfwd \mbbP)$ for each real~$s$, a seemingly stronger condition than the conclusion~\eqref{eq:i:t:main3.2-2} of Theorem~\ref{i:t:main3.2}. In fact though, under Assumption~\ref{ass:P4}, condition~\eqref{eq:DefF} is equivalent to~\eqref{eq:i:t:main3.2-2} for every $W_2$-Lipschitz~$u$, since the~$L^2(\msP,\Fl^{w,s}_\pfwd\mbbP)$-topology is equivalent to the~$L^2(\msP,\mbbP)$-topology for every~$s$.
\end{rem}

\begin{rem}[On the definition of \emph{volume measure} on~$\msP_2$]\label{r:VolMeas}
Assumptions~\ref{ass:P0} and~\ref{ass:P1} are of a general kind. In fact, Assumption~\ref{ass:P0} is not necessary to the conclusions of Theorem~\ref{t:main}. Rather, it rules out some trivial cases, as, e.g., Example~\ref{ese:Trivial}.

On the contrary, Assumptions~\ref{ass:P3} and~\ref{ass:P4} are ---~as already noticed in~\cite[Rmk.~p.~329]{RoeSch99} for measures on configuration spaces~--- specifically proper of a volume measure, as discussed in the Introduction. In particular, Assumption~\ref{ass:P3} may be regarded as a form of gradient-divergence duality for~$\mbbP$.
Assumption~\ref{ass:P4}, and its stronger version Assumption~\ref{ass:P5}, is also expected from a differential geometry point of view and it is equally important in light of Proposition~\ref{p:P5} below.
\end{rem}

\begin{rem} As already noticed in the case of configuration spaces in~\cite[Prop.~1.4(iii)]{RoeSch99}, the Dirichlet forms~$(\mcE,\msF_0)$,~$(\mcE,\msF_\cont)$ and~$(\mcE,\msF)$ do in principle differ. A sufficient condition for their coincidence is the essential self-adjointness of the generator of~$(\mcE,\msF)$ on the core~$\Test^\infty$.

It is readily seen that, by compactness of~$\msP_2$ and the Stone--Weierstra\ss\, Theorem, the spaces $\Test^\infty$ and~$\mcF_\cont$ are uniformly dense in~$\mcC(\msP_2)$. Together with the Theorem, this implies that the Dirichlet forms~$(\mcE,\msF_0)$ and~$(\mcE,\msF_\cont)$ are regular strongly local Dirichlet forms on~$\msP_2$, thus properly associated to Markov diffusion processes by the theory of Dirichlet forms. See e.g.~\cite{MaRoe92}.
\end{rem}

\begin{rem}[On the definition of \emph{Rademacher-type} properties] Assume we have already shown that $u_\nu\colon \mu\mapsto W_2(\nu,\mu)$ belongs to~$\msF_\cont$, resp.~$\msF_0$, cf.~Lemmas~\ref{l:DerDist} and~\ref{l:Rademacher} below, and that $\boldGamma(u_\nu)\leq \car$. 
Then,~\iref{i:t:main3.1} may be deduced by the general results on not necessarily local Dirichlet forms in~\cite{FraLenWin14}. On the contrary ---~even if it is proven that the Dirichlet form~$(\mcE,\msF)$ is strongly local and regular~--- the finer estimate~\eqref{eq:i:t:main3.2-1} does not follow by~\cite[Thm.~2.1]{KosZho12}, where the reference measure~$\mbbP$ is assumed to be doubling. In fact it may be proved that \emph{no} doubling measure with full support exists on~$\msP_2$, since the latter is infinite-dimensional.

Both of the previous results may be considered as Rademacher-type properties for the Dirichlet form(s) in question. Nonetheless, in the case of the Wasserstein space~$\msP_2$, we have ---~in addition to the general assumptions of~\cite{FraLenWin14} or~\cite{KosZho12}~--- a good notion of \emph{directional} derivative for functions on~$\msP_2$. As a consequence, the statement of what we call a ``Rademacher Theorem on~$(\msP_2,W_2,\mbbP)$'' comprises more properly assertion~\iref{i:t:main3.2}, where we check that each directional derivative of a differentiable function~$u\in\msF$ along a smooth direction~$w\in\Vect^\infty$ coincides with the scalar product of the gradient~$\mbfD u$ and direction~$w$.
\end{rem}

To conclude this preliminary section we anticipate that the statement of our main theorem is non-void, and that our assumptions pose no restriction to the subset of measures in~$\msP$ whereon~$\mbbP$ is concentrated. In particular
\begin{rem*}[See Rmk.~\ref{r:Support} below]
Define
\begin{itemize}
\item $\msA_1$ the set of measures in~$\msP$ absolutely continuous w.r.t.~$\dvol_\mssg$;
\item $\msA_2$ the set of measures in~$\msP$ singular continuous w.r.t.~$\dvol_\mssg$;
\item $\msA_3$ the set of purely atomic measures in~$\msP$;
\item $\msA_4$ the set of \emph{transport regular} measures in~$\msP$ (Dfn.~\ref{d:TranspReg}).
\end{itemize}

Then,~$\MFD=\mbbS^1$ satisfies Assumption~\ref{ass:B}, and, for any choice of~$a_1,a_2,a_3\geq 0$ with $a_1+a_2+a_3=1$, there exists~$\mbbP\in\msP(\msP)$, satisfying Assumption~\ref{ass:P} and such that~$\mbbP(\msA_i)=a_i$ for every~$i=1,2,3$ and~$\mbbP(\msA_4)= a_1+a_2$.
\end{rem*}

\section{Preliminaries}\label{s:Prelim}
Everywhere in the following let~$(Y,\tau)$ be a second countable compact topological space with Borel $\sigma$-algebra~$\mcB$, and let~$\mssn$ be a finite measure on~$(Y,\mcB)$ with full support.
Note that every such space is Polish, i.e.~separable and completely metrizable, and that every finite measure on a Polish space is Radon.

Let~$\rho$ be any metric metrizing~$(Y,\tau)$.
In the rest of this section, the metric measure space~$(Y,\rho,\mssn)$ will play the role of~$(\msP,W_2,\mbbP)$.

\paragraph{Notation}
By a \emph{measure} we always mean a \emph{non-negative} measure. We denote by~$I$, resp.~$I^\circ$, the unit interval~$[0,1]$, resp.~$(0,1)$, always endowed with the standard metric, $\sigma$-algebra and with the one-dimensional Lebesgue measure~$\diff\Leb^1(r)=\diff r$. Analogously, we denote by~$\Leb^d$ the $d$-dimensional Lebesgue measure on~$\R^d$.
Set~$\mssm\eqdef\dvol_\mssg$. We indicate by~$\msP^\mssm\subset \msP$ the space of probability measures~$\mu\ll\mssm$, by~$\msP^\infty$ the subset of probability measures~$\mu\in\msP^\mssm$ with smooth densities, by~$\msP^{\infty,\times}$ the subset of measures in~$\msP^\infty$ whose densities with respect to~$\mssm$ are bounded away from~$0$; such densities are bounded above as a consequence of their continuity and of the compactness of~$\MFD$.

\subsection{Lipschitz functions}
We say that a real-valued function~$h\colon Y\rar \R$ is $L$-\emph{Lipschitz} with respect to $\rho$ if there exists a constant~$L>0$ such that
\begin{align*}
\abs{h(y_1)-h(y_2)}\leq L\, \rho(y_1,y_2) \comm \qquad y_1,y_2\in Y \comm
\end{align*}
in which case we denote by~$\Lip_\rho[h]$ the smallest such constant and by
\begin{align}\label{eq:slo}
\slo{h}_\rho(y)\eqdef \limsup_{z\rar y} \frac{\abs{h(y)-h(z)}}{\rho(y,z)}\leq L
\end{align}
the \emph{slope} (or \emph{local Lipschitz constant}) of~$h$ at a point~$y\in Y$.
The metric~$\rho$ is omitted in the notation whenever apparent from context. We set~$\rho_z(\emparg)\eqdef\rho(z,\emparg)$ and, for any~$A\eqdef \seq{a_i}_i^n\subset \R$ and~$E\eqdef\seq{z_i}_i^n\subset Y$, we let~$\rho_{A,E,L}(\emparg)\eqdef \vee_{i\leq n} (a_i-L\rho_{z_i}(\emparg))$.
For~$\eps>0$, an \emph{$\eps$-net} in~$Y$ is a finite set~$E_\eps\eqdef \seq{z_{\eps,i}}_i^{n_\eps}$ satisfying
\begin{align*}
\rho(z_{\eps,i},z_{\eps,j})>\eps/2\comm \quad i\neq j\comm \qquad \text{~~and~~} \qquad \sup_{y\in Y} \rho(y,E_\eps)\leq \eps\fstop
\end{align*}

\begin{lem}\label{l:McShane}
Fix~$h\in\Lip_\rho Y$ and~$\eps>0$. For a dense subset~$Z\subset Y$, let further~$E_\eps\eqdef \seq{z_{\eps,i}}_i^{n_\eps}\subset Z$ be an $\eps$-net in~$Y$, and set~$A_\eps\eqdef \seq{h(z_{\eps,i})}_i^{n_\eps}\subset \R$ and~$h_\eps\eqdef \rho_{A_\eps,E_\eps,\Lip[h]}$.
Then, there exists~$C_h>0$ such that
\begin{align*}
\Lip[h_\eps]\leq \Lip[h]\comm \qquad \norm{h-h_\eps}_{\mcC^0}\leq C_h \, \eps\fstop
\end{align*}

\begin{proof} The existence of~$E_\eps$ as above follows by density of~$Z$ in~$Y$ and compactness of~$Y$. This shows that the statement is well-posed.
The function~$h_\eps$ is $\rho$-Lipschitz continuous with~$\Lip[h_\eps]\leq \Lip[h]$ for it is a maximum of $\rho$-Lipschitz continuous functions with Lipschitz constants all equal to~$\Lip[h]$.
Since~$h$ is Lipschitz continuous, it coincides with its lower McShane extension~\cite{McS34}, i.e.~$h(y)=\sup_{z\in Y} \set{h(z)-\rho_y(z)}$. Thus,~$h_\eps\leq h$.
Furthermore, for all~$y\in Y$ there exists~$\bar z\eqdef \bar z(y)$ such that~$h(y)\leq h(\bar z)-\rho(y,\bar z)+\eps$ and, by definition of~$E_\eps$, there exists~$\bar \imath\eqdef \bar\imath(y)$ such that~$\rho(\bar z,z_{\eps,\bar \imath})\leq \eps$. Hence,
\begin{align*}
h_\eps(y)\leq h(y)\leq& h(\bar z)-\rho(y,\bar z)+\eps
\\
\leq& h(\bar z)-h(z_{\eps,\bar \imath})+h(z_{\eps,\bar\imath})-\rho(y,\bar z)+\rho(y,z_{\eps,\bar\imath})-\rho(y,z_{\eps,\bar\imath})+\eps
\\
\leq& h(z_{\eps,\bar\imath})-\rho(y,z_{\eps,\bar\imath})+\abs{h(\bar z)-h(z_{\eps,\bar \imath})}+\abs{\rho(y,z_{\eps,\bar\imath})-\rho(y,\bar z)}+\eps
\\
\leq& h_\eps(y)+\Lip[h]\eps+\eps+\eps
\end{align*}
respectively by definition of~$h_\eps$, Lipschitz continuity of~$h$ and by reverse triangle inequality and definition of~$z_{\eps,\bar\imath}$. The conclusion follows with~$C_h\eqdef \Lip[h]+2$.
\end{proof}
\end{lem}

\subsection{Dirichlet forms}\label{ss:DirForm}
Whenever~$(Q,\dom{Q})$ is a non-negative definite symmetric bilinear form, we denote by the same symbol the associated quadratic form, defined as~$Q(u)\eqdef Q(u,u)$ if~$u\in\dom{Q}$ and~$Q(u)\eqdef+\infty$ otherwise. 

We recall some basic facts about Dirichlet forms, following the exposition in~\cite[\S1.1B--D]{Stu95}. Let~$(\mcE,\dom\mcE)$ be a strongly local Dirichlet form on~$L^2_\mssn(Y)$, additionally such that~$\car\in\dom{\mcE}$ and~$\mcE(\car)=0$. Every such~$(\mcE,\dom\mcE)$ may be written as
\begin{align*}
\mcE(u,v)=\int_Y \diff\boldGamma(u,v)
\end{align*}
for all~$u,v\in\dom \mcE$, where~$\boldGamma$, called the \emph{energy measure of~$(\mcE,\dom\mcE)$}, is a non-negative definite symmetric bilinear form with values in the space of signed Radon measures on~$(Y,\mcB)$, and defined by the formula
\begin{align*}
\int_Y \varphi \diff\boldGamma(u,v)\eqdef \tfrac{1}{2}\ttonde{\mcE(u,\varphi v)+\mcE(v,\varphi u)-\mcE(uv,\varphi)}\comm \qquad \varphi\in \dom\mcE\cap \mcC(Y)
\end{align*}
for all~$u,v\in\dom\boldGamma\eqdef \dom\mcE\cap L^\infty_\mssn(Y)$. Note that~$\mcC(Y)=\mcC_c(Y)$ by compactness of~$Y$.

Say further that~$(\mcE,\dom\mcE)$ admits \emph{carr\'e du champ operator} if~$\boldGamma(u,v)\ll \mssn$ for every~$u,v\in\dom\boldGamma$, in which case, with usual abuse of notation, we indicate again by~$(\boldGamma,\dom\boldGamma)$ the $L^1_\mssn(Y)$-valued non-negative definite symmetric bilinear form~$\tfrac{\diff \boldGamma(u,v)}{\diff\mssn}$. By~$\boldGamma(u)\leq \mssn$ we mean that~$\boldGamma(u)$ is absolutely continuous with respect to~$\mssn$ and~$\boldGamma(u)\leq 1$ $\mssn$-a.e..

\begin{defs}[Intrinsic distance] A strongly local Dirichlet form~$(\mcE,\dom\mcE)$ on~$L^2_\mssn(Y)$ with carr\'e du champ operator~$\boldGamma$ induces an intrinsic extended pseudo-metric~$\mssd_\mcE$ on~$Y$, called the \emph{intrinsic distance of~$(\mcE,\dom\mcE)$} and defined by
\begin{align}\label{eq:IntMet}
\mssd_\mcE(y_1,y_2)\eqdef \sup \set{u(y_1)-u(y_2) : u\in\dom\boldGamma\cap \mcC(Y),  \boldGamma(u)\leq \mssn}\fstop
\end{align}
By \emph{extended} we mean that~$\mssd_\mcE$ may attain the value~$+\infty$, by the prefix \emph{pseudo-} that it may vanish outside the diagonal in~$Y^{\times 2}$.
\end{defs}

We will make wide use of the following lemma, which is thus worth to state separately. A proof is standard, see e.g.~\cite[Lem.~I.2.12]{MaRoe92} for the first part.
\begin{lem}\label{l:Ma} Let~$(\mcE,\dom\mcE)$ be a Dirichlet form on~$L^2_\mssn(Y)$ with energy measure~$(\boldGamma,\dom\boldGamma)$ and let $\seq{u_n}_n\subset \dom\mcE$ be such that~$\sup_n \mcE(u_n)<\infty$. If there exists~$u\in L^2_\mssn(Y)$ such that~$L^2_\mssn$-$\nlim u_n=u$, then
\begin{align*}
u\in\dom\mcE \quad\textrm{ and }\quad \mcE(u)\leq \nliminf \mcE(u_n)\fstop
\end{align*}

\noindent If additionally~$\seq{u_n}_n\!\subset \!\dom\boldGamma$ and~$\nlimsup \!\boldGamma(u_n)\!\leq\!\mssn$, then, additionally,~$u\in\dom\boldGamma$ and~$\boldGamma(u) \!\leq \! \mssn$.
\end{lem}

\begin{lem}\label{l:Koskela}
Let~$(\mcE,\dom\mcE)$ be a not necessarily regular, strongly local Dirichlet form on~$L^2_\mssn(Y)$ with energy measure~$(\boldGamma,\dom\boldGamma)$. Let~$\rho$ be any metric metrizing~$(Y,\tau)$ and assume further that $\rho_z\eqdef \rho(z,\emparg)\in\dom\boldGamma$ and~$\boldGamma(\rho_z)\leq \mssn$ for every~$z\in Z$ a dense subset of~$Y$.
Then, every~$\rho$-Lipschitz function~$u\colon Y\rar \R$ satisfies~$u\in\dom\boldGamma$ and~$\boldGamma(u)\leq \Lip[u]^2\,\mssn$.

\begin{proof} Without loss of generality, up to rescaling, we can restrict ourselves to the case when~$\Lip[u]\leq 1$, for which we claim~$\boldGamma[u]\leq \mssn$. Let~$u_\eps$ be defined as in Lemma~\ref{l:McShane}. Since~$Y$ is compact and~$\car\in\dom{\mcE}$, functions locally in the domain of the form belong to~$\dom\mcE$, thus we have~$u_\eps\in\dom\mcE$ and~$\boldGamma(u_\eps)\leq \mssn$ by~\cite[Thm.~2.1]{KosZho12} in which the regularity of~$(\mcE,\dom\mcE)$ is in fact not needed and the fact that~$\boldGamma(\rho_{z_i})\leq\mssn$ is granted by assumption.
Choose now~$\eps\eqdef\eps_n\searrow 0$ as~$n\rar \infty$. Since~$u_{\eps_n}$ converges to~$u$ uniformly as~$n\rar \infty$ by Lemma~\ref{l:McShane}, the conclusion follows by Lemma~\ref{l:Ma}. 
\end{proof}
\end{lem}

\subsection{Optimal transport}
We collect here some known results in metric geometry based on optimal transport. The reader is referred to~\cite{AmbGig11} for an expository treatment.
Everywhere in the following let~$\exp_x\colon T_x\MFD\rar \MFD$ be the exponential map of~$(\MFD,\mssg)$ at a point~$x\in \MFD$ and set~$\mssc\eqdef \tfrac{1}{2}\mssd^2\colon \MFD^{\times 2}\rar \R$.

\begin{defs}[$\mssc$-transform, $\mssc$-convexity, conjugate map]
For any $\phi\colon \MFD\rar \R$, we define its \emph{$\mssc$-transform} (often:~$\mssc_-$-transform, e.g.,~\cite{AmbGig11}) by
\begin{align}\label{eq:CConjugate}
\phi^\mssc(x)\eqdef -\inf_{y\in M}\set{\mssc(x,y)+\phi(y)} \fstop
\end{align}
Any such~$\phi$ is called \emph{$\mssc$-convex} if there exists~$\psi\colon \MFD\rar \R$ with~$\phi=\psi^\mssc$.
\end{defs}

\begin{rem}
If~$\phi$ is $\mssc$-convex, then~$\phi=(\phi^\mssc)^\mssc$~\cite[Dfn.~1.9]{AmbGig11}, and~$\phi$ is Lipschitz on~$\MFD$. By the classical Rademacher Theorem on~$\MFD$, the set~$\Sigma_\phi$ of singular points of~$\phi$ is $\mssm$-negligible.
For a proof of the Lipschitz property, see~\cite[Prop.~1.30]{AmbGig11}, where the statement is proven for~$\mssc$-\emph{concave} functions. It is equivalent to our claim by~\cite[Rmk.~1.12]{AmbGig11}.
\end{rem}

\begin{defs}[Regular measures]\label{d:TranspReg}
We say that~$\mu\in\msP$ is (\emph{transport}-)\emph{regular} if~$\mu \Sigma_\phi=0$ for every semi-convex function~$\phi$. We denote the set of regular measures in~$\msP$ by~$\msP^\reg$.
\end{defs}

It is well-known that every finite measure on a Polish space is \emph{regular} in the classical sense of measure theory. Thus we will henceforth refer to transport-regular measures simply as to \emph{regular} measures. Since we only consider finite measures on Polish spaces, no confusion may arise.

\begin{rem} The above definition of a regular measure is rather intrinsic. Regularity is a local property. For an extrinsic definition in local charts we refer the reader to~\cite[Dfn.~2.8]{Gig11}. The equivalence of our definition to the one in~\cite{Gig11} is shown in the proof of~\cite[Prop.~2.10]{Gig11}.
\end{rem}

\begin{thm}[Brenier--McCann,~{\cite[Thm.~1.33]{AmbGig11},~Gigli,~\cite[Prop.~2.10 and Thm.~7.4]{Gig11}}]\label{t:McCann}
The following are equivalent:
\begin{enumerate}[$(i)$]
\item $\mu\in\msP^\reg$;
\item for each $\nu\in\msP$ there exists a unique optimal transport plan~$\pi\in\Opt(\mu,\nu)$ and~$\pi$ is induced by a map~$g_{\mu\rar\nu}$.
\end{enumerate}

Furthermore, if any of the previous holds, then there exists a $\mssc$-convex~$\phi_{\mu\rar\nu}$, unique up to additive constant, called a \emph{Kantorovich potential}, such that $g_{\mu\rar\nu}=\exp\nabla\phi_{\mu\rar\nu}$ $\mu$-a.e.~on $\MFD$.
\end{thm}

\begin{prop}[AC curves in~$(\msP_2,W_2)$,~{\cite[Thm.~2.29]{AmbGig11}}]\label{p:AC}
For every~$\seq{\mu_t}_{t\in I}\in\AC^1(I;\msP_2)$  there exists a Borel measurable time-dependent family of vector fields~$\seq{w_t}_{t\in I}$ such that~$\norm{w_t}_{\Vect_{\mu_t}}\leq \abs{\dot\mu_t}$ for $\diff t$-a.e.~$t\in I$ and the \emph{continuity equation}
\begin{align}\label{eq:ContEq}
\partial_t\mu_t+\div(w_t\mu_t)=0 
\end{align}
holds in the sense of distributions on~$I\times \MFD$, that is
\begin{align}\label{eq:ContEqDistr}
\int_0^1 \!\! \int_\MFD \ttonde{\partial_t\phi(t,x) +\gscal{\nabla\phi(t,x)}{w_t(x)}}\diff\mu_t(x) \diff t=0 \comm \qquad \phi\in \mcC^\infty_c(I\times \MFD) \fstop
\end{align}

Conversely, if~$\seq{\mu_t,w_t}_{t\in I}$ satisfies~\eqref{eq:ContEq} in the sense of distributions and~$\norm{w_t}_{\Vect_{\mu_t}}\in L^1(I)$, then, up to redefining~$t\mapsto \mu_t$ on a $\diff t$-negligible set of times,~$\seq{\mu_t}_t\in \AC^1(I;\msP_2)$ and~$\abs{\dot\mu_t}\leq \norm{w_t}_{\Vect_{\mu_t}}$ for $\diff t$-a.e.~$t\in I$.
\end{prop}

\subsection{Geometry of~\texorpdfstring{$\msP_2$}{P2}}
A detailed study of the Riemannian structure of~$\msP_2$ has been carried out by N.~Gigli in~\cite{Gig11,Gig12}. We shall need the following definitions and results from~\cite{Gig11} to which we refer the reader for further references.

We consider the tangent bundle~$T\MFD$ as endowed with the Sasaki metric~$\mssg_*$ and the associated Riemannian distance~$\mssd_*\eqdef\mssd_{\mssg_*}$ which turn it into a non-compact connected oriented Riemannian manifold.

\begin{defs}[Tangent plans]
For~$\mu\in\msP_2$ we let~$\msP_2(T\MFD)_\mu\subset \msP_2(T\MFD)$ be the space of \emph{tangent plans}~$\boldgamma\in\msP(T\MFD)$ such that
\begin{align}
\label{eq:TPa}
\pr^\MFD_\pfwd\boldgamma=\mu 
\qquad \text{and} \qquad
\int_{T\MFD}  \abs{\mrmv}_{\mssg_x}^2 \diff\boldgamma(x,\mrmv)<\infty \fstop
\end{align}
\end{defs}

\begin{defs}[Exponential map]
We denote by~$\bexp_\mu\colon \msP_2(T\MFD)_\mu\rar \msP_2$ the \emph{exponential map} $\bexp_\mu(\boldgamma)=\exp_\pfwd \boldgamma$, with right-inverse $\bexp_\mu^{-1}\colon \msP_2\rar \msP_2(T\MFD)_\mu$ defined by
\begin{align*}
\bexp_\mu^{-1}(\nu)\eqdef \set{\boldgamma\in\msP_2(T\MFD)_\mu : \bexp_\mu(\boldgamma)=\nu \text{~~and~~} \int_{T\MFD} \abs{\mrmv}_{\mssg_x}^2 \diff\boldgamma(x,\mrmv)=W^2_2(\mu,\nu)} \fstop
\end{align*}
\end{defs}

Equivalently, $\bexp_\mu^{-1}(\nu)$ is the set of all tangent plans~$\boldgamma\in\msP_2(T\MFD)$ such that
\begin{align}
\label{eq:TPb}
(\pr^\MFD,\exp)_\pfwd \boldgamma\in& \ \Cpl(\mu,\nu)
\intertext{and}
\label{eq:TPc}
\int_{T\MFD} \abs{\mrmv}^2_{\mssg_x} \diff\boldgamma(x,\mrmv)=&\ W^2_2(\mu,\nu)\fstop
\end{align}

\begin{rem}[{Cf.~\cite[p.~131]{Gig11}}] Notice that~\eqref{eq:TPc} may not be dropped even if~\eqref{eq:TPb} is strengthened to
\begin{align}\label{eq:TPb'}
(\pr^\MFD,\exp)_\pfwd \boldgamma\in\Opt(\mu,\nu)\fstop
\end{align}

The joint requirement of~\eqref{eq:TPb} \emph{and}~\eqref{eq:TPc} is however equivalent to that of~\eqref{eq:TPb'} \emph{and}~\eqref{eq:TPc}.
\end{rem}

\begin{rem}
Notably,~$\bexp_\mu^{-1}(\nu)$ need not be a singleton even when~$\Opt(\mu,\nu)$ is. Consider e.g.~the case when~$\mu=\delta_p$ and~$\nu=\delta_q$ are Dirac masses at antipodal points~$p,q\in\mbbS^1$ and let~$\mrmv\eqdef \tfrac{1}{2}\partial_p\in T_p\mbbS^1$. Then~$\Cpl(\mu,\nu)=\Opt(\mu,\nu)=\tset{\delta_{(p,q)}}$, yet
$\bexp_{\mu}^{-1}(\nu)=\set{(1-r)\delta_{p,\mrmv}+ r\delta_{p,-\mrmv} : r\in I}$.
\end{rem}

\begin{defs}[Rescaling of tangent plans]
For~$t\in\R$ we denote by~$t\cdot\boldgamma$ the \emph{rescaling}
\begin{align}\label{eq:RescaledPlan}
t\cdot \boldgamma\eqdef (\pr^\MFD,t\, \pr^1)_\pfwd \boldgamma \fstop
\end{align}
\end{defs}

\begin{defs}[Double tangent]
We denote by~$T^2\MFD\eqdef \tset{(x,\mrmv_1,\mrmv_2) : \mrmv_1,\mrmv_2\in T_x\MFD}$ the \emph{double tangent bundle} to~$\MFD$, with natural projections
\begin{align*}
\pr^\MFD\colon (x,\mrmv_1,\mrmv_2)\mapsto x\in \MFD\comm \qquad \pr^i\colon (x,\mrmv_1,\mrmv_2)\mapsto \mrmv_i\in T_x\MFD\comm i=1,2
\end{align*}
and endowed with the distance
\begin{align*}
\mssd_{*2}\eqdef \ttonde{\mssd_*^2\circ(\pr^1,\pr^1)+ \mssd_*^2\circ(\pr^2,\pr^2)}^{1/2}\fstop
\end{align*}
\end{defs}

All of the previous definitions are instrumental to the statement of the following result by N.~Gigli, concerned with the one-sided differentiability of the squared $L^2$-Wasserstein distance along a family of nice curves including $W_2$-geodesic curves.

\begin{thm}[Directional derivatives of the squared Wasserstein distance,~{\cite[Thm.~4.2]{Gig11}}]\label{t:DirDerDistGigli}
Fix~$\mu_0\in\msP$ and~$\boldgamma\in\msP_2(T\MFD)_{\mu_0}$ and set~$\mu_t\eqdef \bexp_{\mu_0}(t \cdot \boldgamma)$. Then, for every~$\nu\in\msP$ there exists the right derivative
\begin{align}\label{eq:t:DirDerDistGigli0}
\diff_t^+\restr_{t=0} \tfrac{1}{2} W_2^2(\mu_t,\nu)=-\sup_{\boldalpha}\int_{T^2\MFD} \scalar{\mrmv_1}{\mrmv_2}_{\mssg_x} \diff\boldalpha(x, \mrmv_1,\mrmv_2)
\end{align}
where the supremum is taken over all~$\boldalpha\in\msP_2(T^2\MFD)$ such that
\begin{equation}\label{eq:t:DirDerDistGigli1}
(\pr^\MFD,\pr^1)_\pfwd \boldalpha=\boldgamma\comm \qquad \textrm{and} \qquad
(\pr^\MFD,\pr^2)_\pfwd\boldalpha\in \bexp^{-1}_{\mu_0}(\nu) \fstop
\end{equation}
\end{thm}

The following is a straightforward corollary. We provide a proof for the sake of completeness.

\begin{cor}\label{c:DirDerDistGigli}
In the same notation of Theorem~\ref{t:DirDerDistGigli}, there exists the left derivative
\begin{align}\label{eq:c:DirDerDistGigli:0}
\diff_t^-\restr_{t=0} \tfrac{1}{2} W_2^2(\mu_t,\nu)=-\inf_{\boldalpha}\int_{T^2\MFD} \scalar{\mrmv_1}{\mrmv_2}_{\mssg_x} \diff\boldalpha(x, \mrmv_1,\mrmv_2)
\end{align}
where the infimum is taken over all~$\boldalpha\in\msP_2(T^2\MFD)$ satisfying~\eqref{eq:t:DirDerDistGigli1}.

\begin{proof}
Given~$\boldgamma^+\in\msP_2(T\MFD)_{\mu_0}$ let~$\boldgamma^{-}\eqdef (-1)\cdot\boldgamma$ be defined by~\eqref{eq:RescaledPlan} and set~$\mu_t^{\pm}\eqdef \bexp_\mu(t\cdot\boldgamma^\pm)$ for~$t\geq 0$. Notice that~$\mu^+_{-t}=\mu^{-}_t$ for every~$t\geq 0$, hence, by definition,
\begin{align*}
\diff^{-}_t\restr_{t=0} \tfrac{1}{2}W_2^2(\mu_t^+,\nu)=-\diff_t^+\restr_{t=0}\tfrac{1}{2}W_2^2(\mu_t^-,\nu)
\end{align*}
which exists by choosing~$\boldgamma=\boldgamma^-$ in Theorem~\ref{t:DirDerDistGigli}. Let~$A^\pm$ be the set of plans~$\boldalpha\in\msP_2(T^2\MFD)$ satisfying~\eqref{eq:t:DirDerDistGigli1} with~$\boldgamma^\pm$ in place of~$\boldgamma$ and define~$\re^1\eqdef (\pr^\MFD,-\pr^1,\pr^2)\colon T^2\MFD\rar T^2\MFD$. It is straightforward that~$A^\pm=\re^1_\pfwd A^\mp$, thus, by Theorem~\ref{t:DirDerDistGigli},
\begin{align*}
-\diff_t^+\restr_{t=0}\tfrac{1}{2}W_2^2(\mu_t^-,\nu)=&\sup_{\boldalpha\in A^-} \int_{T^2\MFD} \scalar{\mrmv_1}{\mrmv_2}_{\mssg_x} \diff\boldalpha(x,\mrmv_1,\mrmv_2)
\\
=&\sup_{\boldalpha\in A^+} \int_{T^2\MFD} \scalar{-\mrmv_1}{\mrmv_2}_{\mssg_x} \diff\boldalpha(x,\mrmv_1,\mrmv_2)
\\
=&-\inf_{\boldalpha\in A^+} \int_{T^2\MFD} \scalar{\mrmv_1}{\mrmv_2}_{\mssg_x} \diff\boldalpha(x,\mrmv_1,\mrmv_2) \comm
\end{align*}
whence the conclusion by combining the last two chains of equalities.
\end{proof}
\end{cor}

\section{Proof of the main result}\label{s:Proof}
Let us briefly outline the proof of Theorem~\ref{t:main}. The proofs of~\iref{i:t:main1} and~\iref{i:t:main2} are mainly a standard consequence of Assumption~\ref{ass:P3} and the abstract theory of Dirichlet forms. Therefore, we shall rather focus on the proofs of Theorem~\iref{i:t:main3.1} and~\iref{i:t:main3.2}. These are respectively reminiscent of the ``global'' proof of the Rademacher Theorem for strongly local regular Dirichlet forms in the work~\cite{KosZho12}  by P.~Koskela and Y.~Zhou, and of the ``local'' proof~\cite{NekZaj88} of the classical Rademacher Theorem on~$\R^n$ by A.~Nekvinda and L.~Zaj\'i\v{c}ek.
Informally, the ``global'' method consists in showing that Lipschitz functions belong to the Sobolev space~$W^{1,2}_\loc$ (here: to the domain~$\msF$) and that~$\abs{\nabla \emparg}\leq \Lip[\emparg]$. The ``local'' method consists instead in showing that a Lipschitz function is G\^ateaux differentiable (along every direction) at a.e.~point and, subsequently, that a Lipschitz function G\^ateaux differentiable at a point is in fact Fr\'echet differentiable at that point.

\paragraph{A ``local'' proof of~\ref{i:t:main3.2} under quasi-invariance}
Let~$u\colon\msP_2\rar\R$ be a $W_2$-Lipschitz function. In~\S\ref{ss:DiffFlow}, we study the differentiability of any such~$u$ along the flow curves~$t\mapsto \Fl^{w,t}\mu$.
In Corollary~\ref{c:RoeSch99} we show how we may trade the~$\diff t$-a.e.~differentiability of the curve~$t\mapsto u(\Fl^{w,t}\mu)$ for every~$\mu$, with the differentiability of~$t\mapsto u(\Fl^{w,t}\mu)$ at~$t=0$ (i.e.~the G\^ateaux differentiability of~$u$ along~$w$) for $\mbbP$-a.e.~$\mu$. This result is the only one relying on the quasi-invariance Assumption~\ref{ass:P4}.
Finally, assuming that~$u$ is G\^ateaux differentiable along every~$w$ in a countable dense subspace of directions, we improve its differentiability to the Fr\'echet differentiability in~\ref{i:t:main3.2}. This is the content of Proposition~\ref{p:RoeSch99}, adapted from the proof of~\cite[Thm.~1.3]{RoeSch99}.

\paragraph{A ``global'' proof of~\ref{i:t:main3.1} without quasi-invariance}
Set~$U_{\mu,\nu}(t)\eqdef W_2(\nu,\Fl^{w,t}\mu)$. In~\S\ref{ss:DiffCone} we show that~$t\mapsto U_{\mu,\nu}(t)$ is differentiable at~$t=0$ for every fixed~$\nu$ in the dense set~$\msP^\reg$ and every fixed~$\mu\neq \nu$. This shows the G\^ateux differentiability of~$u_\nu\eqdef W_2(\nu,\emparg)$ at every~$\mu\neq \nu$ along every~$w\in\Vect^\infty$, hence we can apply Proposition~\ref{p:RoeSch99} to conclude~\ref{i:t:main3.1} for all metric-cone functions~$u_\nu$ with~$\nu\in\msP^\reg$. The extension to all Lipschitz functions is given by the general Lemma~\ref{l:Koskela} with~$Z=\msP^\reg$.

\paragraph{Refined statements under the smooth transport property} We start by considering truncated metric cones of the form~$u_\nu\vee\theta$ with~$u_\nu$ as above and~$\theta>0$. In Lemma~\ref{l:Rademacher} we construct an approximating sequence~$u_{\nu,\theta,n}\in \Test^\infty$ and $\mcE_1$-convergent to~$u_\nu\vee\theta$ as~$n\rar\infty$. This shows that~$u_\nu\vee\theta\in \msF_0$ for every~$\nu\in\msP_2$. By applying Proposition~\ref{p:RoeSch99} to~$u_\nu\vee\theta$ we improve the bound~$\norm{\mbfD (u_\nu\vee\theta)(\mu)}_{\Vect_\mu}\lesssim \theta^{-1}$ obtained in Lemma~\ref{l:Rademacher} to the sharp bound~$\norm{\mbfD(u_\nu\vee\theta)(\mu)}_{\Vect_\mu}\leq \Lip[u_\nu\vee\theta]=1$. By Lemma~\ref{l:Ma} we may then remove the truncation by~$\theta$, thus showing that~$u_\nu\in\msF_0$ for every~$\nu\in\msP_2$. Once more, the extension to all Lipschitz functions follows by Lemma~\ref{l:Koskela}.

\subsection{On the differentiability of \texorpdfstring{$W_2$}{W2}-cone functions}\label{ss:DiffCone}
In this section we collect some results on the differentiability of the Wasserstein distance along (flow) curves. 
We aim to show (Lem.~\ref{l:DerDist}) that, for nice~$\mu$ and~$\nu\in\msP_2$, the function~$t\mapsto U_{\mu,\nu}(t)\eqdef W_2(\nu,\Fl^{w,t}\mu)$ is differentiable at~$t=0$. Since~$U_{\mu,\nu}\geq 0$, with equality only if~$\mu=\nu$, it suffices to study the differentiability of~$t\mapsto U_{\mu,\nu}(t)^2$.
In order to do so, we firstly need to trace back the computation of~$\diff_t\restr_{t=0}U_{\mu,\nu}(t)^2$ to the setting of Gigli's Theorem~\ref{t:DirDerDistGigli}. Informally, we exploit the following fact: The curves~$\fl^{w,t}(x)$ and~$\exp_x(tw)$ are tangent to each other at every point~$x\in \MFD$ (Lem.~\ref{l:ExpLie}). As a consequence, their lifts on~$\msP_2$ by push-forward, namely~$\Fl^{w,t}\mu\eqdef\fl^{w,t}_\pfwd \mu$ and~$\bexp_\mu(t\cdot \boldgamma)=\exp_\emparg(t w_\emparg)_\pfwd \mu$ (here:~$\boldgamma$ is some tangent plan depending on~$w$) are themselves tangent to each other in a suitable sense. See Step~2 in the proof of Lem.~\ref{l:DerDist}.
Setting~$V_{\mu,\nu}(t)\eqdef W_2(\nu,\bexp_\mu(t\cdot \boldgamma))$, this shows that~$\diff_t\restr_{t=0} U_{\mu,\nu}(t)=\diff_t\restr_{t=0} V_{\mu,\nu}(t)$. Thus, as for~$U_{\mu,\nu}$, it suffices to show the differentiability of~$t\mapsto V_{\mu,\nu}(t)^2$.

While Theorem~\ref{t:DirDerDistGigli} and Corollary~\ref{c:DirDerDistGigli} provide the \emph{one-sided} differentiability of~$t\mapsto V_{\mu,\nu}(t)^2$ at $t=0$, the two-sided differentiability generally fails, since the left and right derivatives need not coincide. However,~$t\mapsto V_{\mu,\nu}(t)^2$ is differentiable as soon as we show that, for given~$\mu$,~$\nu$, the set of plans~$\boldalpha$ over which we are extremizing in the right-hand sides of~\eqref{eq:t:DirDerDistGigli0} and~\eqref{eq:c:DirDerDistGigli:0} is in fact a singleton. This is the case if either~$\mu$ or~$\nu$ is regular, as shown in Proposition~\ref{p:SingleValued}.

\medskip

We denote by~$\inj_\MFD>0$ the injectivity radius of~$\MFD$.

\begin{lem}\label{l:ExpLie} Let~$w\in\Vect^\infty$. Then,
\begin{align*}
\mssd\ttonde{\exp_x (tw_x), \fl^{w,t}(x)}\in o(t) \quad \textrm{as }\, t\rar 0
\end{align*}
uniformly in~$x\in \MFD$.
\begin{proof}
Let~$\seq{\partial_i}_{i=1,\dotsc, d}$ be a $\mssg$-orthonormal basis of~$T_x\MFD$,~$\seq{\diff_i}_{i=1,\dotsc,d}$ be its $\mssg$-dual basis in~$T_x^*\MFD$ and recall the Lie series expansion of~$\fl^{w,t}$ about~$t=0$, viz.
\begin{align*}
f\ttonde{\fl^{w,t}(x)}=\sum_{k\geq 0} \frac{t^k}{k!} w^k(f)_x \comm \qquad f\in\mcC^\infty(\MFD) \fstop
\end{align*}
Set~$c_0\eqdef\inj_\MFD\ttonde{1\wedge \norm{w}_{\Vect^0}^{-1}}$ and let~$0<c_1<c_0$ be such that~$\fl^{w,t}(x)\in B_{c_0}(x)$ for all~$t<c_1$. Letting~$w_x=\mrmw^j\partial_j$ and choosing~$f=\diff_i\circ \exp_x^{-1}$ (suitably restricted to a coordinate chart around~$x$) above yields
\begin{align*}
(\diff_i\circ\exp_x^{-1})(\fl^{w,t}(x))=&(\diff_i\circ\exp^{-1}_x)(x)+tw(\diff_i\circ\exp^{-1}_x)_x+o(t)
\\
=&t\mrmw^j\partial_j(\diff_i\circ\exp_x^{-1})_x+o(t)=t\mrmw^i+o(t)\comm
\end{align*}
whence~$(\exp_x^{-1}\circ\fl^{w,t})(x)=tw+o(t)$. Since~$\exp_x$ is a smooth diffeomorphism on~$B_{c_1}(\zero_{T_x\MFD})$, there exists~$L>0$ such that
\begin{align*}
\mssd(y_1,y_2)\leq L\abs{\exp_{x}^{-1}(y_1)-\exp_x^{-1}(y_2)} \comm \qquad y_1,y_2\in B_{c_1}(x) \fstop
\end{align*}

Thus, finally
\begin{align*}
\mssd\ttonde{\exp_x(tw),\fl^{w,t}(x)}\leq&L\abs{tw-tw-o(t)}_{\mssg_x}\in o(t) \comm
\end{align*}
which concludes the proof.
\end{proof}
\end{lem}

\begin{prop}\label{p:SingleValued}
Let either~$\mu\in\msP^\reg$ or~$\nu\in\msP^\reg$. Then,~$\bexp_\mu^{-1}(\nu)$ is a singleton.

\begin{proof} Assume first~$\mu\in\msP^\reg$. By Theorem~\ref{t:McCann} there exists a $\mssc$-convex~$\phi$ (unique up to additive constant) such that
\begin{align}\label{eq:tp0}
\nu=(\exp_\emparg\nabla \phi_\emparg)_\pfwd \mu \qquad \text{and} \qquad W_2^2(\mu,\nu)=\int_\MFD \mssd^2(x,\exp_x\nabla\phi_x) \diff\mu(x) \fstop
\end{align}

Moreover, for~$\mu$-a.e.~$x\in \MFD$ there exists a unique geodesic~$\seq{\alpha^x_r}_{r\in I}$ connecting~$x$ to~$g_{\mu\rar\nu}(x)$ given by~$\alpha^x_r\eqdef\exp_x (r\nabla \phi_x)$, cf.~\cite[Rmk.~1.35]{AmbGig11}. We call this property the \emph{geodesic uniqueness} property.

\paragraph{Claim: $\bexp_\mu^{-1}(\nu)\neq \emp$. Proof} Set~$\boldgamma_0\eqdef (\id_\MFD(\emparg),\nabla \phi_\emparg)_\pfwd \mu\in\msP(T\MFD)$. It is straightforward that $\boldgamma_0\in\msP_2(T\MFD)_\mu$. Additionally,
\begin{align}\label{eq:tp1}
\int_{T\MFD} \abs{\mrmv}_{\mssg_x}^2 \diff\boldgamma_0(x,\mrmv)=&\int_\MFD \abs{\nabla\phi_x}_{\mssg_x}^2 \diff\mu(x)
\\
\nonumber
=&\int_\MFD \mssd(x,\exp_x\nabla\phi_x)^2 \diff\mu(x) =W_2^2(\mu,\nu) \comm
\end{align}
where~$\abs{\nabla\phi_x}_{\mssg_x}=\mssd(x,\exp_x\nabla\phi_x)$ for~$\mu$-a.e.~$x$ by geodesic uniqueness. This shows~\eqref{eq:TPc}, hence that~$\boldgamma_0\in\bexp_\mu^{-1}(\nu)$.

\paragraph{Claim:~$\bexp_\mu^{-1}(\nu)=\set{\boldgamma_0}$. Proof} Let~$\boldgamma\in\bexp_\mu^{-1}(\nu)$. By~\eqref{eq:TPa},~$\pr^\MFD_\pfwd \boldgamma=\mu$, thus there exists the Rokhlin disintegration~$\set{\boldgamma^x}_{x\in \MFD}$ of~$\boldgamma$ along~$\pr^\MFD$ with respect to~$\mu$. By~\eqref{eq:TPb'},~$\bexp_\mu\boldgamma=\nu=(\exp_\emparg\nabla\phi_\emparg)_\pfwd\mu$, thus, for~$\mu$-a.e.~$x\in \MFD$,~$\boldgamma^x$ is concentrated on the set~$A_x\eqdef\exp_x^{-1}(\exp_x\nabla \phi_x)$. Moreover,~\eqref{eq:tp1} holds with~$\boldgamma$ in place of~$\boldgamma_0$ by~\eqref{eq:TPc}, hence, by optimality,~$\boldgamma^x$ is in fact concentrated on the set
\begin{align}\label{eq:DefPrExp}
\pr^{A_x}(\zero_{T_x\MFD})\eqdef \argmin_{\mrmv\in T_x\MFD} \dist_{\mssg_x}(A_x,\zero_{T_x\MFD}) \fstop
\end{align}

By geodesic uniqueness, one has~$\pr^{A_x}(\zero_{T_x\MFD})=\set{\nabla\phi_x}$ for~$\mu$-a.e.~$x\in \MFD$, hence~$\boldgamma^x=\bolddelta_{(x,\nabla \phi_x)}$ for~$\mu$-a.e.~$x\in \MFD$. Thus finally~\eqref{eq:tp1} holds and~$\boldgamma=\boldgamma_0$.

\bigskip

Assume now~$\nu\in\msP^\reg$. By Theorem~\ref{t:McCann} there exists a~$\mssc$-convex~$\psi$, unique up to additive constant, so that~\eqref{eq:tp0} holds when exchanging~$\nu$ with~$\mu$ and replacing~$\phi$ with~$\psi$. Moreover, geodesic uniqueness holds too, for the geodesics defined by~$\beta^y_r\eqdef \exp_y (r \nabla\psi_y)$.

For a measurable vector field~$w$, we denote by~$\boldTau_s^t(\seq{\alpha})w_{\alpha_s}$ the parallel transport (of the Levi-Civita connection) from~$\alpha_s$ to~$\alpha_t$ of the vector~$w_{\alpha_s}$ along the curve~$\seq{\alpha}\eqdef\seq{\alpha_r}_r$.
We set further
\begin{align*}
\mbfR\colon T\MFD&\longrightarrow T\MFD
\\
(x,\mrmv)&\longmapsto \tonde{\exp_x \mrmv, -\boldTau_0^1\ttonde{\seq{ \exp_x(r\mrmv)}_r}\mrmv}
\end{align*}

\paragraph{Claim: $\bexp_\mu^{-1}(\nu)\neq \emp$. Proof} Set~$\boldgamma_0\eqdef \mbfR_\pfwd (\id_\MFD(\emparg),\nabla\psi_\emparg)_\pfwd \nu$. Since~
\begin{align*}
\pr^\MFD\circ \mbfR\circ(\id_\MFD(\emparg),\nabla\psi_\emparg)=\exp_\emparg \nabla\psi_\emparg 
\end{align*}
and~$\mu=(\exp_\emparg\nabla\psi_\emparg)_\pfwd \nu$, then~$\boldgamma_0\in\msP_2(T\MFD)_\mu$. Additionally,
\begin{align*}
\int_{T\MFD} \abs{\mrmv}_{\mssg_x}^2 \diff\boldgamma_0(x,\mrmv) =& \int_\MFD \abs{-\boldTau_0^1(\seq{\beta^y_r}_r)\nabla\psi_y}_{\mssg_{x(y)}}^2 \diff\nu(y) \qquad x(y)\eqdef \beta^y_1
\\
=&\int_\MFD \abs{\nabla\psi_y}_{\mssg_y}^2 \diff\nu(y) \comm
\end{align*}
where the last equality holds since, being~$\seq{\beta^y_r}_r$ a geodesic and the Levi-Civita connection being a metric connection, the parallel transport
\begin{align*}
\boldTau_0^1(\seq{\beta^y_r}_r)\colon (T_{\beta^y_0}\MFD,\mssg_{\beta^y_0})\rar (T_{\beta^y_1}\MFD,\mssg_{\beta^y_1})\end{align*}
 is an isometry. Thus, arguing as in the proof of the first claim,~$\boldgamma_0\in \bexp_\mu^{-1}(\nu)$.
 
\paragraph{Claim: $\bexp_\mu^{-1}(\nu)=\set{\boldgamma_0}$. Proof} Let~$\boldgamma\in\bexp_\mu^{-1}(\nu)$. By definition~$\bexp_\mu\boldgamma=\exp_\pfwd\boldgamma=\nu$, thus there exists the Rokhlin disintegration~$\set{\boldgamma^y}_{y\in \MFD}$ of~$\boldgamma$ along~$\exp$ with respect to~$\nu$.
By~\eqref{eq:TPb'},
\begin{align*}
(\id_\MFD(\emparg), \exp_\emparg -)_\pfwd\boldgamma\in\Opt(\mu,\nu)=&(\pr^2,\pr^1)_\pfwd \Opt(\nu,\mu)
\\
=&\set{(\exp_\emparg \nabla\psi_\emparg,\id_\MFD(\emparg))_\pfwd \nu} \comm
\end{align*}
thus, for~$\nu$-a.e.~$y\in \MFD$, the measure~$\boldgamma^y$ is concentrated on the set
\begin{align*}
C_y\eqdef \exp_{\beta^y_1}^{-1}(y)\subset T_{\beta^y_1}\MFD \fstop
\end{align*}

By a similar reasoning to that in the second claim, for~$\nu$-a.e.~$y\in \MFD$,~$\boldgamma^y$ is in fact concentrated on~$\pr^{C_y}(\zero_{T_{\beta^y_1}\MFD})$, defined analogously to~\eqref{eq:DefPrExp}. By definition of parallel transport and since~$\seq{\beta^y_r}_r$ is a geodesic, the latter set is a singleton
\begin{align*}
\pr^{C_y}(\zero_{T_{\beta^y_1}\MFD})=\set{-\boldTau_0^1(\seq{\beta^y_r}_r)\nabla \psi_y}\fstop
\end{align*}

This concludes the proof analogously to that of the second claim.
\end{proof}
\end{prop}

\begin{lem}[Derivatives of the Wasserstein distance along flow curves]\label{l:DerDist}
Fix~$w\in\Vect^\infty$, $\mu_0\in\msP$ and set~$\mu_t\eqdef \Fl^{w,t} \mu_0$. Then, for every~$\nu\in\msP\setminus\set{\mu_0}$, there exists the right derivative
\begin{align}\label{eq:DirDerDist0}
\diff^+_t\restr_{t=0} W_2(\mu_t,\nu)=-W_2^{-1}(\mu_0,\nu)\,\sup_\boldgamma\int_{T\MFD} \scalar{w_x}{\mrmv}_{\mssg_x} \diff\boldgamma(x,\mrmv)
\end{align}
where the supremum is taken over all~$\boldgamma\in\bexp_{\mu_0}^{-1}(\nu)$.
Moreover, if additionally either $\mu_0\in\msP^\reg$ or $\nu\in\msP^\reg$, then there exists the two-sided derivative~$\diff_t\restr_{t=0} W_2(\mu_t,\nu)$.

\begin{proof} The proof is divided into several steps. Firstly, let~$\mu_t'\eqdef (\exp_\emparg(tw))_\pfwd \mu_0$ and~$\boldgamma$ be as above. We show that there exists
\begin{align}\label{eq:DirDerDist1}
\lim_{t\downarrow 0} \frac{W_2(\mu'_t,\nu)-W_2(\mu_0,\nu)}{t}=-W_2^{-1}(\mu_0,\nu)\,\sup_\boldgamma\int_{T\MFD} \scalar{w_x}{\mrmv}_{\mssg_x} \diff\boldgamma(x,\mrmv) \fstop
\end{align}

Next, profiting the fact that for small~$t>0$ the flow $\exp_{\emparg}(tw_{\emparg})$ is tangent to the flow~$\fl^{w,t}(\emparg)$ at each point in~$\MFD$, we show that the same holds for the corresponding lifted flows~$(\exp_{\emparg}(tw_{\emparg}))_\pfwd$ and~$(\fl^{w,t}(\emparg))_\pfwd$ at each point in~$\msP$, hence that the right derivative~\eqref{eq:DirDerDist0} exists and coincides with~\eqref{eq:DirDerDist1}. 

\paragraph{Step 1} Set~$\iota^w\eqdef (\id_\MFD(\emparg), w_{\emparg})\colon \MFD\rar T\MFD$, let~$\boldgamma_0\eqdef \iota^w_\pfwd \mu_0\in\msP_2(T\MFD)$ and notice that
\begin{align*}
\bexp_{\mu_0}(t\cdot \boldgamma_0)= \tonde{\exp_\pfwd \circ (\pr^\MFD,t\, \pr^1)_\pfwd \circ \iota^w_\pfwd} \mu_0=(\exp_{\emparg}(tw_\emparg))_\pfwd \mu_0\defeq \mu_t' \fstop
\end{align*}

By Theorem~\ref{t:DirDerDistGigli}, there exists the right derivative
\begin{align*}
\diff_t^+\restr_{t=0} \tfrac{1}{2} W_2^2(\mu_t',\nu)=-\sup_{\boldalpha}\int_{T^2\MFD} \scalar{\mrmv_1}{\mrmv_2}_{\mssg_x}  \diff\boldalpha(x, \mrmv_1,\mrmv_2)
\end{align*}
where~$\boldalpha$ is as in~\eqref{eq:t:DirDerDistGigli1}. In particular, for every such~$\boldalpha$, it holds that~$(\pr^\MFD,\pr^1)_\pfwd\boldalpha=\boldgamma_0=\iota^w_\pfwd \mu_0$, that is~$(\pr^\MFD,\pr^1)_\pfwd\boldalpha$ is supported on the graph~$\Graph(\iota^w)\subset T\MFD$ of the map~$\iota^w$. As a consequence,~$\boldalpha$ is concentrated on the set
\begin{align*}
\set{(x,\mrmv_1,\mrmv_2) : (x,\mrmv_1)\in\Graph(\iota^w)}=\set{(x,w_x,\mrmv_2)\in T^2\MFD}\subset T^2\MFD \comm
\end{align*}
thus, in fact
\begin{align*}
\diff_t^+\restr_{t=0} \tfrac{1}{2} W_2^2(\mu_t',\nu)=-\sup_\boldgamma\int_{T\MFD} \scalar{w_x}{\mrmv}_{\mssg_x} \diff\boldgamma(x,\mrmv) 
\end{align*}
where the supremum is taken over all~$\boldgamma\in \bexp_{\mu_0}^{-1}(\nu)$. The existence of~$\diff_t^+\restr_{t=0}W_2(\mu_t',\nu)$ and~\eqref{eq:DirDerDist1} follow from the existence of~$\diff_t^+\restr_{t=0}\tfrac{1}{2} W_2^2(\mu_t',\nu)$ by chain rule.

\paragraph{Step 2} By Lemma~\ref{l:ExpLie} there exists a constant~$c_1>0$ such that
\begin{align}\label{eq:DirDerDist3}
\mssd^2\ttonde{\exp_x(tw),\fl^{w,t}(x)}\in o(t^2)\comm \qquad t\in (0,c_1)\comm x\in \MFD \fstop
\end{align}

Furthermore, since~$\ttonde{\exp_{\emparg}(tw), \fl^w_t(\emparg)}_\pfwd \mu_0$ is a coupling between~$\mu_t'$ and~$\mu_t$, equation~\eqref{eq:DirDerDist3} yields
\begin{align*}
W_2^2(\mu_t',\mu_t)\leq \int_\MFD \mssd^2\ttonde{\exp_x(tw),\fl^w_t(x)}\diff\mu_0(x)  \in o(t^2) \comm \qquad t\in (0,c_1) \comm
\end{align*}
thus there exists
\begin{align*}
\diff_t\restr_{t=0}W_2(\mu_t',\mu_t)=\lim_{t\rar 0} \tfrac{1}{t}\abs{W_2(\mu'_t,\mu_t)-W_2(\mu_0,\mu_0)}= 0\fstop
\end{align*}

\paragraph{Step 3} By triangle inequality
\begin{align*}
W_2(\mu_t,\nu)-W_2(\mu_0,\nu)\leq W_2(\mu_t,\mu_t')+W_2(\mu_t',\nu)-W_2(\mu_0,\nu)\comm
\end{align*}
while by reverse triangle inequality
\begin{align*}
W_2(\mu_t,\nu)-W_2(\mu_0,\nu)\geq& \abs{W_2(\nu,\mu_t')-W_2(\mu_t',\mu_t)}-W_2(\mu_0,\nu)\\
\geq& W_2(\mu_t',\nu)-W_2(\mu_0,\nu)-W_2(\mu_t',\mu_t) \fstop
\end{align*}

As a consequence, setting
\begin{align*}
\overline\diff_t^+\restr_{t=0} W_2(\mu_t,\nu)\eqdef& \limsup_{t\downarrow 0} \frac{W_2(\mu'_t,\nu)-W_2(\mu_0,\nu)}{t} \comm\\
\underline\diff_t^+\restr_{t=0} W_2(\mu_t,\nu)\eqdef& \liminf_{t\downarrow 0} \frac{W_2(\mu'_t,\nu)-W_2(\mu_0,\nu)}{t} \comm
\end{align*}
one has
\begin{align*}
-\diff_t\restr_{t=0} W_2(\mu_t,\mu_t')+\diff_t^+\restr_{t=0} W_2(\mu_t',\nu)\leq&\\
\underline\diff_t^+\restr_{t=0} W_2(\mu_t,\nu)\leq& \overline\diff_t^+\restr_{t=0} W_2(\mu_t,\nu)\\
\leq& \diff_t\restr_{t=0} W_2(\mu_t,\mu_t')+\diff_t^+\restr_{t=0} W_2(\mu_t',\nu)
\end{align*}
where the derivatives above exist by the previous steps. Since~$\diff_t\restr_{t=0} W_2(\mu_t,\mu_t')=0$ by Step 2, the right derivative~$\diff_t^+\restr_{t=0} W_2(\mu_t,\nu)$ exists and coincides with~\eqref{eq:DirDerDist1}.

The last assertion follows by Step~1 and Corollary~\ref{c:DirDerDistGigli} since~$\bexp_{\mu_0}^{-1}(\nu)$ is a singleton by Proposition~\ref{p:SingleValued}.
\end{proof}
\end{lem}

\begin{lem}\label{l:Rademacher} Let~$(\MFD,\mssg)$ be additionally satisfying Assumption~\ref{ass:B}. Then, for every~$\nu\in\msP$ and every~$\theta>0$, the function~$u_{\nu,\theta}\colon \mu\mapsto W_2(\nu,\mu)\vee \theta$ belongs to~$\msF_0$.
\begin{proof} We construct an approximation of~$u_{\nu,\theta}$ by functions in~$\Test^\infty$.
\paragraph{Preliminaries} By Kantorovich duality~\cite[Thm.~1.17]{AmbGig11}
\begin{align*}
W_2^2(\nu,\mu)=2 \cdot \sup\set{\nu\psi+\mu\phi}
\end{align*}
where the supremum is taken over all~$(\psi,\phi)\in L^1_\nu(\MFD)\times L^1_\mu(\MFD)$ satisfying~$\psi(x)+\phi(y)\leq\mssc(x,y)$ for $\nu$-a.e.~$x$ and $\mu$-a.e.~$y$ in~$\MFD$. An optimal pair~$(\psi,\phi)$ always exists and satisfies~$\psi=\phi^\mssc$ $\nu$-a.e., with~$\phi^\mssc$ the~$\mssc$-conjugate~\eqref{eq:CConjugate} of~$\phi$.

Let~$\msP^{\infty,\times}$ be the set of measures in~$\msP^\infty$ with densities bounded away from~$0$ and fix a countable set~$\seq{\mu_i}_i\subset\msP^{\infty,\times}$ and dense in~$\msP_2$.

\paragraph{Construction of the approximation} 
We start by showing that~$W_2(\nu,\emparg)\vee\theta\in\msF_0$ for fixed~$\nu\in \msP^{\infty,\times}$. Let~$(\psi_i,\phi_i)$ be the optimal pair of Kantorovich potentials for the pair~$(\nu,\mu_i)$, so that
\begin{align}\label{eq:Rademacher0}
\tfrac{1}{2} W_2^2(\nu,\mu_i)=\nu\psi_i+\mu_i\phi_i \fstop
\end{align}
By assumption,~$\phi_i$ and~$\psi_i$ are smooth maps for all~$i$'s.
Let further~$\mbft\eqdef\seq{t_1,\dotsc, t_n}\in \R^n$ and, for small~$\eps>0$, let $F_{n,\eps}\colon\R^n\rar [-\eps,\infty)$ be a smooth regularization of the function~$F_n(\mbft)\eqdef 2 \cdot \max_{i\leq n} t_i$. Since~$F_n$ is~$2$-Lipschitz for every~$n$, the functions~$F_{n,\eps}$ may be chosen in such a way that
\begin{subequations}
\begin{equation}\label{eq:RademacherA}
\lim_{\eps\downarrow 0} F_{n,\eps}(\mbft)=F_n(\mbft)\comm\qquad n\in \N\comm \mbft\in\R^n\semicolon
\end{equation}
\begin{equation}\label{eq:RademacherA'}
F_{n,\eps_1}(\mbft)\geq F_{n,\eps_2}(\mbft)\comm \qquad n\in \N\comm \eps_1<\eps_2 \comm \mbft\in\R^n\semicolon
\end{equation}
\begin{equation}\label{eq:RademacherA''}
2\cdot \car_{B_{n,i}}\leq  \partial_i F_{n,\eps}\leq 2\cdot \car_{(B_{n,i})_{\eps}}\comm \qquad n\in \N\comm i\leq n\comm \eps>0\comm
\end{equation}
\end{subequations}
where
\begin{align}\label{eq:RademacherSetB}
B_{n,i}\eqdef \set{\mbft\in \R^n : \; \begin{matrix} t_i>t_j \textrm{ for all } 1\leq j<i \\
t_i\geq t_j \textrm{ for all } i\leq j\leq n \end{matrix}}
\end{align}
and, for any~$B\subset \R^n$, we put $B_\eps\eqdef \set{\mbft\in \R^n : \dist(\mbft,B)< \eps}$.

For small~$0<\delta<\theta$, let~$\varrho_{\theta,\delta}\colon\R\rar [\theta-\delta,\infty)$ be a smooth regularization of~$\varrho_\theta\colon t\mapsto \sqrt{t\vee \theta}$ such that 
\begin{subequations}
\begin{equation}\label{eq:RademacherB}
\lim_{\delta\downarrow 0} \varrho_{\theta,\delta}(t)=\varrho_{\theta}(t)\comm \qquad 0<\delta<\theta \comm t\in\R\semicolon
\end{equation}
\begin{equation}\label{eq:RademacherB'}
\varrho_{\theta,\delta_1}(t)\geq \varrho_{\theta,\delta_2}(t)\comm\qquad 0<\delta_1<\delta_2<\theta \comm t\in\R\semicolon
\end{equation}
\begin{equation}\label{eq:RademacherB''}
\car_{[\theta,\infty)}/(2\varrho_\theta)\leq \varrho'_{\theta,\delta}\leq \car_{[\theta-\delta,\infty)}/(2\varrho_\theta) \fstop
\end{equation}
\end{subequations}

Now, by smoothness of all functions involved, the function~$u_{\theta,n,\eps,\delta}\colon \msP\rar\R$ defined by
\begin{align*}
u_{\theta,n,\eps,\delta}(\mu)\eqdef \varrho_{\theta,\delta}\tonde{F_{n,\eps}(c_1+\phi_1^\bid\mu,\dotsc, c_n+\phi_n^\bid\mu)} \quad \textrm{where} \quad c_i\eqdef \psi_i^\bid\nu \fstop
\end{align*}
belongs to~$\Test^\infty$ and one has
\begin{align*}
\grad u_{\theta,n,\eps,\delta}(\mu)=&\sum_i^n \varrho_{\theta,\delta}' \tonde{F_{n,\eps}(c_1+\phi_1^\bid\mu,\dotsc, c_n+\phi_n^\bid\mu)}\times\\
&\times (\partial_i F_{n,\eps})(c_1+\phi_1^\bid\mu,\dotsc, c_n+\phi_n^\bid\mu) \nabla\phi_i \fstop
\end{align*}

By~\eqref{eq:RademacherA} and~\eqref{eq:RademacherA'}, resp.~\eqref{eq:RademacherB} and~\eqref{eq:RademacherB'}, and Dini's Theorem,~$\lim_{\eps\downarrow 0} \lim_{\delta\downarrow 0} (\varrho_{\theta,\delta} \circ F_{n,\eps})(\mbft)=(\varrho_\theta\circ F_n)(\mbft)$ locally uniformly in~$\mbft\in\R^n$ and for all~$n$ and~$\theta>0$. As a consequence, for all~$n$ and uniformly in~$\mu\in\msP$
\begin{align}\label{eq:Rademacher1}
\lim_{\eps\downarrow 0} \lim_{\delta\downarrow 0} u_{\theta, n,\eps,\delta}(\mu)= u_{\theta, n}(\mu)\eqdef \varrho_\theta \ttonde{F_{n}(c_1+\phi_1^\bid\mu,\dotsc, c_n+\phi_n^\bid\mu)} \fstop
\end{align}

Moreover, by~\eqref{eq:RademacherA''}, resp.~\eqref{eq:RademacherB''},~$\lim_{\eps\downarrow 0} \partial_i F_{n,\eps}= 2\cdot \car_{B_{n,i}}$ pointwise on~$\R^n$ for all~$i\leq n$, for all~$n$, resp.~$\lim_{\delta\downarrow 0}\varrho_{\theta,\delta}'\colon t\mapsto \car_{[\theta,\infty)}/(2\varrho_\theta)$ pointwise on~$\R$ for all~$\theta>0$. Thus, for all~$n$ and for all~$\mu\in\msP$ one has
\begin{align}\label{eq:Rademacher2}
\lim_{\eps\downarrow 0} \lim_{\delta\downarrow 0} \grad u_{\theta,n,\eps,\delta}(\mu)=\sum_i^n \frac{\car_{A_{\theta,n,i}}(\mu) }{\varrho_\theta\ttonde{F_{n}(c_1+\phi_1^\bid\mu,\dotsc, c_n+\phi_n^\bid\mu)}}\nabla \phi_i \comm
\end{align}
where the sets
\begin{align*}
A_{\theta, n,i}\eqdef \set{\mu\in\msP : \; \begin{matrix} c_i+\phi_i^\bid\mu \geq \theta \\ c_i+\phi_i^\bid\mu> c_j+\phi_j^\bid\mu \textrm{ for all } 1\leq j<i\\ c_i+\phi_i^\bid\mu\geq c_j+\phi_j^\bid\mu \textrm{ for all } i\leq j \leq n \end{matrix}}
\end{align*}
are, for all~$n$, measurable by continuity of~$\phi_i^\bid$ and pairwise disjoint for all~$i\leq n$, since the same holds for~$B_{n,i}$ in~\eqref{eq:RademacherSetB}.

Finally, again by Brenier--McCann Theorem,~$\abs{\nabla\phi_i}_\mssg\leq \diam \MFD$, hence
\begin{align*}
\abs{\grad u_{\theta,n,\eps,\delta}(\mu)(x)}_\mssg\leq n(\diam \MFD)/\sqrt{\theta}
\end{align*}
whence, by Dominated Convergence,~\eqref{eq:Rademacher1} and~\eqref{eq:Rademacher2},
\begin{align*}
\mcE_1^{1/2}\textrm{-}\lim_{\eps\downarrow 0} \ttonde{ \mcE_1^{1/2}\textrm{-}\lim_{\delta\downarrow 0} u_{\theta,n, \eps, \delta}}=u_{\theta, n}\in \msF_0\comm\\
\mbfD u_{\theta,n}(\mu)(x)=\sum_i^n \frac{ \car_{A_{\theta,n,i}}(\mu) }{\varrho_\theta\tonde{F_{n}(c_1+\phi_1^\bid\mu,\dotsc, c_n+\phi_n^\bid\mu)}}\nabla \phi_i(x) \fstop
\end{align*}

\paragraph{Pre-compactness of the approximation}
Since $L^2_{\mbbP}\textrm{-}\nlim u_{\theta,n}=u_{\nu,\theta}$ by Dominated Convergence and~\eqref{eq:Rademacher1}, by Lemma~\ref{l:Ma} it suffices to show that
\begin{align}\label{eq:Rademacher5}
\forallae{\mbbP} \mu\qquad \nlimsup \boldGamma(u_{\theta,n})(\mu) \leq C_{\nu,\theta}
\end{align}
for some constant~$C_{\nu,\theta}$ to get~$u_{\nu,\theta}\in \msF_0$ and~$\boldGamma(u_{\nu,\theta})\leq C_{\nu,\theta}$ $\mbbP$-a.e.. Indeed,
\begin{equation}\label{eq:Bound}
\begin{aligned}
\norm{\mbfD u_{\theta,n}(\mu)}_{\Vect_\mu}^2=&\int_\MFD \abs{\mbfD u_{\theta,n}(\mu)(x)}_\mssg^2 \diff \mu(x)\\
=&\sum_i^n\frac{ \car_{A_{\theta, n,i}}(\mu) }{\varrho_\theta^2\ttonde{F_{n}(c_1+\phi_1^\bid\mu,\dotsc, c_n+\phi_n^\bid\mu)}} \int_\MFD \abs{\nabla\phi_i}_\mssg^2 \diff \mu
\end{aligned}
\end{equation}
since the sets~$A_{\theta,n,i}$ are pairwise disjoint. Thus
\begin{align*}
\norm{\mbfD u_{\theta,n}(\mu)}_{\Vect_\mu}^2=& \sum_i^n\frac{\car_{A_{\theta,n,i}}(\mu) }{\theta\vee 2(c_i+\phi_i^\bid\mu)} \int_\MFD \abs{\nabla\phi_i}_\mssg^2 \diff \mu \leq \frac{(\diam \MFD)^2}{\theta}\defeq C_\theta \fstop
\end{align*}

\paragraph{General case}
Fix an arbitrary~$\nu\in\msP$ and let~$\seq{\nu_k}_k$ be a sequence in~$\msP^{\infty,\times}$ narrowly converging to~$\nu$. It is readily seen that~$u_{\nu_k,\theta}$ converges to~$u_{\nu,\theta}$ in~$L^2_{\mbbP}(\msP)$ and~$\norm{\mbfD u_{\nu_k,\theta}}_{\Vect_{\emparg}}^2\leq C_\theta$ $\mbbP$-a.e.~by the previous step. Thus,~$u_{\nu,\theta}\in \msF_0$ and~$\norm{\mbfD u_{\nu,\theta}}_{\Vect_{\emparg}}^2\leq C_\theta$ $\mbbP$-a.e.~by Lemma~\ref{l:Ma}.
\end{proof}
\end{lem}

\subsection{On the differentiability of functions along flow curves}\label{ss:DiffFlow}
\begin{lem}\label{l:FlowAC}
Fix $w\in\Vect^\infty$,~$\mu_0\in\msP$ and set~$\mu_t\eqdef \Fl^{w,t} \mu_0$. Then, the curve~$\seq{\mu_t}_{t\in \R}$ is Lipschitz continuous with Lipschitz constant~$M\leq \norm{w}_{\Vect^0}$ and satisfies~$\abs{\dot\mu_t}=\norm{w}_{\Vect_{\mu_t}}$ for every~$t\in\R$.
\begin{proof}
Since constant functions are in particular Lipschitz, we can assume without loss of generality~$w\neq 0$.
Set~$c_1\eqdef \inj_\MFD/\norm{w}_{\Vect^0}$ and let~$\mu_{t,\eps}'\eqdef (\exp_\emparg(\eps w))_\pfwd \mu_t$. For $\eps\in (-c_1,c_1)$, the curve $\eps\mapsto \exp_x(\eps w)$ is a minimizing geodesic. Thus, $(\exp_\emparg(\eps w))_\pfwd \mu_t\in \Opt(\mu_t,\mu_{t,\eps}')$ and, for every~$t\in\R$,
\begin{align*}
\diff_\eps\restr_{\eps=0} W_2(\mu_t,\mu_{t,\eps}')=\lim_{\eps\rar0}\tonde{\frac{1}{\eps^2}\int_\MFD \mssd^2\ttonde{x,\exp_x(\eps w)} \diff\mu_t(x)}^{1/2}=\norm{w}_{\Vect_{\mu_t}} \fstop
\end{align*}

Arguing as in Step 3 in the proof of Lemma~\ref{l:DerDist} with~$\mu_{t,\eps}'$, $\mu_{t+\eps}$ and~$\mu_t$ in place of~$\mu_t'$,~$\mu_t$ and~$\nu$ respectively,
\begin{align*}
\abs{\dot\mu_t}\eqdef\diff_\eps\restr_{\eps=0} W_2(\mu_t,\mu_{t+\eps})=\diff_\eps\restr_{\eps=0} W_2(\mu_t,\mu_{t,\eps}') \fstop
\end{align*}

Combining the last two equalities yields the second assertion. Moreover, by~\cite[Thm.~1.1.2]{AmbGigSav08},
\begin{align*}
W_2(\mu_s,\mu_t)\leq& \int_s^t \abs{\dot\mu_r} \diff r=\int_s^t \norm{w}_{\Vect_{\mu_r}} \diff r \leq \norm{w}_{\Vect^0} \abs{t-s} \comm \qquad s<t \fstop
\end{align*}

This concludes the proof.
\end{proof}
\end{lem}

\begin{lem}
\label{l:DirLip}
Fix $w\in\Vect^\infty$,~$\mu_0\in\msP$ and set~$\mu_t\eqdef \Fl^{w,t} \mu_0$. If~$u$ is $L$-Lipschitz continuous, then the map~$U\colon t\mapsto u(\mu_t)$ is Lipschitz continuous with ~$\Lip[U]\leq L\norm{w}_{\Vect^0}$ for every choice of~$\mu_0$ and
\begin{align}\label{eq:l:DirLip0}
\slo{U}(t) \leq& \slo{u}(\mu_t) \norm{w}_{\Vect_{\mu_t}}\comm \qquad t\in \R \fstop
\end{align}
\begin{proof} The Lipschitz continuity of~$U$ follows from those of~$u$ and~$t\mapsto\mu_t$ (Lem.~\ref{l:FlowAC}). By definition of slope,
\begin{align*}
\slo{U}(t) \leq\limsup_{\nu\rar \mu_t}\frac{\abs{u(\mu_t)-u(\nu)}}{W_2(\mu_t,\nu)} \limsup_{s\rar t}\frac{W_2(\mu_t,\mu_s)}{\abs{t-s}}= \slo{u}(\mu_t)\abs{\dot\mu_t}
\end{align*}
for every~$t\in\R$, whence~\eqref{eq:l:DirLip0} again by Lemma~\ref{l:FlowAC}.
\end{proof}
\end{lem}

\begin{lem}\label{l:DistrDer}
Let~$\seq{\mu_t}_{t\in I}$ be an absolutely continuous curve in~$\msP_2$ connecting~$\mu_0$ to~$\mu_1$, and~$(\mu_t,w_t)$ be any distributional solution of the continuity equation~\eqref{eq:ContEq}. Then, for every~$u\in\Test^\infty$ there exists for a.e.~$t\in \R$ the derivative
\begin{align*}
\diff_t u(\mu_t)=\scalar{\grad u(\mu_t)}{w_t}_{\Vect_{\mu_t}} \comm
\end{align*}
and one has
\begin{align}\label{eq:l:DistrDer0}
u(\mu_t)-u(\mu_0)=\int_0^t  \scalar{\grad u(\mu_s)}{w_s}_{\Vect_{\mu_s}} \diff s \fstop
\end{align}
\begin{proof}
Let~$f$ be in~$\mcC^\infty(\MFD)$,~$\phi\in\mcC^\infty_c(\R)$ be an arbitrary test function and denote by~$\scalar{\emparg}{\emparg}$ the canonical duality pair of distributions. Then,
\begin{align*}
\scalar{\diff_t f^\bid\mu_t}{\phi}=&\int_{\R} \phi'(t)\, f^\bid\mu_t \diff t=\int_{\R} \phi'(t) \int_\MFD f(x)  \diff\mu_t(x) \diff t\\
=&\int_{\R} \int_\MFD \partial_t(f\phi)(t,x) \diff\mu_t(x) \diff t\\
=&\int_{\R} \phi(t) \int_\MFD \gscal{\nabla f(x)}{w_t(x)} \diff\mu_t(x) \diff t
\end{align*}
for any time dependent vector field~$\seq{w_t}_t$ such that~$\seq{\mu_t,w_t}_t$ is a solution of~\eqref{eq:ContEq}. Thus the distributional derivative is representable by
\begin{align*}
\diff_t f^\bid\mu_t=\int_\MFD \gscal{\nabla f(x)}{w_t(x)} \diff\mu_t(x)
\end{align*}
and
\begin{align*}
\abs{\diff_t f^\bid\mu_t}\leq \norm{\nabla f}_{\mcC^0} \norm{w_t}_{\Vect_{\mu_t}} \fstop
\end{align*}

By Proposition~\ref{p:AC} and absolute continuity of~$\seq{\mu_t}_t$ the function~$t\mapsto \norm{w_t}_{\Vect_{\mu_t}}$ is in~$L^1_{\mathrm{loc}}(\R)$. Thus~$t\mapsto \diff_t f^\bid\mu_t$ is itself in~$L^1_{\mathrm{loc}}(\R)$.
Let now~$u\eqdef F\circ(f_1^\bid,\dotsc, f_k^\bid)\in\Test^\infty$. The above reasoning yields, in the sense of distributions,
\begin{align*}
\diff_t u(\mu_t)=&\sum_i^k (\partial_i F)(f_1^\bid\mu_t,\dotsc, f_k^\bid\mu_t) \diff_t (f_i^\bid\mu_t)=\sum_i^k (\partial_i F)(f_1^\bid\mu_t,\dotsc, f_k^\bid\mu_t) \int_\MFD \gscal{\nabla f_i}{w_t} \diff\mu_t \\
=& \scalar{\grad u(\mu_t)}{w_t}_{\Vect_{\mu_t}}\comm
\end{align*}
where~$\seq{\mu_t,w_t}_t$ is a solution of~\eqref{eq:ContEq} as above and we used~\eqref{eq:grad0}. Since~$t\mapsto \grad u(\mu_t)$ is continuous and bounded by definition of~$u$, the distributional derivative of the function~$t\mapsto u(\mu_t)$ is again representable by some function in~$L^1_{\mathrm{loc}}(\R)$. Thus, the Fundamental Theorem of Calculus applies and one has
\begin{align*}
u(\mu_t)-u(\mu_0)=\int_0^t \diff_r\restr_{r=s} u(\mu_r) \diff s=\int_0^t \scalar{\grad u(\mu_r)}{w_r}_{\Vect_{\mu_r}} \diff s \fstop
\end{align*}

This concludes the proof.
\end{proof}
\end{lem}

The following Lemma is taken ---~almost verbatim~--- from~\cite{RoeSch99}.

\begin{lem}[{\cite[Lem.~6.1]{RoeSch99}}]\label{l:DirG}
Fix~$w\in\Vect^\infty$. Then, for every bounded measurable~$u\colon \msP\rar\R$ and every~$v\in \Test^\infty$, for every~$t\in \R$
\begin{align}
\label{eq:l:DirG1}
\int \tonde{u\circ \Fl^{w,t} - u} v \diff\mbbP =& - \int_0^t \!\! \int u\circ \fl^{w,s}_\pfwd \cdot \grad_w^* v \, \diff\mbbP \, \diff s \fstop
\end{align}
\end{lem}

\subsection{On the differentiability of Lipschitz functions}
In the following let~$u\in \Lip\msP_2$,~$w\in\Vect^\infty$ and set
\begin{align}\label{eq:Omega}
\Omega^u_w\eqdef\set{\mu\in\msP : \exists\, G_w u(\mu)\eqdef \diff_t\restr_{t=0} (u\circ\Fl^{w,t})(\mu)} \fstop
\end{align}

Since the function~$u\circ\Fl^{w,t}$ is continuous, the existence of~$G_w u$ coincides with that of the limit $\lim_{r\rar 0} \tfrac{1}{r}(u(\fl^r_\pfwd \mu)-u(\mu))$, $r\in \Q$. As a consequence the set~$\Omega^u_w$ is measurable.

\begin{prop}\label{p:RoeSch99}
Fix~$u\in\Lip\msP_2$ and for any~$w\in\Vect^\infty$ let~$\Omega^u_w$ be defined as in~\eqref{eq:Omega}. Let further~$\VectCount\subset \Vect^\infty$ be a countable $\Q$-vector space dense in~$\Vect^0$ and assume~$\mbbP\, \Omega^u_w=1$ for all~$w\in\VectCount$.
Then, the assertions~\iref{i:t:main3.1} and~\iref{i:t:main3.2} in Theorem~\ref{t:main} hold for~$u$.

\begin{proof} Fix~$w\in \VectCount$. By assumption on~$\VectCount$, there exists
\begin{align*}
G_w u(\mu)= \lim_{t\rar 0} \frac{u(\Fl^{w,t}\mu)-u(\mu)}{t}
\end{align*}
for all~$\mu$ in the set~$\Omega^u_w$ of full $\mbbP$-measure. Moreover, by~\eqref{eq:l:DirLip0},
\begin{align*}
\sup_{t\in[-1,1]} \abs{\frac{u(\Fl^{w,t} \mu)-u(\mu)}{t}}\leq \sup_{t\in[-1,1]} \frac{\Lip[u]}{t}\int_0^t  \norm{w}_{\Vect_{\Fl^{w,r} \mu}} \diff r\leq \Lip[u]\norm{w}_{\Vect^0} \comm
\end{align*}
thus, by Dominated Convergence,
\begin{align}\label{eq:p:DirGradOm1}
G_w u= L^2_{\mbbP}\textrm{-}\lim_{t\rar 0} \frac{u\circ\Fl^{w,t}-u}{t} \fstop
\end{align}

By continuity of~$t\mapsto \tfrac{1}{t}(u\circ \fl^{w,t}-u)$, combining Lemma~\ref{l:DirG} with~\eqref{eq:p:DirGradOm1} yields
\begin{align*}
\int G_wu \cdot v \diff \mbbP=\int u \cdot \grad_w^* v  \diff\mbbP\comm \qquad v\in\Test^\infty \fstop
\end{align*}

Next, notice that the map~$w\mapsto \grad_w^* v$ is linear for all~$v\in\Test^\infty$ by Assumption~\ref{ass:P3}. Hence, if~$w=s_1 w_{1}+\cdots+ s_k w_k$ for some~$s_i\in \R$ and~$w_i\in \VectCount$, then
\begin{align*}
\int G_wu\cdot v \diff\mbbP =\sum_i^k s_i \int u \cdot \grad_{w_i}^* v \diff\mbbP =\sum_i^ks_i \int G_{w_i} u\cdot v \diff\mbbP \comm
\end{align*}
thus
\begin{align}\label{eq:p:DirGradOm2}
G_wu=\sum_i^k s_i \, G_{w_i}u \quad \mbbP\textrm{-a.e.}\fstop
\end{align}

Since~$\VectCount$ is countable, the set~$\Omega^u_0\eqdef \bigcap_{w\in\VectCount} \Omega^u_w$ has full $\mbbP$-measure by assumption. Therefore, the set~$\Omega^u$ of measures~$\mu\in\Omega^u_0$ such that~$w\mapsto G_wu(\mu)$ is a $\Q$-linear functional on~$\VectCount$ has itself full $\mbbP$-measure by~\eqref{eq:p:DirGradOm2}.

For fixed~$\mu\in\Omega^u$ we have~$\abs{G_wu(\mu)}\leq \slo{u}(\mu)\norm{w}_{\Vect_\mu}$ for every~$w\in\VectCount$ by Lemma~\ref{l:DirLip}. Since~$\VectCount$ is~$\Vect^0$-dense in~$\Vect^\infty$, it is in particular~$\Vect_\mu$-dense in~$\Vect^\infty$ for every~$\mu\in\msP$. Hence the map~$w\mapsto G_wu(\mu)$ is a $\Vect_\mu$-continuous linear functional on the dense subset~$\VectCount$ and may thus be extended on the whole space~$\Vect^\infty$ (in fact: on $\Vect_\mu$) to a continuous linear functional, again denoted by~$w\mapsto G_w u(\mu)$ and again such that~$\abs{G_w u(\mu)}\leq \slo{u}(\mu)\norm{w}_{\Vect_\mu}$.

Thus, for every~$\mu$ in the set of full $\mbbP$-measure~$\Omega^u$ there exists~$\mbfD u(\mu)\in T_\mu\msP_2$ such that $G_wu(\mu)=\scalar{\mbfD u(\mu)}{w}_{\Vect_\mu}$ and~$\norm{\mbfD u(\mu)}_{\Vect_\mu}\leq \slo{u}(\mu)$. This concludes the proof of the first statement in~\iref{i:t:main3.2}, which in turn implies~\iref{i:t:main3.1} since~$\slo{u}(\emparg)\leq \Lip[u]$.

\smallskip

By definition of~$\Omega^u$ one has~$\Omega^u\subset \Omega^u_w$ for all~$w\in\VectCount$, hence~\iref{i:t:main3.2} is already proven for all~$w\in\VectCount$. In order to prove it for~$w\in\Vect^\infty\setminus\VectCount$, fix~$\eps>0$ and let~$w'\in\VectCount$ be such that~$\norm{w-w'}_{\Vect^0}<\eps$. 
Since~$\MFD$ is compact, a straightforward modification of~\cite[Lem.~5.5]{RoeSch99} yields
\begin{align*}
\abs{u(\Fl^{w,t} \mu)-u(\Fl^{w',t} \mu)}\leq \Lip[u] \, W_2(\Fl^{w,t} \mu,\Fl^{w',t} \mu)\leq t\, \Lip[u]\, c_0 \,e^{c_0\, t}\, \eps
\end{align*}
for some constant~$c_0\eqdef c_0(\MFD,w)<\infty$. As a consequence,
\begin{align*}
\abs{\frac{u\circ\Fl^{w,t}-u}{t}-\scalar{\mbfD u(\mu)}{w}_{\Vect_\mu}}\leq& \eps \,\Lip[u] \, c_0\, e^{c_0\, t} +\eps \norm{\mbfD u(\mu)}_{\Vect_\mu}\\
&+\abs{\frac{u\circ\Fl^{w',t}-u}{t}-\scalar{\mbfD u(\mu)}{w'}_\mu} \comm \qquad \mu\in\Omega^u\comm
\end{align*}
and letting~$t\rar 0$ yields the conclusion of~\iref{i:t:main3.2} by arbitrariness of~$\eps$.

\smallskip

As consequence of~\iref{i:t:main3.2} and the bound~$\norm{\mbfD u(\mu)}_{\Vect_\mu}\leq\Lip[u]$, by definition,~$u\in\mcF_\cont$.
\end{proof}
\end{prop}

\begin{cor}\label{c:RoeSch99}
Assume~$\mbbP$ additionally satisfies Assumption~\ref{ass:P4} and let~$u\in\Lip\msP_2$. Then, the assertions~\iref{i:t:main3.1} and~\iref{i:t:main3.2} in Theorem~\ref{t:main} hold for~$u$.
\begin{proof}
Let~$w\in\Vect^\infty$ and denote its flow by~$\seq{\fl^{w,t}}_{t\in\R}$. It suffices to show that~$u$ satisfies the assumption on~$\Omega^u_w$ in Proposition~\ref{p:RoeSch99}. By Lemma~\ref{l:DirLip} the set~$\tset{r\in[s,t] : \Fl^{w,r} \mu\in\Omega^u_w}$ has full Lebesgue measure for every~$s<t$ in~$\R$ and every~$\mu\in\msP$. Thus
\begin{align*}
0=&\int_0^1\!\!\int \car_{(\Omega^u_w)^\complement}(\Fl^{w,r} \mu)\diff\mbbP(\mu) \diff r =\int_0^1  (\Fl^{w,r}_\pfwd\mbbP)\ttonde{(\Omega^u_w)^\complement} \diff r= \int_0^1 \!\! \int R^w_r(\mu) \car_{(\Omega^u_w)^\complement}(\mu) \diff\mbbP(\mu) \diff r\comm
\end{align*}
whence~$\mbbP\ttonde{(\Omega^u_w)^\complement}=0$ by~\eqref{eq:ass:P4}.
\end{proof}
\end{cor}

\begin{cor}\label{c:UnderAssB}
Let~$(\MFD,\mssg)$ be additionally satisfying Assumption~\ref{ass:B}. Then, for every~$\nu\in\msP$ the function~$u_{\nu}\colon \mu\mapsto W_2(\nu,\mu)$ belongs to~$\msF_0$ and~$\norm{\mbfD u_\nu}_{\Vect_{\emparg}}\leq 1$ $\mbbP$-a.e..

\begin{proof} Assume first~$\nu\in\msP^\reg$ and set~$S_\theta(\nu)\eqdef \set{\mu\in\msP : u_\nu(\mu)=\theta}$. Since~$\mbbP$ is a probability measure, there exists a sequence~$\theta_n\rar 0$ as~$n\rar \infty$ such that~$\mbbP\, S_{\theta_n}(\nu)=0$. As a consequence of this fact and of Lemma~\ref{l:DerDist}, Proposition~\ref{p:RoeSch99} applies to the map~$u_{\nu,\theta_n}\colon \mu\mapsto W_2(\nu,\mu)\vee \theta_n$ with~$\Omega^{u_{\nu,\theta_n}}\eqdef \msP\setminus S_{\theta_n}(\nu)$, yielding~$\norm{\mbfD u_{\nu,\theta_n}}_{\Vect_{\emparg}}\leq \Lip[u_{\nu,\theta_n}]=1$ $\mbbP$-a.e..
On the other hand,~$u_{\nu,\theta_n}\in\msF_0$ by Lemma~\ref{l:Rademacher} and it is clear by reverse triangle inequality that~$\nlim u_{\nu,\theta_n}=u_\nu$ uniformly, whence~$u_\nu\in\msF_0$ by Lemma~\ref{l:Ma}.

If~$\nu\in\msP\setminus\msP^{\reg}$, choose~$\nu_n\in\msP^\reg$ narrowly convergent to~$\nu$. Again by reverse triangle inequality~$\lim_n u_{\nu_n}=u_\nu$ uniformly and~$\norm{\mbfD u_{\nu_n}}_{\Vect_\emparg}\leq 1$ $\mbbP$-a.e.~as above, hence the conclusion again by Lemma~\ref{l:Ma}.
\end{proof}
\end{cor}

\subsection{Proof of Theorem~\ref{t:main}}

\begin{proof}[Proof of~\iref{i:t:main1} and~\iref{i:t:main2}]
The proof of~\cite[Prop.~1.4(i) and~(iv)]{RoeSch99}, together with the auxiliary results~\cite[Lem.s~6.3,~6.4]{RoeSch99}, carries over \emph{verbatim} to our case. This proves the closability of the forms in assertion~\iref{i:t:main1} and assertion~\iref{i:t:main2}. Since~$\msF_0\subset \msF_\cont\subset \msF$, it suffices to prove the strong locality of~$(\mcE,\msF)$. That is, by~\cite[Rmk.~I.5.1.5]{BouHir91} it suffices to show that if~$u\in\mcF$, then~$\varrho_1\circ u,\varrho_2\circ u\in\msF$ and~$\mcE(\varrho_1\circ u,\varrho_2\circ u)=0$ for~$\varrho_1,\varrho_2\in \mcC^\infty_c(\R)$ such that~$\varrho_1(0)=\varrho_2(0)=0$ and~$\supp \varrho_1\cap \supp\varrho_2=\emp$.

Fix~$w\in\Vect^\infty$ and denote by~$\seq{\fl^{w,t}}_{t\in\R}$ its flow. Since~$u\in\mcF$ is bounded, the map~$U\colon t\mapsto u\circ\Fl^{w,t}$ satisfies~$U(t)\in L^2_\mbbP(\msP)$ for every~$t\in\R$, hence,~\cite[Lem.~6.4]{RoeSch99} yields for~$i=1,2$
\begin{align*}
\diff_t\restr_{t=0} \, \varrho_i(U(t))=\varrho_i'(U(0))\, \diff_t\restr_{t=0} U(t)=(\varrho'_i\circ u) \scalar{\mbfD u}{w}_{\Vect_{\emparg}}
\end{align*}
where all derivatives are taken in~$L^2_\mbbP(\msP)$. Hence, the map~$\mu\mapsto \varrho_i'(u(\mu))\mbfD u(\mu)$ is a measurable section of~$T^\Der\msP_2$, satisfies~\eqref{eq:DefF} and is such that
\begin{align}\label{eq:t:main1}
\mcE(\varrho_i\circ u, \varrho_i\circ u)=\int \varrho_i'(u(\mu)) \norm{\mbfD u(\mu)}^2_{\Vect_\mu} \diff\mbbP(\mu) \leq\norm{\varrho_i'}_{\mcC^0}^2 \mcE(u,u)<\infty \fstop
\end{align}

As a consequence,~$\varrho_i\circ u\in\mcF$ and the locality property follows now by~\eqref{eq:t:main1} and polarization.
\end{proof}

\begin{proof}[Proof of~\iref{i:t:main3}] For fixed~$\nu\in\msP^\reg$ let~$u_\nu\colon \msP\rar \R$ be defined by~$u_\nu\colon\mu\mapsto W_2(\nu,\mu)$. By Lemma~\ref{l:DerDist}, for every~$\mu\in\Omega^\nu\eqdef \msP\setminus \set{\nu}$ and every~$w\in\Vect^\infty$ there exists the limit~$G_wu_\nu(\mu)$ defined in~\eqref{eq:Omega}. Since~$\mbbP$ is atomless by Assumption~\ref{ass:P1}, the set~$\Omega^\nu$ has full $\mbbP$-measure, hence Proposition~\ref{p:RoeSch99} applies to~$u_\nu$ with~$\Omega^{u_\nu}=\Omega^\nu$ and one has~$\norm{\mbfD u_\nu(\mu)}_{\Vect_\mu}\leq \Lip[u_\nu]=1$.

Since additionally~$u_\nu\in\mcF_\cont$ by Proposition~\ref{p:RoeSch99}, if~$u$ is $W_2$-Lipschitz continuous, then~$u\in\msF_\cont$ and~$\norm{\mbfD u}_{\Vect_{\emparg}}\leq \Lip[u]$ $\mbbP$-a.e.~by strong locality of~$(\mcE,\msF)$ and Lemma~\ref{l:Koskela} applied to the dense set~$\msP^\reg$, which proves~\iref{i:t:main3.1}.
If~$\MFD$ additionally satisfies Assumption~\ref{ass:B}, then we may replace~$\msF_\cont$ in the above reasoning with~$\msF_0$ thanks to Corollary~\ref{c:UnderAssB}.

If~$\mbbP$ additionally satisfies Assumption~\ref{ass:P4}, then assertion~\iref{i:t:main3.2} reduces to Corollary~\ref{c:RoeSch99}.
\end{proof}

\paragraph{Intrinsic distances} Given a family of functions~$\msA\subset \msF$ set, for all~$\mu,\nu\in\msP$,
\begin{align*}
\mssd_\msA(\mu,\nu)\eqdef \sup\set{u(\mu)-u(\nu) : u\in\msA\cap \mcC(\msP), \boldGamma(u)\leq 1\; \mbbP\textrm{-a.e.~on~$\msP$}} \fstop
\end{align*}

\begin{cor}[Intrinsic distances]
Suppose that~$\mbbP$ satisfies Assumption~\ref{ass:P} and let
\begin{align*}
\mssd_{\msF_0}\leq \mssd_{\msF_\cont}\leq \mssd_{\msF}
\end{align*}
be the intrinsic distances~\eqref{eq:IntMet} of the Dirichlet forms~$(\mcE,\msF_0)$, $(\mcE,\msF_\cont)$ and~$(\mcE,\msF)$ respectively. Then,
\begin{align*}
\mssd_{\Test^\infty}\leq W_2\leq \mssd_{\msF_\cont} \fstop
\end{align*}

If additionally Assumption~\ref{ass:B} holds, then the above statement holds with~$\mssd_{\msF_0}$ in place of~$\mssd_{\msF_\cont}$.

\begin{proof}
Let~$\msA=\msF_0,\msF_\cont,\msF$. If~$u_\nu\in\msA$ then
\begin{align*}
\mssd_\msA(\mu,\nu)\geq u_\nu(\mu)-u_\nu(\nu)=W_2(\mu,\nu) \comm
\end{align*}
hence it suffices to keep track of the assumptions under which~$u_\nu\in\msF_0,\msF_\cont,\msF$ respectively in order to show~$W_2\leq \mssd_\msA$. One has~$u_\nu\in\msF_\cont\subset \msF$ by the proof of Theorem~\iref{i:t:main3} above, while~$u_\nu\in\msF_0$ under Assumption~\ref{ass:B} by Corollary~\ref{c:UnderAssB}.

Let now~$u\in\Test^\infty$ with~$\norm{\mbfD u}_{\Vect_\emparg}\leq 1$ $\mbbP$-a.e.. Since~$\mbfD u=\grad u$ is continuous, if~$\mbbP$ has full support by Assumption~\ref{ass:P0}, then~$\norm{\mbfD u(\mu)}_{\Vect_\mu}\leq 1$ for all~$\mu\in\msP$. In the same notation of Lemma~\ref{l:DistrDer}, it follows from~\eqref{eq:l:DistrDer0} that
\begin{align*}
u(\mu_1)-u(\mu_0)=\int_0^1 \scalar{\grad u(\mu_s)}{w_s}_{\Vect_{\mu_s}} \diff s \leq \int_0^1 \norm{w_s}_{\Vect_{\mu_s}} \diff s\fstop
\end{align*}

Taking the infimum of the above inequality over all distributional solutions~$(\mu_s,w_s)_{s\in I}$ of~\eqref{eq:ContEq} with fixed~$\mu_0,\mu_1$ yields~$u(\mu_1)-u(\mu_0)\leq W_2(\mu_0,\mu_1)$ by e.g.~\cite[Prop.~2.30]{AmbGig11}.

This settles all the inequalities in the assertion.
\end{proof}
\end{cor}

\section{Examples}\label{s:Examples}
We collect here some examples of measures satisfying our main Theorem~\ref{t:main}. These include the family of normalized mixed Poisson measures~\S\ref{ss:NMPM} (for any~$\MFD$), the entropic measure~\S\ref{ss:Entropic} and an image on~$\msP_2$ of the Malliavin--Shavgulidze measure~\S\ref{ss:MSI} (both in the case~$\MFD=\mbbS^1$). We notice that a proof of Theorem~\ref{t:main} for the entropic measure was already sketched in~\cite[Prop.~7.26]{vReStu09}.
Again when~$\MFD$ is arbitrary, we also provide an example of a measure \emph{not} satisfying Assumption~\ref{ass:P}, namely the Dirichlet--Ferguson measure~\S\ref{s:DF}. However, relying on results in the present work, we show in~\cite{LzDS17+} that the assertions \iref{i:t:main1}--\iref{i:t:main3.1} in Theorem~\ref{t:main} hold for this measure too.
Finally, we show how to construct more examples from those listed above, by considering shifted measures, weighted measures and convex combinations.

\paragraph{Notation} Everywhere in this section let~$\upphi\in\Diffeo^\infty(\MFD)$ and denote by $\Phi\colon \Mbp\rar \Mbp$ the shift by~$\upphi$, by~$\upphi^*\colon L^0(\MFD)\rar L^0(\MFD)$ the pullback by~$\upphi$, and by~$J_\upphi^\mssm$ the modulus of the Jacobian determinant of~$\upphi$ with respect to~$\mssm$.

Denote further by~$N\colon \Mbp\rar \msP$ the normalization map~$N\colon \nu\mapsto \n\nu\eqdef \nu/\nu \MFD$. It is straightforward that~$N$ is continuous with respect to the chosen topologies, hence measurable with respect to the chosen $\sigma$-algebras.
Moreover, it is readily verified that~$N$ and~$\Phi\eqdef\upphi_\pfwd$ commute, i.e.
\begin{align}\label{eq:ese:PP3}
N\circ \Phi=\Phi \circ N\colon \Mbp\rar \msP\fstop
\end{align}

\subsection{On Assumption~\texorpdfstring{\ref{ass:P}}{(P)}}\label{ss:AssP}
We collect here some comments on Assumption~\ref{ass:P}. First of all, let us show how to construct examples of measures satisfying Assumption~\ref{ass:P} starting from a single one.
\begin{lem}\label{l:GradPfwd}
Let~$w\in\Vect^\infty$ and~$u\in\Test^\infty$. Then,
\begin{align*}
\grad_w (u\circ \Phi)=\grad_{\upphi_* w} u\circ \Phi \fstop
\end{align*}

\begin{proof} Let~$f\in\mcC^\infty(\MFD)$. Then
\begin{align*}
\grad (f^\bid\circ\Phi)=\grad ((f\circ\upphi)^\bid)=\nabla (f\circ \upphi) \fstop
\end{align*}

By~\eqref{eq:DirDer0}, the proof reduces now to the following computation
\begin{align*}
\scalar{\grad (u\circ \Phi)(\mu)}{w}_{\Vect_\mu}=&\sum_i^k (\partial_i F)\ttonde{f_1^\bid(\Phi\mu),\dotsc, f_k^\bid(\Phi\mu)} \int_\MFD \diff (f\circ \upphi)_x(w_x) \diff\mu(x)
\\
=& \sum_i^k (\partial_i F)\ttonde{f_1^\bid(\Phi\mu),\dotsc, f_k^\bid(\Phi\mu)} \int_\MFD \diff f_{\upphi(x)}(\diff \upphi_x w_x) \diff\mu(x)
\\
=& \sum_i^k (\partial_i F)\ttonde{f_1^\bid(\Phi\mu),\dotsc, f_k^\bid(\Phi\mu)} \int_\MFD \diff f_{y}(\diff \upphi_{\upphi^{-1}(y)} w_{\upphi^{-1}(y)}) \diff \Phi \mu(y)
\\
=&\scalar{\grad u(\Phi \mu)}{\upphi_* w}_{\Vect_{\Phi \mu}} \fstop & \qedhere
\end{align*}
\end{proof}
\end{lem}

\begin{prop}\label{p:Pfwd}
Let~$\mbbP\in\msP(\msP)$,~$\upphi\in \Diffeo^\infty(\MFD)$ and~$\phi\in\msF$ be such that~$\phi>0$ $\mbbP$-a.e.~and~$\norm{\phi}_{L^2_\mbbP}=1$. Set~$\mbbP'\eqdef \Phi_\pfwd \mbbP$ and~$\mbbP^\phi\eqdef \phi^2\cdot \mbbP$. Then,
\begin{enumerate}[label=\ensuremath{(\roman*)}]
\item\label{i:p:Pfwd0} if~$\mbbP$ satisfies Assumption~\ref{ass:P0}, then so do~$\mbbP'$ and~$\mbbP^\phi$;
\item\label{i:p:Pfwd1} if~$\mbbP$ satisfies Assumption~\ref{ass:P1}, then so do~$\mbbP'$ and~$\mbbP^\phi$;
\item\label{i:p:Pfwd2} if~$\mbbP$ satisfies Assumption~\ref{ass:P3}, then so do~$\mbbP'$ and~$\mbbP^\phi$;
\item\label{i:p:Pfwd3} if~$\mbbP$ satisfies Assumption~\ref{ass:P4}, then so does~$\mbbP^\phi$. If additionally~$\upphi=\fl^{w,t}$ for some~$w\in\Vect^\infty$,~$t\in\R$, then, additionally,~$\mbbP'$ satisfies Assumption~\ref{ass:P4} too.
\end{enumerate}
\begin{proof}
Since~$\upphi$ is bijective, so are~$\Phi\eqdef\upphi_\pfwd$ and~$\Phi_\pfwd$. This proves~\iref{i:p:Pfwd0} and~\iref{i:p:Pfwd1} for~$\mbbP'$; they are also straightforward for~$\mbbP^\phi$ since~$\phi^2>0$ $\mbbP$-a.e.. In both cases,~\iref{i:p:Pfwd3} is straightforward by~\eqref{eq:ass:P4}.

In order to show~\iref{i:p:Pfwd2} for~$\mbbP'$, we need to show that there exists an operator~$\grad_w^{*'}\colon \Test^\infty\rar L^2_{\mbbP'}(\msP)$ such that~\eqref{eq:DefDirForm} holds with~$\mbbP'$ in place of~$\mbbP$ and~$\grad_w^{*'}$ in place of~$\grad_w^*$.
Since~$\upphi$ is a diffeomorphism, the notations~$\upphi_*^{-1}$ and~${\upphi^{-1}}_\pfwd={\upphi_\pfwd}^{-1}=\Phi^{-1}$ are unambiguous. Then,~by Lemma~\ref{l:GradPfwd},
\begin{align*}
\int \grad_w u \cdot v \diff\mbbP' =&\int \grad_w u\circ\Phi\cdot v\circ \Phi \diff\mbbP=\int \grad_{\upphi_*^{-1} w}(u\circ \Phi) \cdot v\circ\Phi \diff\mbbP
\\
=&\int u\circ\Phi \cdot \grad_{\upphi_*^{-1}w}^*(v\circ \Phi) \diff\mbbP=\int u \cdot \grad_{\upphi_*^{-1}w}^*(v\circ \Phi)\circ\Phi^{-1} \diff\mbbP' \fstop
\end{align*}

Assertion~\iref{i:p:Pfwd2} follows by putting~$\grad_w^{*'}v\eqdef \grad_{\upphi_*^{-1}w}^*(v\circ \Phi)\circ\Phi^{-1}$.

In order to show~\iref{i:p:Pfwd2} for~$\mbbP^\phi$ assume first that~$\phi\in\Test^\infty$, whence~$\phi$ is continuous and bounded by Rmk.~\ref{r:ContTest}. Then, by~\eqref{eq:grad0} and~\eqref{eq:DirDer0}
\begin{align*}
\int u \phi^2 \cdot \grad^*_w v \diff\mbbP=\int \grad_w (u\phi^2)\cdot v \diff\mbbP=&\int \grad_w u \cdot  \phi^2 v \diff\mbbP+\int \phi^2 uv \cdot (2\phi^{-1}\grad_w \phi ) \diff\mbbP
\end{align*}
and the assertion follows by setting~$\grad^{*,\phi}_w v\eqdef \grad^*v-(2\phi^{-1}\grad_w \phi ) v$.
The general case follows by approximation as soon as we show that the pre-Dirichlet form
\begin{align*}
\msF^\phi\eqdef& \set{u\in \msF : \int \ttonde{u^2+\norm{\mbfD u}_{\Vect_\emparg}^2} \phi^2\diff \mbbP <\infty}
\\
\mcE^\phi(u,v)\eqdef& \int \phi^2\scalar{\mbfD u}{\mbfD v}_{\Vect_\emparg} \diff\mbbP
\end{align*}
is closable. Provided that~$(\mcE,\msF)$ is a strongly local Dirichlet form by Theorem~\iref{i:t:main1}, this last assertion is the content of~\cite[Thm.~1.1]{Ebe96}.
\end{proof}
\end{prop}

\begin{rem}\label{r:Milnor} While points~\iref{i:p:Pfwd0}--\iref{i:p:Pfwd2} of the Proposition suggest that Assumptions~\ref{ass:P0}--\ref{ass:P3} are quite generic with respect to shifting~$\mbbP$ by (the lift of) a diffeomorphism, point~\iref{i:p:Pfwd3} is by far more restrictive, as the inclusion~$\Flow(\MFD)\subsetneq \DiffSP(\MFD)$ is always strict, even on~$\mbbS^1$, see e.g.~\cite{Gra88}.
\end{rem}

It is clear that the closability of the pre-Dirichlet forms~$(\mcE,\Test^\infty)$ and~$(\mcE,\mcF_\cont)$ associated to~$\mbbP$ is essential to our approach in discussing Rademacher-type theorems, which settles the necessity of Assumption~\ref{ass:P3}. 
Assumption~\ref{ass:P0} is instead motivated by the following trivial example.

\begin{ese}\label{ese:Trivial}
Denote by~$\delta\colon \MFD\mapsto \msP$ the Dirac embedding~$x\mapsto \delta_x$ and set~$\mbbP\eqdef \delta_\pfwd \mssm$. Since~$\mbbP$ is supported on the family of Dirac masses, it does not satisfy~\ref{ass:P0}. On the other hand, since~$W_2(\delta_{x_1},\delta_{x_2})=\mssd_\mssg(x_1,x_2)$ for every~$x_1,x_2\in \MFD$, it is also clear that~$(\msP,W_2,\mbbP)$ and~$(\MFD,\mssd_\mssg,\mssm)$ are isomorphic as metric measure spaces,  which shows~\ref{ass:P1}.
Moroever,
\begin{align*}
\Phi_\pfwd \delta_\pfwd \mssm=(\Phi \circ\delta)_\pfwd\mssm=(\delta\circ \upphi)_\pfwd \mssm=\delta_\pfwd(\Phi \mssm)=\delta_\pfwd (J^\mssm_\upphi \cdot \mssm)=(J^\mssm_\upphi)^\bid \cdot \delta_\pfwd \mssm
\end{align*}
and~\ref{ass:P3} holds for~$\mbbP$ as well.
\end{ese}

\begin{rem} Incidentally, notice that Theorem~\ref{t:main} applied to Example~\ref{ese:Trivial} provides a non-local proof of the classical Rademacher Theorem on a closed Riemannian manifold. Indeed it suffices to notice that~$T^\Der_{\delta_x}\cong T_x\MFD$ as Hilbert spaces for every~$x\in \MFD$ and that every Lipschitz function~$f\in\Lip(\MFD)$ induces a Lipschitz function~$\tilde f\in \Lip(\msP_2)$, namely the (e.g., lower) McShane extension~$\tilde f$ of the function~$f\circ\delta^{-1}$ defined on the image of~$\delta$. 
\end{rem}

\begin{prop}\label{p:P5}
If~$\mbbP$ satisfies Assumption~\ref{ass:P5}, then it satisfies Assumption~\ref{ass:P}.

\begin{proof}
It suffices to show that
\begin{align*}
\text{Ass.~}\ref{ass:P5} \Rar \ref{ass:P4} \text{ and } \ref{ass:P3} \Rar \ref{ass:P4} \Rar \ref{ass:P1} \fstop
\end{align*}

The implication Assumption~$\ref{ass:P5}\Rar\ref{ass:P4}$ is by definition and it is readily seen that Assumption~\ref{ass:P1} is already implied by the first part of~\ref{ass:P4}.
It remains to show that Assumption~$\ref{ass:P5}\Rar\ref{ass:P3}$. Indeed,
\begin{align*}
\int \grad_w u \cdot v \diff\mbbP=&\int \lim_{t\rar 0} \frac{u\circ\Fl^{w,t}-u}{t} \cdot v \diff\mbbP
\\
=& \lim_{t\rar 0} \frac{1}{t}\int \ttonde{u\cdot v\circ\Fl^{w,-t}\cdot R^{w}_{-t} -u v} \diff\mbbP
\\
=& \lim_{t\rar 0} \frac{1}{t}\int u \ttonde{v\circ\Fl^{w,-t}-v} \diff\mbbP\\
&+\lim_{t\rar 0} \frac{1}{t}\int u \tonde{v\circ \Fl^{w,-t}-v} \tonde{R^{w}_{-t}-1} \diff\mbbP\\
&+\lim_{t\rar 0} \frac{1}{t}\int u v \tonde{R^{w}_{-t}-1} \diff\mbbP
\fstop
\end{align*}

The first limit in the last equality satisfies, by Dominated Convergence,
\begin{align*}
\lim_{t\rar 0} \frac{1}{t}\int u \ttonde{v\circ\Fl^{w,-t}-v} \diff\mbbP =& -\int u\cdot \grad_w v \diff\mbbP \fstop
\end{align*}

The second limit vanishes, again by Dominated Convergence, since~$t\mapsto R^w_t(\mu)$ is continuous (differentiable) at~$t=0$ for $\mbbP$-a.e.~$\mu$. In light of Assumption~\ref{ass:P5}, differentiating under integral sign, the third limit satisfies
\begin{align*}
\lim_{t\rar 0} \frac{1}{t}\int u v \tonde{R^{w}_{-t}-1} \diff\mbbP=\int uv\cdot \partial_t\restr_{t=0} R^w_{-t} \diff\mbbP \fstop
\end{align*}

As a consequence, Assumption~\ref{ass:P3} is satisfied by letting
\begin{align*}
\grad_w^* v\eqdef&-\grad_w v- \partial_t\restr_{t=0}R^w_{-t}\cdot v \fstop
\end{align*}

This concludes the proof.
\end{proof}
\end{prop}

\subsection{On the Smooth Transport Property}\label{ss:AssB}
The reader is referred to~\cite{Fig10} and references therein for an expository treatment of  regularity theory of optimal transport maps on Riemannian manifolds, whereof we make use in the present section.
We denote by~$\mbbS T_x\MFD\eqdef\{\mrmw\in T_x\MFD : \abs{\mrmw}_{\mssg_x}=1\}$ the unit tangent space to~$(\MFD,\mssg)$ at~$x$.
Everywhere in the following also let~$\mssc\eqdef\tfrac{1}{2}\mssd^2$.

\paragraph{Further geometrical assumptions}
For~$x\in \MFD$ and~$\mrmw\in T_x\MFD$ define the \emph{cut}, resp.~\emph{focal}, \emph{time} by
\begin{align*}
t_\mathsc{c}(x,\mrmw)\eqdef& \inf\set{t>0 : s\mapsto\exp_x(s\mrmw) \textrm{ is not a $\mssd$-minimizing curve from $x$ to $\exp_x(t\mrmw)$}}\\
t_\mathsc{f}(x,\mrmw)\eqdef& \inf\set{t>0 : \diff_{t\mrmw}\exp_x\colon T_x\MFD\rar T_{\exp_x(t\mrmw)}\MFD \textrm{ is not invertible}}
\end{align*}
and the (\emph{tangent}), resp.~(\emph{tangent}) \emph{focal}, \emph{cut locus} and \emph{injectivity domain} by 
\begin{align*}
TCL(x)\eqdef& \set{t_\mathsc{c}(x,\mrmw)\mrmw : \mrmw\in \mbbS T_x\MFD}\comm & \cut(x)\eqdef& \exp_x(TCL(x))\comm \\
TFL(x)\eqdef& \set{t_\mathsc{f}(x,\mrmw)\mrmw : \mrmw\in  \mbbS T_x\MFD}\comm & \fcut(x)\eqdef& \exp_x(TCL(x)\cap TFL(x))\comm \\
I(x)\eqdef& \set{t\mrmw : \mrmw\in \mbbS T_x\MFD, 0\leq t < t_\mathsc{c}(x,\mrmw)} \fstop
\end{align*}

Finally, let~$x\in \MFD$, $y\in I(x)$, $\mrmw,\mrmw'\in T_x\MFD$ and~$\mrmv\eqdef \exp_{x}^{-1}(y)$, and recall the definition of the Ma--Trudinger--Wang tensor
\begin{align*}
\mfS_{(x,y)}(\mrmw,\mrmw')\eqdef 
-\tfrac{3}{2}\diff^2_s\restr_{s=0}\diff^2_t\restr_{t=0}\, \mssc\!\tonde{\exp_{x}(t\mrmw),\exp_{x}(\mrmv+s\mrmw')} \fstop
\end{align*}

The following definitions are taken from~\cite{Fig10}.
\begin{defs}[Non-focality of cut loci] We say that~$(\MFD,\mssg)$ is \emph{non-focal} if it additionally satisfies~$\fcut(x)=\emp$ for all~$x\in \MFD$.
\end{defs}

\begin{defs}[Strong Ma--Trudinger--Wang condition $\mathrm{MTW}(K)$]
We say that~$(\MFD,\mssg)$ satisfies the \emph{strong Ma--Trudinger--Wang condition with constant~$K$} (in short:~$\MFD$ is~$\mathrm{MTW}(K)$) if there exists a constant~$K>0$ such that
\begin{align*}
\mfS_{(x,y)}(\mrmw,\mrmw')\geq K\tabs{\mrmw}_{\mssg_x}^2\tabs{\mrmw'}_{\mssg_x}^2 \comm \quad x\in \MFD\comm y\in \exp_x(I(x))\comm \qquad\textrm{whenever} \quad \mrmw^\tra [\mssc_{\emparg,\emparg}] \mrmw'=0\comm
\end{align*}
where~$[\mssc_{\emparg,\emparg}]$ denotes the matrix of derivatives~$\mssc_{i,j}\eqdef \partial^2_{x_i, y_j}\mssc$.
\end{defs}

Our main interest in the previous definitions is the following regularity result.
\begin{thm}[Loeper--Villani, e.g.,~{\cite[Cor.~3.13]{Fig10}}]\label{t:LoepVil} Let~$(\MFD,\mssg)$ be additionally non-focal and satisfying~$\mathrm{MTW}(K)$. Then~$\MFD$ satisfies Assumption~\ref{ass:B}.
\end{thm}

\begin{rem} The strong $\mathrm{MTW}$ condition is sufficient, whereas not necessary, to establish the above result. A discussion of optimal assumptions is here beyond our purposes. It will suffice to say that the proof strategy of Lemma~\ref{l:Rademacher} fails as soon as~$\mathrm{MTW}(0)$ is negated, which in turn implies that~$\mssc$-convex $\mcC^1$ functions are \emph{not} uniformly dense in (Lipschitz) $\mssc$-convex functions by~\cite[Thm.~3.4]{Fig10}.
\end{rem}

\subsection{Normalized mixed Poisson measures}\label{ss:NMPM}
We denote by~$\ddot\Gamma$ the \emph{space of integer-valued Radon measures} over $(\MFD,\mssg)$ with arbitrary \emph{finite} number of atoms, always regarded as a subspace of~$\Mbp$, endowed with the vague topology, which coincides with the narrow topology by compactness of~$\MFD$, and with the associated Borel $\sigma$-algebra. Similarly to~\cite{AlbKonRoe98,RoeSch99}, we let~$\rho\in \mcC^1(\MFD;\R^+)$ and denote by~$\PP_\sigma$ the Poisson measure of intensity~$\sigma\eqdef \rho\mssm$ on~$\ddot\Gamma$. Given~$\lambda\in\msP(\R^+)$ such that~$\lambda(1\wedge\id_{\R^+})<\infty$, henceforth a \emph{L\'evy measure}, we denote by~$\RP_{\lambda,\sigma}$ the \emph{mixed Poisson measure}
\begin{align*}
\RP_{\lambda,\sigma}(\emparg)=\int_{\R^+} \PP_{s\cdot\sigma}(\emparg)\diff\lambda(s) \fstop
\end{align*}

Recall that~$\PP_\sigma$, hence~$\RP_{\lambda,\sigma}$, is concentrated on the \emph{configuration space}
\begin{align*}
\Gamma\eqdef \set{\gamma\in \ddot\Gamma : \gamma\set{x}\in\set{0,1} \textrm{ for all } x\in \MFD}\fstop
\end{align*}

Moreover, by~\cite[Prop.~2.2]{AlbKonRoe98},
\begin{align}\label{eq:ese:PP1}
\frac{\diff \ttonde{\Phi_\pfwd\PP_{\sigma}}}{\diff \PP_\sigma}(\gamma)=\exp\ttonde{\sigma(\car-p^\sigma_\upphi)} \prod_{x\in \gamma} p^\sigma_\upphi(x) \comm \quad \gamma\in \Gamma\comm \qquad \textrm{where}\qquad p^\sigma_{\upphi}(x)\eqdef \frac{\upphi^*\rho(x)}{\rho(x)}J_\upphi^\mssm(x)
\end{align}
and by~$x\in\gamma$ we mean~$\gamma\set{x}>0$. Since we chose~$\rho\in L^1_\mssm(\MFD)$, the measure~$\sigma$ is finite, hence~$\gamma \MFD<\infty$ for~$\PP_\sigma$-a.e.~$\gamma$, i.e.~$\PP_\sigma$-a.e.~$\gamma$ is concentrated on a finite number of points. As a consequence, the same statement holds for~$\RP_{\lambda,\sigma}$ in place of~$\PP_\sigma$ and one has
\begin{align}\label{eq:ese:PP2}
\forallae{\RP_{\sigma,\lambda}}\gamma\qquad R^\sigma_\upphi(\gamma)\eqdef\prod_{x\in \gamma} p^\sigma_\upphi(x)= \exp \int_\MFD \ln\ttonde{ p^\sigma_\upphi(x)} \diff\gamma(x) \fstop
\end{align}

\begin{ese}[Normalized mixed Poisson measures]\label{ese:NPP}
Let~$\lambda\in\msP(\R^+)$ be a L\'evy measure with \emph{compact} support and set~$\mbbP\eqdef N_\pfwd \RP_{\lambda,\sigma}$. 
Since~$\sigma$ is atomless, so are~$\PP_\sigma$,~$\RP_{\lambda,\sigma}$ and~$\mbbP$. Thus, Assumption~\ref{ass:P1} is satisfied.
Assumptions~\ref{ass:P4} and~\ref{ass:P3} are respectively verified in Lemmas~\ref{l:NPPassQI} and~\ref{l:NPPassCl} below. In particular, the closability of the pre-Dirichlet form in~\eqref{eq:Form} is obtained as a consequence of the quasi-invariance of~$\mbbP$. Assumption~\ref{ass:P0} is verified in Lemma~\ref{l:NPPassFS} below.
\end{ese}

Denote now by~$\MFD^{\odot n}\eqdef \MFD^{\times n}/\mfS_n$ the quotient of the $n$-fold cartesian product~$\MFD^{\times n}$ by the symmetric group~$\mfS_n$ acting by permutation of coordinates. Let further~$\MFD^{\times n}_\circ$ denote the set of points~$\mbfx\eqdef \seq{x_1,\dotsc, x_n}$ in~$\MFD^{\times n}$ such that~$x_i\neq x_j$ for~$i\neq j$, and set
\begin{align*}
\MFD^{(n)}\eqdef \MFD^{\times n}_\circ/\mfS_n\fstop
\end{align*}

Denote by $\pr^{\mfS_n}\colon \MFD^{\times n}_\circ\rar \MFD^{(n)}$ the quotient projection, and set~$\sigma^{(n)}\eqdef \pr^{\mfS_n}_\pfwd \sigma^{\otimes n}$.
It is well-known that, when~$(\MFD,\sigma)$ is a finite Radon measure space, then~$(\Gamma,\PP_\sigma)$ is isomorphic, as a measure space, to the space
\begin{align}\label{eq:IsoPoisson}
\bigoplus_{n\in \N_1} \ttonde{ \MFD^{(n)},e^{-\sigma \MFD}\sigma^{(n)}/n! }\fstop
\end{align}
More explicitly, the isomorphism is given by identifying~$\MFD^{(n)}$ with~$\Gamma^{(n)}$, the space of configurations~$\gamma\in\Gamma$ such that~$\gamma \MFD=n$.
Finally, define the following subsets of~$\msP$
\begin{equation}\label{eq:ConcentrationSs}
\begin{aligned}
N\ttonde{\Gamma^{(n)}}\subset& \mathring\Delta^n\eqdef \set{\sum_i^n s_i\delta_{x_i} : \mbfx\in \MFD^{(n)},s_i\in\R^+}\subset\mathring\Delta^\fin\eqdef \bigcup_n \mathring\Delta^n\comm
\\
N\ttonde{\ddot\Gamma^{(n)}}\subset& \Delta^n\eqdef \set{\sum_i^n s_i\delta_{x_i} : \mbfx\in \MFD^{\times n}, s_i\in \R^+}\subset \Delta^\fin\eqdef \bigcup_n \Delta^n\fstop
\end{aligned}
\end{equation}

\begin{rem}\label{r:Small}
While the support~$\mathring\Delta^1=\Delta^1\cong \MFD$ of the measure constructed in Example~\ref{ese:Trivial} is ``small'' in various senses ---~e.g., it is a closed nowhere dense subset of~$\msP$~---, the normalized (mixed) Poisson measures in Example~\ref{ese:NPP} have full support.
On the other hand though, even these measures are concentrated on~$\mathring\Delta^\fin$, which may itself be still regarded as ``small'' ---~e.g., since the measure space~$(\Delta^\fin,N_\pfwd \RP_{\sigma,\lambda})$ may be approximated in many senses via the sequence of compact \emph{finite-dimensional} measure spaces $(\Delta^n,N_\pfwd \RP_{\sigma,\lambda}\restr_{\Delta^n})$.
\end{rem}

\subsection{The Dirichlet--Ferguson measure}\label{s:DF}
Example~\ref{ese:NPP} shows that the laws of (normalized) point processes on~$\MFD$ may be examples of measures on~$\msP$ satisfying Assumption~\ref{ass:P}. In light of Remark~\ref{r:Small}, the question arises, whether such laws may be chosen to be concentrated on sets richer than~$\mathring\Delta^\fin$, and in particular on the whole set of purely atomic measures.

In this section we introduce for further purposes a negative example, the \emph{Dirichlet--Ferguson measure} over~$\MFD$ (see below), satisfying Assumptions~\ref{ass:P0}--\ref{ass:P1} and the closability of the form~$(\mcE,\dom\mcE)$, whereas not~\ref{ass:P3} nor~\ref{ass:P4}. These properties are verified in~\cite{LzDS17+}, basing on the characterization of the measure in Theorem~\ref{t:Mecke} below.

\paragraph{Notation}
Everywhere in the following, denote by~$\n\mssm$ the \emph{normalized} volume measure of~$\MFD$, let~$\beta\in(0,\infty)$ be defined by~$\mssm=\beta\n\mssm$, and let
\begin{align*}
\diff\Beta_\beta(r)\eqdef \beta (1-r)^{\beta-1}\diff r
\end{align*}
be the Beta distribution on~$I$ with parameters~$1$ and~$\beta$. Set further~$\hat \MFD\eqdef \MFD\times I$, always endowed with the product topology, $\sigma$-algebra and with the measure~$\hat\mssm_\beta\eqdef \n\mssm\otimes \Beta_\beta$.

Finally, we denote by~$\eta$ any \emph{purely atomic} measure in~$\msP$. Usually, we think of any such~$\eta$ as an infinite marked configuration and thus we write, with slight abuse of notation,~$\eta_x$ in place of~$\eta\!\set{x}$ and~$x\in\eta$ whenever~$\eta_x>0$. We denote further by~$\ptws\eta$ the set of points~$x\in \MFD$ such that~$\eta_x>0$, called the \emph{pointwise support of~$\eta$}.

\paragraph{The Dirichlet--Ferguson measure}
We denote by~$\DF_\mssm$ the Dirichlet--Ferguson measure~\cite{Fer73} over $(\MFD,\mcB)$ with intensity~$\mssm$. The measure is also known as: \emph{Dirichlet}, \emph{Poisson--Dirichlet}~\cite{TsiVerYor01}, (law of the) \emph{Fleming--Viot process with parent-independent mutation}~\cite{OveRoeSch95}.
The characteristic functional of~$\DF_\mssm$ may be found in~\cite{LzDS17} together with further properties of the measure. The following characterization is originally found, in the form of a distributional equation, in~\cite[Eqn.~(3.2)]{Set94}.

\begin{thm}[Mecke-type identity for $\DF_\mssm$ {\cite{Set94}, see also~\cite{LzDSLyt17}}]\label{t:Mecke}
Let $u\colon\msP\times \hat \MFD\rar \R$ be measurable semi-bounded. Then, there exists a unique probability measure~$\DF_\mssm$ on~$\msP$ satisfying
\begin{align}\label{eq:MeckeDFCor}
\iint_\MFD u(\eta, x, \eta_x) \diff\eta(x) \diff\DF_\mssm(\eta) =\iint_{\hat \MFD} u\ttonde{(1-r)\eta + r\delta_x, x, r}\diff\hat \mssm_\beta (x,r) \diff\DF_\mssm(\eta) \fstop
\end{align}
\end{thm}

The unacquainted reader may take this result as a definition of~$\DF_\mssm$.

\subsection{The entropic measure}\label{ss:Entropic}
In this section we recall an example showing that there exist measures on~$\msP$ ---~other than normalized mixed Poisson measures~--- satisfying Assumption~\ref{ass:P}.

\paragraph{Preliminaries} Similarly to~\cite[\S{2.2}]{vReStu09}, define
\begin{align*}
\msG(\R)\eqdef\set{g\colon \R\rar \R, \textrm{ right-continuous non-decreasing, s.t.} \quad g(x+1)=g(x)+1\comm \quad x\in \R } 
\end{align*}

In light of the equi-variance property, each~$g\in\msG(\R)$ uniquely induces a Borel function $\pr^\msG(g)\colon \mbbS^1\rar\mbbS^1$ and we set~$\msG\eqdef \pr^\msG (\msG(\R))$, endowed with the $L^2$-distance
\begin{align*}
\norm{g_1-g_2}_{\msG}\eqdef \tonde{\int_{\mbbS^1} \abs{g_1(t)-g_2(t)}^2 \diff\mssm(t)}^{1/2} \fstop
\end{align*}

Letting~$\mbbS^1\cong\R/\Z$, define further for every~$a\in\mbbS^1$ the \emph{translation}~$\tau_a\colon \mbbS^1\rar \mbbS^1$ by
\begin{align*}
\tau_a\colon t\mapsto t+a \pmod 1 \comm
\end{align*}
and define an equivalence relation~$\sim$ on~$\msG$ by setting~$g\sim h$ for~$g,h\in\msG$ if and only if~$g=h\circ\tau_a$ for some~$a\in\mbbS^1$.
Denote by~$\pr^{\msG_1}$ the quotient map of~$\msG$ modulo this equivalence relation, with values in the quotient space~$\msG_1\eqdef\pr^{\msG_1}(\msG)=\msG/\mbbS^1$ endowed with the quotient $L^2$-distance
\begin{align*}
\norm{g_1-g_2}_{\msG_1}\eqdef \tonde{\inf_{s\in\mbbS^1}\int_{\mbbS^1} \abs{g_1(s)-g_2(t+s)}^2 \diff\mssm(t)}^{1/2} \fstop
\end{align*}
Equivalently,~$\msG_1$ is the semi-group of right-continuous non-decreasing functions on~$\mbbS^1\cong [0,1)$ fixing~$0\in\mbbS^1$.
Finally, the space~$\ttonde{\msG_1,\norm{\emparg}_{\msG_1}}$ is \emph{isometric}~\cite[Prop.~2.2]{vReStu09} to $\msP_2\eqdef\ttonde{\msP_2(\mbbS^1),W_2}$ via the map
\begin{align}\label{eq:Chi}
\chi\colon g\mapsto g_\pfwd \mssm \fstop
\end{align}

\paragraph{The conjugation map~$\mfC^{\n\mssm}$~{\cite[\S3]{Stu11}}} For~$\mu\in\msP$ let~$\phi_\mu\eqdef \phi_{\n\mssm\rar\mu}$ be given by Theorem~\ref{t:McCann} (recall that~$\n\mssm\in\msP^\reg$). The \emph{conjugation} map~$\mfC^{\n\mssm}\colon\msP\rar\msP$ is defined by
\begin{align*}
\mfC^{\n\mssm}\colon \mu\mapsto \ttonde{\exp\nabla (\phi_\mu^\mssc)}_\pfwd\n\mssm \fstop
\end{align*}

It was shown in~\cite[Thm.~3.6]{Stu11} that~$\mfC^{\n\mssm}$ is an involutive homeomorphism of~$\msP_2$.
If~$\MFD=\mbbS^1$, then the conjugation map may be alternatively defined in the following equivalent way. Let
\begin{align*}
g_\mu(t)\eqdef \inf\set{s\in I : \mu[0,s]>t}\comm \qquad \inf \emp\eqdef1\fstop
\end{align*}
denote the cumulative distribution function of~$\mu\in\msP(\mbbS^1)$. Observe that~$g_\mu\in\msG_1$, hence it admits a left inverse~$g_\mu^{-1}$ in~$\msG_1$ given by
\begin{align*}
g_\mu^{-1}(t)\eqdef \inf\set{s\geq 0 : g(s)>t} \fstop
\end{align*}

Then,~$\mfC^{\n\mssm}(\mu)=\diff g_\mu^{-1}$ where, for any~$g\in\msG_1$, we denoted by~$\diff g$ the Lebesgue--Stieltjes measure associated to~$\phi$. See~\cite{vReStu09} for the detailed construction.

\begin{defs}[entropic measure over~$\MFD$~{\cite[Dfn.~6.1]{Stu11}}]
The \emph{entropic measure}~$\mbbP_\mssm$ is the Borel probability measure on~$\msP_2$ defined by~$\mbbP_\mssm\eqdef \mfC^{\n\mssm}_\pfwd \DF_\mssm$, where~$\DF_\mssm$ is the Dirichlet--Ferguson measure of~\S\ref{s:DF}.
\end{defs}

Since~$\mfC^{\n\mssm}$ is a homeomorphism,~$\mbbP_\mssm$ satisfies Assumptions~\ref{ass:P0},~\ref{ass:P1} because so does~$\DF_\mssm$. The quasi-invariance of~$\mbbP_\mssm$ as in Assumption~\ref{ass:P4} and Assumption~\ref{ass:P3}, and the closability of the Dirichlet form~\eqref{eq:Form} are a challenging problem. They have been proven in the seminal work~\cite{vReStu09} for the case~$\MFD=\mbbS^1$, which leads us to the following example.

\begin{ese}[The entropic measure over~{$\mbbS^1$}~{\cite[Dfn.~3.3]{vReStu09}}]\label{ese:EntMeas}
Let~$\beta>0$ be a fixed constant and let~$\MFD=\mbbS^1$ be endowed with the rescaled volume measure~$\mssm\eqdef \beta \Leb^1$. The quasi-invariance of~$\mbbP_\mssm$ ---~as in Assumption~\ref{ass:P4}~--- was proven in~\cite[Cor.~4.2]{vReStu09}. In fact, it was proven for the action of the whole of~$\Diffeo^2(\MFD)$ rather than only for~$\Flow(\MFD)$, cf.~Rmk.~\ref{r:Milnor}. Although not apparent, the bound~\eqref{eq:ass:P4} for the Radon--Nikod\'ym derivative~$R^w_r$ may be deduced from the explicit computations in~\cite[Lem.~4.8]{vReStu09}. 
In fact, Assumption~\ref{ass:P5} holds too, because of~\cite[Lem.~5.1(ii)]{vReStu09}.
Together with the previous discussion, this shows that~$\mbbP_\mssm$ satisfies Assumption~\ref{ass:P} by Proposition~\ref{p:P5}.

The closability of the form~$(\mcE,\msF_0)$ is proven in~\cite[Thm.~7.25]{vReStu09}, where the family of cylinder functions~$\Test^\infty$ is introduced in~\cite[Dfn.~7.24]{vReStu09} and denoted by~$\mfZ^\infty(\msP)$. A proof of the Rademacher property in the form of our Theorem~\iref{i:t:main3.1} is sketched in~\cite[Prop.~7.26]{vReStu09}.
\end{ese}

\begin{rem} Finally, let us notice that~$\mbbP_\mssm$-a.e.~$\mu$ is concentrated on an~$\mssm$-negligible set~\cite[Cor.~3.11]{vReStu09}. In fact, it is not difficult to show that~$\mbbP_\mssm$-a.e.~$\mu$ is concentrated on the set of irrational points of a Cantor space, i.e.~any (non-empty) totally disconnected perfect metrizable compact space.
\end{rem}

\subsection{An image on~\texorpdfstring{$\msP$}{P} of the Malliavin--Shavgulidze measure}\label{ss:MSI}

As a final example, we introduce here an image on~$\msP(\mbbS^1)$ of the Malliavin--Shavgulidze measure on~$\Diffeo^1_+(\mbbS^1)$.

\paragraph{Preliminaries} We refer the reader to~\cite{Kuz07} for a detailed exposition and further references. Let~$\MFD=\mbbS^1$ with volume measure~$\mssm\eqdef \Leb^1$ and denote by~$\mcC_0(I^\circ)$ the space of continuous functions on~$I$ vanishing at both~$0$ and~$1$, endowed with the trace topology of~$\mcC(I)$. Consider the space~$\Diffeo^1_+(\mbbS^1)$ of orientation preserving $\mcC^1$-diffeomorphisms of~$\mbbS^1$, endowed with the topology of uniform convergence, and let~$\xi\colon \Diffeo^1_+(\mbbS^1)\rar \mbbS^1\times \mcC_0(I^\circ)$ be the homeomorphism defined by
\begin{align*}
\xi\colon g(t)\mapsto \ttonde{g(0), \ln g'(t)-\ln g'(0)} \fstop
\end{align*}

\begin{defs}[The Malliavin--Shavgulidze measure]
Let~$\BB$ be the Borel probability on~$\mcC(I)$ defined as the law of the Brownian Bridge connecting~$0$ to~$0$ in time~$1$, concentrated on~$\mcC_0(I^\circ)$. The \emph{Malliavin--Shavgulidze measure}~$\mcM$ on $\Diffeo^1_+(\mbbS^1)$~\cite{MalMal90} is the Borel probability measure defined by~$\mcM\eqdef (\xi^{-1})_\pfwd (\mssm\otimes \BB)$.
\end{defs}

Denote further by~$S$ the \emph{Schwarzian derivative} operator
\begin{align*}
S\colon \upphi\mapsto \frac{\upphi'''}{\upphi'}-\frac{3}{2} \tonde{\frac{\upphi''}{\upphi'}}^2 \comm
\end{align*}
and consider the left action~$L_\upphi \colon g\mapsto \upphi\circ g$ of the subgroup~$\Diffeo^3_+(\mbbS^1)$. The measure~$\MS$ is quasi-invariant with respect to~$L_\upphi$ and the following quasi-invariance formula~\cite[Thm.~3.1, p.~235]{MalMal90} holds true for every Borel~$A\subset\Diffeo^1_+(\mbbS^1)$,
\begin{align}\label{eq:MSqi}
\MS(L_\upphi(A))=\int_A \exp\quadre{\int_{\mbbS^1} S(\upphi)(g(t)) \cdot g'(t)^2 \diff \mssm(t)} \diff \MS(g)\fstop
\end{align}

\paragraph{The Malliavin--Shavgulidze image measure} 
Every $\mcC^1$-function in~$\msG$ is a $\mcC^1$-diffeomorphism of~$\mbbS^1$, ori\-en\-ta\-tion-preserving since induced by a non-decreasing function, and every such diffeomorphism arises in this way. Furthermore,~$\Diffeo^1_+(\mbbS^1)$ embeds continuously into~$\msG$. It follows that~$\MS$ may be regarded as a (non-relabeled) measure on~$\msG$.

\begin{ese}[The Malliavin--Shavgulidze image measure]\label{ese:MSI}
Consider the Borel probability measure~$\MS$ on~$\msG$. 
The measure~$\MS_1\eqdef \pr^{\msG_1}_\pfwd \MS$ is a well-defined Borel probability measure on~$\msG_1$ by measurability (continuity) of~$\pr^{\msG_1}$.
The \emph{Malliavin--Shavgulidze image measure}~$\MSI$ is the Borel probability measure on~$\msP$ defined by
\begin{align*}
\MSI\eqdef\chi_\pfwd (\pr^{\msG_1}_\pfwd \MS)\fstop
\end{align*}
 
Assumption~\ref{ass:P0} for~$\MSI$ is readily verified from the properties of the Malliavian--Shavgulidze measure~$\MS$. In fact,~$\MSI$ is concentrated on the set
\begin{align*}
\ttonde{\Diffeo^1_+(\mbbS^1)/\Isom_+(\mbbS^1)}_\pfwd\mssm\subset \msP^\mssm(\mbbS^1)\comm
\end{align*}
which also shows Assumption~\ref{ass:P1}.
Assumption~\ref{ass:P5} is verified in Lemma~\ref{l:MSIassQI} below.
\end{ese}

\begin{rem}\label{r:Support}
Examples~\ref{ese:NPP},~\ref{ese:EntMeas} and~\ref{ese:MSI} clarify that Assumption~\ref{ass:P} poses no restriction to the subset of~$\msP$ whereon~$\mbbP$ is concentrated. Indeed, as argued above
\begin{itemize}
\item $N_\pfwd \RP_{\lambda,\sigma}$-a.e.~$\mu\in\msP_2(\mbbS^1)$ is \emph{purely atomic};
\item $\mbbP_\mssm$-a.e.~$\mu\in\msP_2(\mbbS^1)$ is \emph{singular continuous} w.r.t.~the volume measure of~$\mbbS^1$;
\item $\MSI$-a.e.~$\mu\in\msP_2(\mbbS^1)$ is \emph{absolutely continuous} w.r.t.~the volume measure of~$\mbbS^1$.
\end{itemize}

Furthermore, it is readily seen that, if~$\mbbP$ and~$\mbbP'$ both satisfy Assumption~\ref{ass:P}, then so does any convex combination thereof. Thus, it is possible to construct a measure~$\mbbP$ on~$\msP_2(\mbbS^1)$ such that~$\mbbP$-a.e.~$\mu$ has Lebesgue decomposition consisting of both a singular, a singular continuous and an absolutely continuous part.
\end{rem}

\section{Appendix}\label{s:Appendix}

\subsection{On the notion of tangent bundle to~\texorpdfstring{$\msP_2$}{P2}}\label{ss:Tangent}
The concept of \emph{tangent space} to~$\msP_2$ at a point~$\mu$ or \emph{space of directions} through~$\mu$ has been widely investigated. See~\cite{AmbGigSav08, Gig11,Gig12, GanKimPac10} and, especially, the bibliographical notes~\cite[\S6.4]{AmbGig11}. At least the following three different notions are available
\begin{itemize}
\item the \emph{tangent space} $T^\nabla_\mu\msP_2\eqdef \cl_{\Vect_\mu} \Vect^\infty_\nabla$;
\item the \emph{geometric tangent space}, denoted here by~$\mbfT_\mu\msP_2$, defined in~\cite[Dfn.~5.4]{Gig11};
\item the \emph{pseudo-tangent space}, denoted here by~$T^\Der_\mu\msP_2\eqdef \Vect_\mu$, considered as auxiliary space in~\cite{GanKimPac10, ChoGan17}.
\end{itemize}

It was proven in~\cite[Prop.s~6.1, 6.3]{Gig11} 
that~$T^\nabla_\mu\msP_2\cong\mbfT_\mu\msP_2$ if and only if~$\mu\in\msP^\reg$; if otherwise, then~$T^\nabla_\mu\msP_2$ embeds canonically non-surjectively in~$\mbfT_\mu\msP_2$ and the latter is not a Hilbert space.
The relation between~$T^\nabla_\mu\msP_2$ and~$T^\Der_\mu\msP_2$ is made explicit in the following.

\paragraph{Preliminaries}
By a \emph{Fr\'echet space} we mean a locally convex completely metrizable topological vector space.
In this section, we endow~$\mcC^\infty(\MFD)$ with its usual Fr\'echet topology~$\tau_{\mcC^\infty(\MFD)}$ and denote by~$\mcC^\infty(\MFD)^*$ the topological dual~$\ttonde{\mcC^\infty(\MFD),\tau_{\mcC^\infty(\MFD)}}^*$ endowed with the weak* topology, e.g.,~\cite[\S1.9]{Tre66}.
Analogously, we endow~$\Test^\infty$ with the locally convex metrizable linear topology~$\tau_{\Test^\infty}$ induced by the countable family of semi-norms
\begin{align*}
\abs{u}_k\eqdef \sup_{w_1,\dotsc, w_k\in\Vect^\infty} \norm{\tonde{\grad_{w_1}\circ \cdots \circ \grad_{w_k}} u}_{\mcC(\msP_2)}\comm \qquad k\in\N_0\comm
\end{align*}
where it is understood that~$\abs{u}_0$ is but the uniform norm on~$\mcC(\msP_2)$.
We denote by~${\Test^\infty}^*$ the topological dual of~$\tonde{\Test^\infty,\tau_{\Test^\infty}}$, endowed with the weak* topology.

\paragraph{Divergence operator, cf.~{\cite[\S2.3]{GanKimPac10}}}
The \emph{divergence operator}~$\div_\mu\colon \Vect^\infty\rar \mcC^\infty(\MFD)^*$ mapping~
\begin{align*}
w\mapsto \tonde{\scalar{\div_\mu w}{\emparg}\colon f\mapsto -\int_\MFD (\diff f(w))(x) \diff\mu(x) }
\end{align*}
satisfies~
\begin{align}\label{eq:ContDiv}
\scalar{\div_\mu w}{f}\leq \norm{\nabla f}_{\Vect_\mu}\norm{w}_{\Vect_\mu}\comm
\end{align}
hence it extends by continuity to a (non-relabeled) operator~$\div_\mu\colon T^\Der_\mu\msP_2\rar \mcC^\infty(\MFD)^*$ and one has that,~\cite[Rmk.~2.7]{GanKimPac10},
\begin{align}\label{eq:DecoTanSp}
T^\Der_\mu\msP_2=T^\nabla_\mu\msP_2 \oplus \ker\div_\mu\comm
\end{align}
where the symbol~$\oplus$ denotes the orthogonal direct sum of Hilbert spaces.

On the one hand, it is clear that, if~$\mu\in\msP^{\infty,\times}$, then~$\ker\div_\mu$ is non-trivial as soon as~$\Vect^\infty\neq\Vect^\infty_\nabla$. This holds in particular if~$(\MFD,\mssg)$ has non-trivial de Rham cohomology group~$H_{\textrm{dR}}^1(\MFD;\R)$. On the other hand, if~$\eta\in\msP$ has finite support, then,~\cite[Example~2.8]{GanKimPac10},
\begin{align}\label{eq:TSpDirac}
T^\nabla_\eta\msP_2=T^\Der_\eta\msP_2=\bigoplus_{x\in\ptws \eta} \ttonde{T_x\MFD, \eta_x\cdot\mssg_x}\fstop
\end{align}

\paragraph{Local derivations} Motivated by the definition, for finite-dimensional differentiable manifolds, of \emph{space of derivatives at a point} (or \emph{pointwise derivations}, e.g.,~\cite[Cor.~2.2.22]{Con01}), we define for fixed~$\mu\in\msP$ the linear functional~$\del^\mu_w\colon \Test^\infty\rar\R$ by
\begin{align*}
\del^\mu_w\colon u\mapsto \scalar{\grad u(\mu)}{w}_{\Vect_\mu} \fstop
\end{align*}
Letting~$\ev_\mu\colon \Test^\infty\rar\R$ be defined by~$\ev_\mu(u)\eqdef u(\mu)$, it is readily verified by Lemma~\ref{l:Derivations} below, that~$\del^\mu_w$ satisfies Leibniz rule in the form
\begin{align}\label{eq:pLeibniz}
\partial^\mu_w(u v)=\ev_\mu(v)\, \partial^\mu_w u + \ev_\mu(u)\,\partial^\mu_w v\fstop
\end{align}

We denote by~$\Der(\Test^\infty)_\mu\subset {\Test^\infty}^*$ the space of continuous linear functionals on~$\Test^\infty$ satisfying~\eqref{eq:pLeibniz}, endowed with the trace topology.
Since~${\Test^\infty}^*$ is Hausdorff and complete, the (uniformly) continuous linear operator~$\partial^\mu_\emparg \colon\Vect^\infty\rar {\Test^\infty}^*$ extends to a uniquely-defined (non-relabeled) operator~$\del^\mu_\emparg\colon \Vect_\mu\rar {\Test^\infty}^*$ by~\cite[\S{I.5}, Thm.~5.1, p.~39]{Tre67}.
Moreover,
\begin{align}\label{eq:ContDel}
\del^\mu_w(u)\leq \norm{\grad u(\mu)}_{\Vect_\mu}\norm{w}_{\Vect_\mu}\comm
\end{align}
hence one has in fact~$\partial^\mu_w\in \Der(\Test^\infty)_\mu$ for every~$w\in\Vect_\mu$.

\begin{prop}\label{p:IdentKer} Denote by~$\mfj\colon f\mapsto f^\bid$ the canonical injection~$\mcC(\MFD)\rar \mcC(\MFD)^{**}$.

Then,~$\div_\mu(\emparg)=-\partial^\mu_\emparg \circ \mfj\colon \Vect_\mu\rar\mcC^\infty(\MFD)^{*}$ and~$\ker\div_\mu \cong \ker\del^\mu_\emparg\subset \Vect_\mu$ as Hilbert spaces.
\begin{proof}
For any~$f\in\mcC^\infty(\MFD)$ and~$w\in\Vect^\infty$ it holds that
\begin{align}\label{eq:DerDiv}
(\del^\mu_w\circ \mfj)(f)=\del^\mu_w(f^\bid)=\int_\MFD \gscal{\nabla f_x}{w_x} \diff \mu(x)=\mu(\diff f(w))=-\scalar{\div_\mu w}{f} \comm
\end{align}
that is~$\div_\mu(\emparg)(\empargop)=-\del^\mu_\emparg (\mfj(\empargop))$ on~$\Vect^\infty\otimes \mcC^\infty(\MFD)$. By~\eqref{eq:ContDel} applied to~$u=f^\bid\in\Test^\infty$, the operator~$\del^\mu_\emparg\circ\mfj\colon \Vect^\infty\rar \mcC^\infty(\MFD)^*$ may be extended to a uniquely defined (non-relabeled) operator $\del^\mu_\emparg\circ\mfj\colon \Vect_\mu\rar\mcC^\infty(\MFD)^*$ and the notation is consistent in the sense that this operator coincides with the previously defined extension of~$\del^\mu_\emparg$ applied to~$\mfj$. Since both~$\del^\mu_\emparg\circ \mfj$ and~$-\div_\mu(\emparg)$ are linear and $\norm{\emparg}_{\Vect_\mu}$-continuous and coincide on the dense set~$\Vect^\infty\subset \Vect_\mu$, they coincide on the whole space~$\Vect_\mu$.
It remains to show that~$\ker\del^\mu_\emparg=\ker\del^\mu_\emparg\circ \mfj$, which follows immediately by noticing that for any~$u= F\circ (f_1^\bid,\dotsc, f_k^\bid)\in\Test^\infty$ and~$w\in\Vect_\mu$
\begin{align*}
\del^\mu_w(u)=&\sum_i^k(\partial_i F)(f_1^\bid\mu,\dotsc, f_k^\bid\mu) \int_\MFD \gscal{\nabla_x f_i}{w_x} \diff\mu(x)=-\sum_i^k(\partial_iF)(f_1^\bid\mu,\dotsc, f_k^\bid\mu) \scalar{\div_\mu w}{f_i}
\\
=&\sum_i^k(\partial_i F)(f_1^\bid\mu,\dotsc, f_k^\bid\mu) \, (\del^\mu_w\circ \mfj)(f_i) \fstop
\end{align*}

This concludes the proof.
\end{proof}
\end{prop}

\paragraph{Tangent bundles} Let us denote by~$T^\nabla \msP_2$ the \emph{tangent bundle} to~$\msP_2$, set-wise defined as the disjoint union of~$T^\nabla_\mu\msP_2$ varying~$\mu\in\msP_2$. The \emph{pseudo-tangent bundle}~$T^\Der\msP_2$ is analogously defined.
Whereas this terminology is well-established, it is clear that~$T^\nabla\msP_2$ is \emph{not} a vector \emph{bundle} in the standard sense ---~nor in any reasonable sense~---, since it admits \emph{no} local trivialization by reasons of the dimension of~$T^\nabla_\mu\msP_2$. Indeed, for any~$x_0\in \MFD$ and every~$\eps>0$ one can find a smooth function~$\rho_\eps\in\mcC^\infty(\MFD)$ such that~$\mu_\eps\eqdef\rho_\eps\n\mssm\in\msP^\mssm$ and~$W_2(\delta_{x_0},\mu_\eps)<\eps$, yet~$T^\nabla_{\delta_{x_0}}\msP_2\cong T_{x_0}\MFD$ while~$T^\nabla_\mu\msP_2$ is infinite-dimensional. The same is true for~$T^\Der\msP_2$.

Despite this fact, the gradient~$\grad u$ of a cylinder function $u\in\Test^\infty$ may well be regarded as a smooth section of~$T^\nabla\msP_2$ since~$\grad u(\mu)\in T^\nabla_\mu\msP_2$ by~\eqref{eq:grad0}. Again by~\eqref{eq:grad0} the space of all such gradients is a subspace of the space~$\Test^\infty\otimes_{\R} \Vect^\infty_\nabla$ of $\Test^\infty$-linear combinations of gradient-type vector fields. This motivates the Definition~\ref{d:VectCinfty} of \emph{cylinder vector fields}~$\TestV^\infty\eqdef \Test^\infty\otimes_{\R} \Vect^\infty$, henceforth regarded ---~in analogy to the case of finite-dimensional manifolds~--- as (a subspace of) the space of smooth sections of the tangent bundle~$T^\Der \msP_2$.
In spite of Proposition~\ref{p:IdentKer}, the fiber-bundle~$T^\Der\msP_2$ does in fact convey more information than the fiber-bundle~$T^\nabla\msP_2$.

\paragraph{Global derivations} Consider the space $\Der(\Test^\infty)$ of abstract $\R$-derivations of~$\Test^\infty$.

\begin{lem}\label{l:Derivations}
Let~$w\in\Vect^\infty$. Then, the map
\begin{align*}
\del_w\colon u\mapsto \diff_t\restr_{t=0}(u\circ\Fl^{w,t})=\scalar{\grad u}{w}_{\Vect_{\emparg}}
\end{align*}
is an element of~$\Der(\Test^\infty)$.
\end{lem}
\begin{proof} One has
\begin{align}
\nonumber
\diff_t\restr_{t=0}u(\Fl^{w,t}\mu)=&\sum_i^k (\partial_i F)\ttonde{f_1^\bid \fl^{w,0}_\pfwd \mu,\dotsc, f_k^\bid \fl^{w,0}_\pfwd \mu} \times \diff_t\restr_{t=0} f_i^\bid\Fl^{w,t} \mu
\\
\nonumber
=&\sum_i^k (\partial_i F)(f_1^\bid\mu,\dotsc, f_k^\bid\mu) \times \diff_t\restr_{t=0} \mu (f_i\circ\fl^{w,t})
\\
\nonumber
=&\sum_i^k (\partial_i F)(f_1^\bid\mu,\dotsc, f_k^\bid\mu) \times \mu \ttonde{\diff_t\restr_{t=0}(f_i\circ\fl^{w,t})}
\\
\nonumber
=&\sum_i^k (\partial_i F)(f_1^\bid\mu,\dotsc, f_k^\bid\mu) \times \mu\gscal{\nabla f_i}{w}
\\
\label{eq:l:Derivation1}
=&\sum_i^k (\partial_i F)(f_1^\bid\mu,\dotsc, f_k^\bid\mu) \times \gscal{\nabla f_i}{w}^\bid\mu
\\
\label{eq:l:Derivation2}
=&\scalar{\grad u(\mu)}{w}_{\Vect_\mu} \fstop
\end{align}

Since~$\gscal{\nabla f_i}{w}\in\mcC^\infty(\MFD)$ by the choice of~$f_i$ and~$w$, and since~$\Test^\infty$ is an algebra, \eqref{eq:l:Derivation1}~shows that~$\del_w\colon \Test^\infty\rar \Test^\infty$. The Leibniz rule is straightforward from the same property of~$\diff_t$, while~$\Test^\infty$-linearity is a consequence of the representation in~\eqref{eq:l:Derivation2}.
\end{proof}

\begin{prop}\label{p:Deriv}
Let~$W\in\TestV^\infty$ be as in~\eqref{eq:TestV}. Then, the map
\begin{align*}
\del\colon W\mapsto \del_W\eqdef \sum_j^n v_j \del_{w_j} 
\end{align*}
is a linear injection~$\del\colon \TestV^\infty\rar \Der(\Test^\infty)$.
\end{prop}
\begin{proof} The fact that~$\del_W\in\Der(\Test^\infty)$ is a consequence of Lemma~\ref{l:Derivations} and of the choice of the~$v_j$'s. The $\Test^\infty$-linearity is immediate, while the $\Vect^\infty$-linearity follows from~\eqref{eq:l:Derivation2}.

Let now~$W\neq \zero_{\TestV^\infty}$, that is, there exists~$\mu_0\in \msP$ and~$x_0\in \MFD$ such that~$W(\mu_0)(x_0)\neq \zero_{T_{x_0}\MFD}$.
Since~$W(\emparg)(x_0)$ is continuous, and the set of purely atomic finitely supported probability measures is dense in~$\msP_2$ by e.g.,~\cite[proof of Thm.~6.18]{Vil09}, we can find a purely atomic finitely supported~$\eta\in\msP$ such that~$W(\eta)(x_0)\neq \zero_{T_{x_0}\MFD}$. Without loss of generality, up to choosing~$\eta'\eqdef (1-\eps)\eta+ \eps \delta_{x_0}$ for some small~$\eps>0$, we can assume~$\eta_{x_0}>0$.
By standard arguments, there exists~$f\in\mcC^\infty(\MFD)$ such that~$\nabla f_{x_0}=W(\eta)(x_0)$. Moreover, since~$\ptws\eta$ is finite, we can find~$g\in \mcC^\infty(\MFD)$ such that~$g\equiv \car$ on an open neighborhood of~$x_0$ and~$g\equiv \zero$ on an open neighborhood of every point in~$\ptws\eta$ other than~$x_0$. Set~$h=fg$ and notice that~$\nabla h_{x_0}=W(\eta)(x_0)$ while~$\nabla h=\zero$ for every point in~$\ptws\eta$ other than~$x_0$.
Now,
\begin{align*}
\del_W(h^\bid)(\eta)=&\scalar{\grad h^\bid}{W(\eta)}_{\Vect_\eta}=\int_\MFD \gscal{\nabla h_x}{W(\eta)(x)} \diff\eta(x)
= \eta_{x_0}\abs{W(\eta)(x_0)}_{\mssg}^2>0 \fstop
\end{align*}

Since~$\del$ is linear, this shows that it is also injective, which concludes the proof.
\end{proof}

\begin{rem}
Proposition~\ref{p:Deriv} is motivated by the analogy with finite-dimensional compact differential manifolds, e.g.~\cite[Prop.~3.5.3]{Con01}, in which case the map
\begin{align*}
\partial\colon \Vect^\infty\ni w\longmapsto \ttonde{\partial_w\colon f \mapsto \diff f(w)}\in\Der(\mcC^\infty(\MFD))
\end{align*}
is straightforwardly injective, and surjective because of the classical Hadamard Lemma. In the case of~$\msP_2$, I do not know whether~$\del$ is surjective, however, it should be noted that, in the case of infinite-dimensional smooth manifolds, this is not necessarily the case, again already at the pointwise scale. See~\cite[Thm.~28.7]{KriMic97}, for a proof of surjectivity under additional assumptions.
\end{rem}

Throughout all computations in Section~\ref{s:Proof}, vector fields $w\in\Vect^\infty$ ought to be interpreted as smooth directions at every point~$\mu\in\msP$. This is the right notion to be compared with the definition of directional derivative given in~\eqref{eq:DirDer0} in light of Proposition~\ref{p:Deriv}.

\subsection{Auxiliary results on normalized mixed Poisson measures}
\begin{lem}\label{l:NPPassQI}
The measure~$\mbbP$ defined in Example~\ref{ese:NPP} satisfies Assumption~\ref{ass:P4}.
\begin{proof}
Retain the notation in~\S\ref{s:Examples}. By~\eqref{eq:ese:PP3} and combining~\eqref{eq:ese:PP1} with~\eqref{eq:ese:PP2},
\begin{align*}
\diff \ttonde{\Phi_\pfwd \mbbP}(\mu)=& \diff\tonde{\int_{\R^+}N_\pfwd \Phi_\pfwd \PP_{s\cdot \sigma}(\emparg)\diff\lambda(s) }(\mu)
\\
=&\diff\tonde{\int_{\R^+} \exp\ttonde{s\cdot \sigma (\car -p^{s\cdot \sigma}_\upphi)} \cdot N_\pfwd \ttonde{R^{s\cdot \sigma}_\upphi \cdot \PP_{s\cdot\sigma}}(\emparg)  \diff\lambda(s)}(\mu)
\fstop
\end{align*}

Noting further that~$N$ is injective on~$\ddot\Gamma$ and denoting by~$N^{-1}$ its right-inverse, the function~$R^{s\cdot \sigma}_\upphi\circ N^{-1}$ is well-defined on~$\ddot\Gamma$, hence~$\PP_{s\cdot\sigma}$-a.e.~on~$\Gamma$. It follows that
\begin{align*}
\diff \ttonde{\Phi_\pfwd \mbbP}(\mu)=&\int_{\R^+} \tonde{\exp\ttonde{s\cdot \sigma (\car -p^{s\cdot \sigma}_\upphi)} \cdot \ttonde{R^{s\cdot \sigma}_\upphi\circ N^{-1}}(\mu)\cdot \diff\tonde{N_\pfwd\PP_{s\cdot\sigma}}(\mu) } \diff\lambda(s)\fstop
\end{align*}

Moreover,~$p^\sigma_\upphi=p^{s\cdot\sigma}_\upphi$ for every~$s>0$ by definition, cf.~\eqref{eq:ese:PP1}, thus~$R^\sigma_\upphi=R^{s\cdot\sigma}_\upphi$ and
\begin{align}
\nonumber
\diff \ttonde{\Phi_\pfwd \mbbP}(\mu)=&\ttonde{R^{\sigma}_\upphi\circ N^{-1}}(\mu) \int_{\R^+} \tonde{\exp\ttonde{s\cdot \sigma (\car -p^{\sigma}_\upphi)} \cdot \diff\tonde{N_\pfwd\PP_{s\cdot\sigma}}(\mu)} \diff\lambda(s)
\\
\label{eq:esePP3.5}
=&\ttonde{R^{\sigma}_\upphi\circ N^{-1}}(\mu) \cdot \diff N_\pfwd\tonde{\int_{\R^+}  \exp\ttonde{s\cdot \sigma (\car -p^{\sigma}_\upphi)} \cdot \PP_{s\cdot\sigma}(\emparg)  \diff\lambda(s)}(\mu)\comm
\end{align}
where it is possible to pull~$N_\pfwd$ outside the integral sign since the integrand does not depend on~$\mu$.

Finally, letting~$c_{\lambda,\sigma,\upphi}\eqdef (\sup\supp\lambda) \abs{\sigma (\car -p^{\sigma}_\upphi)}$, then for every measurable~$A\subset \msP$
\begin{equation}\label{eq:ese:PP4}
\begin{aligned}
e^{-c_{\lambda,\sigma,\upphi}} \ttonde{R^\sigma_\upphi\circ N^{-1}\cdot \mbbP}A \leq\ttonde{\Phi_\pfwd\mbbP} A \leq e^{c_{\lambda,\sigma,\upphi}} \ttonde{R^\sigma_\upphi\circ N^{-1}\cdot \mbbP}A \fstop
\end{aligned}
\end{equation}
Since~$R^\sigma_\upphi(\gamma)>0$ for~$\RP_{\lambda,\sigma}$-a.e.~$\gamma\in\Gamma$, it follows from~\eqref{eq:ese:PP4} that~$\mbbP$ and~$\Phi_\pfwd \mbbP$ are mutually absolutely continuous, hence the quasi-invariance in~\ref{ass:P4} holds. Letting~$w\in\Vect^\infty$, equation~\eqref{eq:ass:P4} is similarly verified since~$\#\supp\mu<\infty$ for~$\mbbP$-a.e.~$\mu$, hence, for all~$t\in\R$,
\begin{align*}
\ttonde{R^\sigma_{\fl^{w,t}}\circ N^{-1}}(\mu) \geq \prod_{x\in\mu} \frac{(\fl^{w,t})^*\rho(x)}{\rho(x)}J_{\fl^{w,t}}^\mssm(x)\geq \tonde{\min_{\MFD} \frac{(\fl^{w,t})^*\rho}{\rho} J^\mssm_{\fl^{w,t}}}^{\#\supp\mu}&>0 \fstop
\end{align*}

This concludes the proof.
\end{proof}
\end{lem}

\begin{lem}\label{l:NPPassCl}
The measure~$\mbbP$ defined in Example~\ref{ese:NPP} satisfies Assumption~\ref{ass:P3}.

\begin{proof} We show the assertion when~$\lambda=\delta_1$, i.e.~when~$\mbbP=N_\pfwd\PP_\sigma$, similarly to~\cite[Thm.~3.1]{AlbKonRoe98}. The general case is readily proved by integration w.r.t.~$\lambda$ in light of the mutual absolute continuity of~$\PP_{s\cdot \sigma}$ w.r.t.~$\PP_{\sigma}$, hence of their normalizations, for every choice of~$s>0$, hence~$\lambda$-a.e., cf.~\eqref{eq:IsoPoisson}.

\paragraph{Preliminaries} Retain the notation established in~\S\ref{s:Examples} and denote by~$\beta^\sigma\eqdef \nabla\rho/\rho$ the logarithmic derivative of~$\sigma$, which is well-defined on~$\MFD$ since~$\rho\in\mcC^1(\MFD;\R^+)$.
Let further~$w\in\Vect^\infty$, denote by~$\div_\mssm$ the divergence on~$\MFD$ with respect to the volume measure~$\mssm$, and set
\begin{align*}
\beta^\sigma_w(x)\eqdef\scalar{\beta^\sigma_x}{w_x}_{\mssg_x}+\div_\mssm w_x\fstop
\end{align*}

Finally, let~$\nabla_w^*$ be the adjoint of~$\nabla_w$ in~$L^2_\sigma(\MFD)$, and denote the closure of~$\nabla_w$ again by the same symbol. By integration by parts~\cite[Eq.~(3.11)]{AlbKonRoe98}, one can readily show that
\begin{align*}
\nabla_w^*=-\nabla_w-\beta^\sigma_w\fstop
\end{align*}

\paragraph{Claim} \emph{ Letting~$B^\sigma_w\eqdef (\beta^\sigma_w)^\bid$, we claim that~$\grad_w^*\eqdef -\grad_w-B^\sigma_w$
satisfies~\eqref{eq:DefDirForm}  for our choice of~$\mbbP$. }

\paragraph{Some differentiation} For all~$u\in\Test^\infty$ denote by the same symbol the natural extension of~$u$ to~$\Mbp$. By~\cite[Prop.~2.1]{AlbKonRoe98} and several applications of~\eqref{eq:ese:PP3} we have
\begin{align}\label{eq:ese:PP5}
\int u\circ\Fl^{w,t} \cdot v \diff N_\pfwd \PP_\sigma=\int u\cdot v\circ\Fl^{w,-t} \diff N_\pfwd \PP_{\Fl^{w,t}\sigma} \fstop
\end{align}

Differentiating the l.h.s.~of~\eqref{eq:ese:PP5} under the sign of integral with respect to~$t$ yields the l.h.s.~of~\eqref{eq:DefDirForm} by~\eqref{eq:DirDer0}.
Moreover, letting~$\lambda=\delta_1$ in~\eqref{eq:esePP3.5} yields
\begin{align}\label{eq:ese:PP6}
\int u\cdot v\circ\Fl^{w,-t} \diff N_\pfwd \PP_{\Fl^{w,t}\sigma} =\int \exp\ttonde{\sigma(\car-p^\sigma_{\fl^{w,t}})} \cdot R^\sigma_{\fl^{w,t}}\circ N^{-1} \cdot u\cdot v\circ \Fl^{w,-t} \diff N_\pfwd \PP_\sigma\comm
\end{align}
where~$R^\sigma_{\fl^{w,t}}\circ N^{-1}$ is well-defined as in Lemma~\ref{l:NPPassQI}. 
Also, with obvious meaning of the notation~$x\in\mu$ for~$\mu\in N(\ddot\Gamma)$,
\begin{align*}
\diff_t\restr_{t=0}& \tonde{\exp\ttonde{\sigma(\car-p^\sigma_{\fl^{w,t}})} \cdot R^\sigma_{\fl^{w,t}}\circ N^{-1}}(\mu)=
\\
&=\diff_t\restr_{t=0} \prod_{x\in\mu} \frac{(\fl^{w,t})^*\rho(x)}{\rho(x)} J^\mssm_{\fl^{w,t}}(x)+\diff_t\restr_{t=0} \exp\quadre{\int_\MFD \tonde{\car-\frac{(\fl^{w,t})^*\rho(x)}{\rho(x)} J^\mssm_{\fl^{w,t}}(x)}\diff\sigma(x)} \fstop
\end{align*}
Now the arguments in the proof of~\cite[Thm.~3.1]{AlbKonRoe98} apply verbatim, yielding
\begin{align}\label{eq:AKR98p.461}
\diff_t\restr_{t=0} \tonde{\exp\ttonde{\sigma(\car-p^\sigma_{\fl^{w,t}})} \cdot R^\sigma_{\fl^{w,t}}\circ N^{-1}}(\mu)=-B^\sigma_w \fstop
\end{align}

\paragraph{Proof of the claim} Finally, differentiating~$v\circ \Fl^{w,-t}$ with respect to~$t$ yields~$-\grad_w v$, again by~\eqref{eq:DirDer0}. Combing this fact with~\eqref{eq:AKR98p.461}, the derivative under integral sign with respect to~$t$ of the r.h.s.~of~\eqref{eq:ese:PP6} reads~$\int u\cdot (-\grad_w v-B^\sigma_w) \diff\mbbP$, which proves the claim.
\end{proof}
\end{lem}

\begin{lem}\label{l:NPPassFS}
The measure~$\mbbP$ defined in Example~\ref{ese:NPP} satisfies Assumption~\ref{ass:P0}.

\begin{proof} By definition,~$\mbbP$ is concentrated on the set~$N(\Gamma)$, which is dense in~$\msP$ by~\cite[proof of Thm.~6.18]{Vil09}. Let~$U\neq \emp$ be open in~$\msP_2$. Then~$U\cap N(\Gamma)\neq \emp$ by density of~$N(\Gamma)$ in~$\msP_2$. By continuity of~$N$ the set~$\widetilde U\eqdef N^{-1}(U)\cap \Gamma=N^{-1}(U\cap N(\Gamma))\neq \emp$ is open in~$\Gamma$. Since~$\RP_{\sigma,\lambda}$ has full support on~$\Gamma$, cf.~\cite[Prop.~5.6]{RoeSch99}, then~$\mbbP U=\RP_{\sigma,\lambda} \widetilde U>0$.
\end{proof}
\end{lem}

\subsection{Auxiliary results on the Malliavin--Shavgulidze image measure}

\begin{lem}\label{l:MSIassQI}
The measure~$\MSI$ defined in Example~\ref{ese:MSI} satisfies Assumption~\ref{ass:P5}.

\begin{proof} Retain all notation from~\S\ref{ss:MSI}. It follows from~\eqref{eq:MSqi} that~$(L_{\tau_a})_\pfwd \MS=\MS$ for every~$a\in\mbbS^1$, hence~$\mcM_1$ is quasi-invariant with respect to the left action on~$\msG_1$ of~$\DiffSP(\mbbS^1)$, and in fact of~$\Diffeo^3_+(\mbbS^1)$, given by post-composition. That is,~\eqref{eq:MSqi} holds true with~$\MS_1$ in place of~$\MS$ for every Borel~$A\subset \msG_1$ and every~$\upphi$ in~$\DiffSP(\mbbS^1)$.

For every~$\upphi$ and~$\Phi$ as in the beginning of~\S\ref{s:Examples}, set
\begin{align*}
R_\upphi(g)\eqdef \exp\quadre{\int_{\mbbS^1} S(\upphi^{-1})(g(r))\cdot g'(r)^2 \diff\mssm(r)} \fstop
\end{align*}

By definition~\eqref{eq:Chi} of~$\chi$, it holds that
\begin{align*}
\chi(L_{\upphi^{-1}}(g))=(\upphi^{-1}\circ g)_\pfwd \mssm=\upphi^{-1}_\pfwd (g_\pfwd \mssm)=\upphi_\pfwd^{-1}\chi(g)=\Phi^{-1}(\chi(g)) \fstop
\end{align*}

As a consequence, for~$\mu=\chi(g)$,
\begin{align*}
\diff \Phi_\pfwd \MSI(\mu)=&\diff \MSI\ttonde{\Phi^{-1}(\chi(g))}=\diff \MSI\ttonde{\chi(L_{\upphi^{-1}}(g))}=\diff\MS_1(L_{\upphi^{-1}}(g))=R_\upphi(g)\cdot \diff\MS_1(g)
\\
=&(R_\upphi\circ \chi^{-1})(\mu)\cdot \diff\MSI(\mu) \fstop
\end{align*}

The conclusion straightforwardly follows from the form of the Radon--Nikod\'ym derivative~$R_\upphi$.
\end{proof}
\end{lem}

{\small

\begin{thebibliography}{10}

\bibitem{AlbKonRoe98}
{Albeverio, S.}, {Kondratiev, Yu.~G.}, and {R{\"{o}}ckner, M.}
\newblock {Analysis and Geometry on Configuration Spaces}.
\newblock {\em {J.~Funct.~Anal.}}, 154:444--500, 1998.

\bibitem{AmbGig11}
{Ambrosio, L.} and {Gigli, N.}
\newblock {A User's Guide to Optimal Transport}.
\newblock In {Ambrosio, L.}, {Bressan, A.}, {Helbing, D.}, {Klar, A.}, and
  {Zuazua, E.}, editors, {\em {Modelling and Optimisation of Flows on Networks
  -- Cetraro, Italy 2009, Editors: Benedetto Piccoli, Michel Rascle}}, volume
  {2062} of {\em {Lecture Notes in Mathematics}}, pages 1--155. {Springer},
  2013.
\newblock Throughout the present work, we refer to (the numbering of) results
  in the extended version, available at
  \url{http://cvgmt.sns.it/media/doc/paper/195/}.

\bibitem{AmbGigSav08}
{Ambrosio, L.}, {Gigli, N.}, and {Savar\'e, G.}
\newblock {\em {Gradient Flows in Metric Spaces and in the Space of Probability
  Measures}}.
\newblock {Lectures in Mathematics - ETH Z\"urich}. {Birkh\"{a}user},
  {$2^{\textrm{nd}}$} edition, 2008.

\bibitem{Ban97}
{Banyaga, A.}
\newblock {\em {The Structure of Classical Diffeomorphism Groups}}, volume 400
  of {\em {Mathematics and Its Applications}}.
\newblock {Springer}, 1997.

\bibitem{BogMay96}
{Bogachev, V.~I.} and {Mayer-Wolf, E.}
\newblock {Some Remarks on Rademacher's Theorem in Infinite Dimensions}.
\newblock {\em {Potential Anal.}}, 5:23--30, 1996.

\bibitem{BouHir91}
{Bouleau, N.} and {Hirsch, F.}
\newblock {\em {Dirichlet forms and analysis on Wiener space}}.
\newblock {De Gruyter}, 1991.

\bibitem{Cho12}
{Chodosh, O.}
\newblock {A lack of Ricci bounds for the entropic measure on Wasserstein space
  over the interval}.
\newblock {\em {J.~Funct.~Anal.}}, 262(10):4570--4581, 2012.

\bibitem{ChoGan17}
{Chow, Y.~T.} and {Gangbo, W.}
\newblock {A partial Laplacian as an infinitesimal generator on the Wasserstein
  space}.
\newblock {\em arXiv:1710.10536}, Oct.~2017.

\bibitem{Con01}
{Conlon, L.}
\newblock {\em {Differentiable Manifolds}}.
\newblock {Birkh\"{a}user Advanced Texts}. {Birkh\"{a}user},
  {$2^{\textrm{nd}}$} edition, 2001.

\bibitem{DePRin16}
{De Philippis, G.} and {Rindler, F.}
\newblock {On the Structure of $\msA$-Free Measures and Applications}.
\newblock {\em {Ann.~Math.}}, 184:1017--1039, 2016.

\bibitem{LzDS17+}
{Dello Schiavo, L.}
\newblock {The Dirichlet--Ferguson Diffusion on the Space of Probability
  Measures over a Closed Riemannian Manifold}.
\newblock {\em {arXiv:1811.11598}}, 2018.

\bibitem{LzDS17}
{Dello Schiavo, L.}
\newblock {Characteristic Functionals of Dirichlet Measures}.
\newblock {\em {arXiv:1810.09790}}, 2018.

\bibitem{LzDSLyt17}
{Dello Schiavo, L.} and {Lytvynov, E.}
\newblock {A Mecke-type Characterization of the Dirichlet--Ferguson Measure}.
\newblock {\em arXiv:1706.07602}, 2017.

\bibitem{Ebe96}
{Eberle, A.}
\newblock {Girsanov-type transformations of local Dirichlet forms: An analytic
  approach}.
\newblock {\em {Osaka J.~Math.}}, 33(2):497--531, 1996.

\bibitem{EncStr93}
{Enchev, O.} and {Stroock, D.~W.}
\newblock {Rademacher's theorem for Wiener functionals}.
\newblock {\em {Ann.~Probab.}}, 21(1):25--33, 1993.

\bibitem{Fer73}
{Ferguson, T.~S.}
\newblock {A Bayesian analysis of some nonparametric problems}.
\newblock {\em {Ann.~Statist.}}, pages {209--230}, 1973.

\bibitem{Fig10}
{Figalli, A.}
\newblock {Regularity of optimal transport maps (After Ma--Trudinger--Wang and
  Loeper)}.
\newblock {\em {Asterisque J.}}, 2010.
\newblock {S\'eminaire BOURBAKI 61\`eme ann\'ee, 2008-2009, no 1009, Juin
  2009}.

\bibitem{FraLenWin14}
{Frank, R.~L.}, {Lenz, D.}, and {Wingert, D.}
\newblock {Intrinsic metrics for non-local symmetric Dirichlet forms and
  applications to spectral theory}.
\newblock {\em {J.~Funct.~Anal.}}, 266(8):4765--4808, 2014.

\bibitem{FukOshTak11}
{Fukushima, M.}, {Oshima, Y.}, and {Takeda, M.}
\newblock {\em {Dirichlet forms and symmetric Markov processes}}, volume~19 of
  {\em {De Gruyter Studies in Mathematics}}.
\newblock {de Gruyter}, extended edition, 2011.

\bibitem{GanKimPac10}
{Gangbo, W.}, {Kim, H.~K.}, and {Pacini, T.}
\newblock {Differential Forms on Wasserstein Space and Infinite-dimensional
  Hamiltonian Systems}.
\newblock {\em {Mem.~Am.~Math.~Soc.}}, 211(995), 2010.

\bibitem{Gig11}
{Gigli, N.}
\newblock {On the inverse implication of Brenier--McCann theorems and the
  structure of $\tonde{\msP_2(M),W_2}$}.
\newblock {\em {Methods and Applications of Analysis}}, 18(2):127--158, 2011.

\bibitem{Gig12}
{Gigli, N.}
\newblock {Second Order Analysis on~$\tonde{\msP_2(M),W_2}$}.
\newblock {\em {Mem.~Am.~Math.~Soc.}}, 216(1018), 2012.

\bibitem{Gra88}
{Grabowski, J.}
\newblock {Free subgroups of diffeomorphism groups}.
\newblock {\em {Fundam.~Math.}}, {131}(2):103--121, {1988}.

\bibitem{KosZho12}
{Koskela, P.} and {Zhou, Y.}
\newblock {Geometry and analysis of Dirichlet forms}.
\newblock {\em {Adv.~Math.}}, 231:2755--2801, 2012.

\bibitem{KriMic97}
{Kriegl, A.} and {Michor, P. W.}
\newblock {\em {The Convenient Setting of Global Analysis}}, volume~53 of {\em
  {Mathematical Surveys and Monographs}}.
\newblock {American Mathematical Society}, 1997.

\bibitem{Kuz07}
{Kuzmin, P. A.}
\newblock {On circle diffeomorphisms with discontinuous derivatives and
  quasi-invariance subgroups of Malliavin--Shavgulidze measures}.
\newblock {\em {J.~Math.~Anal.~Appl.}}, 330:744--750, 2007.

\bibitem{Lot07}
{Lott, J.}
\newblock {Some Geometric Calculations on Wasserstein Space}.
\newblock {\em {Commun.~Math.~Phys.}}, 277(2):423--437, 2007.

\bibitem{MaRoe92}
{Ma, Z.-M.} and {R\"ockner, M.}
\newblock {\em Introduction to the Theory of (Non-Symmetric) Dirichlet Forms}.
\newblock {Graduate Studies in Mathematics}. Springer, 1992.

\bibitem{MalMal90}
{Malliavin, P.~M.} and {Malliavin, P.}
\newblock {An Infinitesimally Quasi Invariant Measure on the Group of
  Diffeomorphisms of the Circle}.
\newblock In {Kashiwara, M.} and {Miwa, T.}, editors, {\em {ICM-90 Satellite
  Conference Proceedings -- Special Functions}}, pages 234--244. {Springer},
  1990.

\bibitem{McS34}
{McShane, E.~J.}
\newblock {Extension of Range of Functions}.
\newblock {\em {Bull.~Amer.~Math.~Soc.}}, 40(12):837--842, 1934.

\bibitem{NekZaj88}
{Nekvinda, A.} and {Zaj\'{\i}\v{c}ek, L.}
\newblock {A simple proof of the Rademacher theorem}.
\newblock {\em {\v{C}asopis pro P\v{e}stov\'{a}n\'{\i} Matematiky}},
  113(4):337--341, 1988.

\bibitem{Ott01}
{Otto, F.}
\newblock {The Geometry of Dissipative Evolution Equations: The Porous Medium
  Equation}.
\newblock {\em {Comm.~Part.~Diff.~Eq.}}, 26(1-2):101--174, 2001.

\bibitem{OveRoeSch95}
{Overbeck, L.}, {R{\"o}ckner, M.}, and {Schmuland, B.}
\newblock {An analytic approach to Fleming--Viot processes with interactive
  selection}.
\newblock {\em {Ann.~Probab.}}, 23(1):1--36, 1995.

\bibitem{vReStu09}
{Renesse, M.-K.~von} and {Sturm, K.-T.}
\newblock {Entropic measure and Wasserstein diffusion}.
\newblock {\em {Ann.~Probab.}}, 37(3):1114--1191, 2009.

\bibitem{RoeSch99}
{R{\"o}ckner, M.} and {Schied, A.}
\newblock {Rademacher's Theorem on Configuration Spaces and Applications}.
\newblock {\em {J.~Funct.~Anal.}}, 169:325--356, 1999.

\bibitem{Set94}
{Sethuraman, J.}
\newblock {A constructive definition of Dirichlet priors}.
\newblock {\em Stat.~Sinica}, 4(2):639--650, 1994.

\bibitem{Stu11}
{Sturm, K.-T.}
\newblock {Entropic Measure on Multidimensional Spaces}.
\newblock In {Dalang, R.}, {Dozzi, M.}, and {Russo, F.}, editors, {\em {Seminar
  on Stochastic Analysis, Random Fields and Applications VI}}, volume~63 of
  {\em {Progress in Probability}}, pages 261--277. {Springer}, 2011.

\bibitem{Stu95}
{Sturm, K.-T.}
\newblock {Analysis on Local Dirichlet Spaces -- II. Upper Gaussian Estimates for the Fundamental Solutions of Parabolic Equations}.
\newblock {\em {Osaka J.~Math.}}, {32}:275--312, 1995.

\bibitem{Tre66}
{Tr{\`{e}}ves, F.}
\newblock {\em {Linear Partial Differential Equations with Constant
  Coefficients -- Existence, Approximation and Regularity of Solutions}},
  volume~6 of {\em {Mathematics and its Applications}}.
\newblock {Gordon and Breach -- Science Publishers}, {London}, 1966.

\bibitem{Tre67}
{Tr{\`{e}}ves, F.}
\newblock {\em {Topological Vector Spaces, Distributions and Kernels}},
  volume~25 of {\em {Pure and Applied Mathematics}}.
\newblock {Academic Press}, {New York}, 1967.

\bibitem{TsiVerYor01}
{Tsilevich, N.}, {Vershik, A.~M.}, and {Yor, M.}
\newblock {An Infinite-Dimensional Analogue of the Lebesgue Measure and
  Distinguished Properties of the Gamma Process}.
\newblock {\em {J.~Funct.~Anal.}}, {185}:274--296, 2001.

\bibitem{Vil03}
{Villani, C.}
\newblock {\em {Topics in Optimal Transportation}}, volume~58 of {\em {Graduate
  Studies in Mathematics}}.
\newblock {American Mathematical Society}, {Providence, RI}, 2003.

\bibitem{Vil09}
{Villani, C.}
\newblock {\em {Optimal transport, old and new}}, volume 338 of {\em
  {Grundlehren der mathematischen Wissenschaften}}.
\newblock {Springer-Verlag}, 2009.

\end{thebibliography}

}

\end{document}